\documentclass[reqno]{amsart}
\usepackage{ifdraft}
\usepackage{latexsym,amsfonts,amssymb,exscale,enumerate,comment}
\usepackage{amsmath,amsthm,amscd}
\usepackage{mathrsfs}
\usepackage{comment}
\usepackage{xcolor}

\definecolor{answercolor}{RGB}{0, 112, 48}
\excludecomment{answer}      

\usepackage{url}
\usepackage[bookmarks=true,%
    colorlinks=true,%
    linkcolor=blue,%
    citecolor=blue,%
    filecolor=blue,%
    menucolor=blue,%
    urlcolor=blue,%
    breaklinks=true]{hyperref}

\input xy
\usepackage[all]{xy}
\xyoption{line}
\xyoption{arrow}
\xyoption{color}
\SelectTips{cm}{}

\usepackage{tikz}
\usetikzlibrary{shapes,decorations}
\usetikzlibrary{decorations.markings}
\usetikzlibrary{decorations.pathreplacing}
\usetikzlibrary{cd}
\tikzstyle directed=[postaction={decorate,decoration={markings,
    mark=at position #1 with {\arrow{>}}}}]

\newcommand{\hackcenter}[1]{
 \xy (0,0)*{#1}; \endxy}

\tikzset{->-/.style={decoration={
  markings,
  mark=at position #1 with {\arrow{>}}},postaction={decorate}}}

\tikzset{middlearrow/.style={
        decoration={markings,
            mark= at position 0.5 with {\arrow{#1}} ,
        },
        postaction={decorate}
    }
}

\usepackage{graphicx}

\usetikzlibrary{shapes.geometric}

\def\bba{\text{\boldmath$a$}}
\def\bbep{\text{$\varepsilon$}}

\def\bbb{\text{\boldmath$b$}}

\usepackage{hyperref}

\usepackage{enumitem}

	\definecolor{darkolivegreen}{rgb}{0.33, 0.42, 0.18}

\usepackage[capitalize]{cleveref}
\crefname{theorem}{Theorem}{Theorems}
\crefname{fact}{Fact}{Facts}
\crefname{note}{Note}{Notes}
\crefname{lemma}{Lemma}{Lemmas}
\crefname{alg}{Algorithm}{Algorithms}
\crefname{remark}{Remark}{Remarks}
\crefname{example}{Example}{Examples}
\crefname{prop}{Proposition}{Propositions}
\crefname{conj}{Conjecture}{Conjectures}
\crefname{cor}{Corollary}{Corollaries}
\crefname{definition}{Definition}{Definitions}
\crefname{Relation}{Relation}{Relations}
\crefname{equation}{\!\!}{\!\!} 

\DeclareFontFamily{OT1}{pzc}{}
\DeclareFontShape{OT1}{pzc}{m}{it}{<-> s * [1.10] pzcmi7t}{}
\DeclareMathAlphabet{\mathpzc}{OT1}{pzc}{m}{it}

\theoremstyle{plain}
\newtheorem{theorem}{Theorem}[subsection]

\newtheorem{proposition}[theorem]{Proposition}
\newtheorem{lemma}[theorem]{Lemma}

\theoremstyle{definition}
\newtheorem{example}[theorem]{Example}
\newtheorem{definition}[theorem]{Definition}

\theoremstyle{definition}
\newtheorem{remark}[theorem]{Remark}

\newcommand{\refequal}[1]{\xy {\ar@{=}^{#1}
(-1,0)*{};(1,0)*{}};
\endxy}

\newcommand{\Hom}{{\rm Hom}}

\renewcommand{\to}{\rightarrow}

\def\pWeb{\mathfrak{p}\textup{-}\mathbf{Web}}

\numberwithin{equation}{section}
\let\hat=\widehat
\let\tilde=\widetilde

\let\tupepsilon=\tupepsilon

\usepackage{bbm}

\def\Z{{\mathbbm Z}}

\def\1{\bar{1}}%

\newcommand{\lrangle}[1]{\langle #1 \rangle}

\def \k {\mathbbm{k}}
\def \Z {\mathbbm{Z}}

\def\bU{\mathbf{\dot U}}

\def \min {{\rm min}}

\newcommand\nc{\newcommand}
\nc\rnc{\renewcommand}
\nc\Kar{\operatorname{Kar}}
\nc\modQ {{\mathbb Q}}
\nc\modZ {{\mathbb Z}}
\nc\simeqto{\overset{\simeq}{\longrightarrow }}
\nc\K{\mathcal {K}}
\nc\CC{\mathbf{C}}

\nc\qh{\mathcal{H}}
\nc\hbm{\mathcal{B}}
\nc\bu{\mathbf{u}}
\nc\bk{\mathbf{\kappa}}
\nc\bZ{\mathbf{Z}}
\nc\theirs{\mathrm{theirs}}
\nc\ours{\mathrm{ours}}
\nc\hE{\mathcal{\hat E}}
\nc\bK{\mathbf{K}}
\nc\bw{\mathbf{w}}
\nc\bh{\mathbf{h}}
\nc\cE{\textcolor{red}{E} }
\nc\cR{\textcolor{darkolivegreen}{R}}
\nc\ba{\mathbf{a}}
\nc\bb{\mathbf{b}}
\nc\cA{\textcolor{red}{A}}
\nc\cB{\textcolor{red}{B}}
\nc\cC{\textcolor{red}{C}}
\nc\te{{e}}
\nc\ta{\tilde{a}}
\nc\tb{\tilde{b}}
\nc\tc{\tilde{c}}
\nc\tf{{f}}
\nc\tcE{\textcolor{violet}{\tilde{E}}}
\nc\tcF{\textcolor{violet}{\tilde{F}}}
\nc\tcH{\textcolor{violet}{\tilde{H}}}
\nc\tcB{\textcolor{violet}{\tilde{B}}}
\nc\tcC{\textcolor{violet}{\tilde{C}}}
\nc\tcX{\textcolor{violet}{\tilde{X}}}
\nc\tcY{\textcolor{violet}{\tilde{Y}}}
\nc\talpha{\tilde{\tupalpha}}
\nc\tbeta{\tilde{\tupbeta}}
\nc\tgamma{\tilde{\tupgamma}}

\nc\cH{\textcolor{red}{H}}
\nc\cF{\textcolor{red}{F}}
\nc\cX{\textcolor{red}{X}}
\nc\cY{\textcolor{red}{Y}}

\nc{\bsym}{\boldsymbol}
\nc{\tuplambda}{{\bsym\lambda}}
\nc{\tupnu}{{\bsym\nu}}
\nc{\tupgamma}{{\bsym{\gamma}}}
\nc{\tupbeta}{{\bsym {\beta}}}
\nc{\tupalpha}{{\bsym {\alpha}}}
\nc{\tupmu}{{\bsym {\mu}}}
\nc{\tupepsilon}{\bsym \varepsilon}
\nc{\tupdelta}{\bsym \delta}

\nc{\urcap}{\textrm{uRCap}}
\nc{\urcup}{\textrm{uRCup}}
\nc{\ulcap}{\textrm{uLCap}}
\nc{\ulcup}{\textrm{uLCup}}
\nc{\drcap}{\textrm{dRCap}}
\nc{\drcup}{\textrm{dRCup}}
\nc{\dlcap}{\textrm{dRCap}}
\nc{\dlcup}{\textrm{dRCup}}

\nc{\cl}{\mathrm{cl}}
\nc{\cala}{\mathcal{A}}
\nc{\calb}{\mathcal{B}}
\nc{\calc}{\mathcal{C}}
\nc{\bx}{\mathbf{x}}
\nc{\by}{\mathbf{y}}
\nc{\bE}{\mathbbm{E}}
\nc{\bElong}{{\mathbbm E^{\bullet, \geq}}}

\nc{\Sk}{\mathrm{Sk}}
\nc{\Hilb}{\mathrm{Hilb}}

\nc\col{\colon\thinspace}

\allowdisplaybreaks

\newcommand{\clr}{rgb:black,1;blue,4;red,1}

\newcommand{\p}[1]{|#1|}
\newcommand{\0}{\bar{0}}
\renewcommand{\1}{\bar{1}}

\newcommand{\ZZ}{\ensuremath{\mathbb{Z}}}
\newcommand{\fp}{\ensuremath{\mathfrak{p}}}

\newcommand{\fg}{\mathfrak{g}}

\newcommand{\fh}{\mathfrak{h}}

\newcommand{\fb}{\mathfrak{b}}

\newcommand{\pmodS}{\ensuremath{\fp(n)\textup{-mod}_{\mathcal{S}}}}
\newcommand{\pmodSm}{\ensuremath{\fp(n)\textup{-mod}}_{\mathcal{S}, m}}
\newcommand{\pmodSS}{\ensuremath{\fp(n)\textup{-mod}_{\mathcal{S, S^{*}}}}}
\newcommand{\pmodSSm}{\ensuremath{\fp(n)\textup{-mod}_{\mathcal{S, S^{*}},m}}}
\newcommand\pmodSD{\ensuremath{\fp(n) \textup{-mod}_{\mathcal{S^*}}}}
\newcommand\pmodSDm{\ensuremath{\fp(n) \textup{-mod}_{\mathcal{S^*},m}}}

\newcommand{\tpmmodSn}{\ensuremath{\tfp(m)\textup{-mod}}_{\mathcal{S}, n}}

\newcommand{\Id}{1}
\newcommand{\mer}{\ensuremath{\operatorname{mer}}}
\newcommand{\spl}{\ensuremath{\operatorname{spl}}}

\newcommand{\End}{\operatorname{End}}

\newcommand{\gl}{\mathfrak{gl}}
\newcommand{\tfp}{\ensuremath{\tilde{\mathfrak{p}}}}

\newcommand{\bUdot}{\dot{\mathbf{U}}}

\newcommand{\ttA}{{\mathtt{A}}}
\newcommand{\ttB}{{\mathtt{B}}}
\newcommand{\ttC}{{\mathtt{C}}}
\newcommand{\ttE}{{\mathtt{E}}}

\newcommand{\ttAprime}{{\mathtt{A}}'}
\newcommand{\ttBprime}{{\mathtt{B}}'}
\newcommand{\ttCprime}{{\mathtt{C}}'}
\newcommand{\ttEprime}{{\mathtt{E}}'}

\newcommand{\tildettA}{\tilde{\mathtt{A}}}
\newcommand{\tildettB}{\tilde{\mathtt{B}}}
\newcommand{\tildettC}{\tilde{\mathtt{C}}}

\newcommand{\fin}{{\mathtt{fin}}}

\newcommand{\tildettAprime}{\tilde{\mathtt{A}}'}
\newcommand{\tildettBprime}{\tilde{\mathtt{B}}'}
\newcommand{\tildettCprime}{\tilde{\mathtt{C}}'}

\newcommand{\zprime}{z'}
\newcommand{\partialprime}{\partial'}

\theoremstyle{theorem}
\newtheorem*{IntroTheorem}{Theorem A}
\newtheorem*{IntroTheoremB}{Theorem D}
\newtheorem*{IntroTheoremC}{Theorem B}
\newtheorem*{IntroTheoremD}{Theorem C}

\def\Fun{{\mathcal{F}\hspace{-0.02in}\mathpzc{un}}}

\def\sVec{{\mathpzc{s}\mathcal{V}\hspace{-0.02in}\mathpzc{ec}}}

\title{
Howe Duality of Type P}

\begin{document}
\setcounter{tocdepth}{2}

\author{Nicholas Davidson}
\email{njd@reed.edu}
\address{Department of Mathematics\\ Reed College \\ Portland, OR}

\author{Jonathan R. Kujawa}
\email{kujawa@ou.edu}
\address{Department of Mathematics\\ University of Oklahoma \\ Norman, OK}

\author{Robert Muth}
\email{rmuth@washjeff.edu}
\address{Department of Mathematics \\ Washington \& Jefferson College \\ Washington, PA}

\date{\today}

\thanks{The second author was supported in part by Simons Collaboration Grant for Mathematicians No.\ 525043.}

\begin{abstract}
We establish classical and categorical Howe dualities between the Lie superalgebras \(\mathfrak{p}(m)\) and \(\mathfrak{p}(n)\), for \(m,n \geq 1\). We also describe a presentation via generators and relations as well as a Kostant \(\mathbb{Z}\)-form for the universal enveloping superalgebra \(U(\mathfrak{p}(m))\).
\end{abstract}

\maketitle

\section{Introduction}

\subsection{Overview}\label{SS:overview}   Let $\k$ be an algebraically closed field of characteristic zero and let $\fp (n)$ denote the type $P$ Lie superalgebra defined over $\k$.  In the Kac classification of the simple Lie superalgebras it is one of the so-called strange families which have no direct analogue among ordinary Lie algebras.  While there have been considerable advances in the understanding of the representation theory of $\fp (n)$ (e.g.,  see \cite{WINART1,WINART2,BK,CCP,CP,Coulembier1,CEIII,CEII} and references therein), many aspects of its representation theory remain mysterious.

 When two Lie (super)algebras have mutually centralizing actions on a symmetric (super)algebra, then they are a \emph{Howe dual pair} and their mutually centralizing action on the superalgebra is a \emph{Howe duality}.  A Howe duality provides a bridge between the representation theories of the two Lie (super)algebras as well as applications to invariant theory and other fields.  See \cite{Howe2} and \cite{CWBook} for an overview of classical and super Howe dualities, respectively.  This paper establishes a Howe duality between the Lie superalgebras $\fp (m)$ and $\fp (n)$ for all $m, n \geq 1$.

\subsection{Main Results}\label{SS:mainresults}  Let $U(\fp (m))$ and $U(\fp (n))$ denote the enveloping superalgebras of $\fp (m)$ and $\fp (n)$, respectively.
Let $V_{n}$ denote the natural supermodule for $\fp (n)$ and let $Y_{m}$ denote a purely even superspace of dimension $m$ with a trivial $\fp (n)$-action.  Then $Y_m \otimes V_n$ is a $\fp(n)$-supermodule via the coproduct of $U(\fp (n))$ and there is an induced action of $\fp (n)$ on the symmetric superalgebra $S(Y_{m} \otimes V_{n})$.  By \cref{L:fpembedding} the action of $\fp (n)$ on $S(Y_{m} \otimes V_{n})$ can be realized by polynomial differential operators.  Moreover, by \cref{L:tfpembedding} there is a less obvious action of $\fp (m)$ on $S(Y_{m} \otimes V_{n})$ which is also given by polynomial differential operators.  The following theorem summarizes  \cref{T:IntertwiningActionsClassical,T:ClassicalHoweDuality1,T:ClassicalHoweDuality2} which collectively establish a Howe duality for $\fp (m)$ and $\fp (n)$. 

\begin{IntroTheorem}\label{T:IntroTheorem} Let $\mathcal{WC}(Y_{m} \otimes V_{n})$ denote the Weyl-Clifford superalgebra of polynomial differential operators on $S(Y_{m} \otimes V_{n})$.  Then there are superalgebra homomorphisms
\[\begin{tikzcd}[row sep = small, column sep = small]
U(\fp(m)) \arrow{dr} &  & U(\fp(n)) \arrow{dl} \\
 & \mathcal{WC}\left( Y_{m} \otimes V_{n}\right)  & 
\end{tikzcd}\]
with mutually centralizing images.
\end{IntroTheorem}  Note that here and below ``mutually centralizing'' is used to mean that each map surjects onto the centralizer (in the graded sense) of the other map's image.

The above theorem has a second incarnation for an idempotent version of the enveloping algebra.  Given a superspace $S$ let $\End^{\fin}_{\k}(S)$ be the subsuperalgebra of all linear maps with finite-dimensional image. Let $\dot{U}(\fp (n))$ be the idempotent form of $U(\fp (n))$  as defined in  \cref{D:Udot}.  We then have the following finitary version of Howe duality.

\begin{IntroTheoremC}\label{T:IntroTheoremC}  For every $m, n \geq 1$ there are superalgebra homomorphisms
$$
\begin{tikzcd}[row sep = small, column sep = small]
\dot{U}(\fp(m)) \arrow{dr}{} &  & \dot{U}(\fp(n)) \arrow{dl}[swap]{} \\
 & \End_{\k}^{\fin}\left(S( Y_{m} \otimes V_{n}) \right) & 
\end{tikzcd}
$$
which have mutually centralizing images.
\end{IntroTheoremC} 

There is a third, categorical, version of this Howe duality.  Namely, let $\bU (\fp (m))$ be the supercategory given in \cref{D:Udot}.  Recall that a representation of a supercategory $\mathcal{C}$ is an even $\k$-linear functor from $\mathcal{C}$ to the category of superspaces, $\sVec$.  The supercategory $\bU (\fp (m)) \boxtimes \bU (\fp (n))$ and the notion of a functor having the double centralizer property (the categorical analogue of two superalgebras having mutually centralizing images) are explained in \cref{SS:CategoricalDoubleCentralizer}.

\begin{IntroTheoremD}\label{T:IntroTheoremD}  There is a representation  
\[
F: \bU (\fp (m)) \boxtimes \bU (\fp (n)) \to \sVec
\]
defined by $S( Y_{m} \otimes V_{n})$ which has the double centralizer property.
\end{IntroTheoremD}

It is worth noting that the verification of most Howe dualities in the literature depends upon the semisimplicity of one or both actions.  In this case neither action is semisimple.  Instead the proof relies on the following theorem. Let $\pmodSm$ denote the full subcategory of all $\fp (n)$-supermodules which are obtained by $m$-fold tensor products of symmetric powers of the natural module for $\fp (n)$. 

\begin{IntroTheoremB}\label{T:IntroTheoremB}   For every $m, n \geq 1$, there is an essentially surjective, full functor
\begin{gather*}
 \bU(\fp(m)) \to \pmodSm. 
\end{gather*}
\end{IntroTheoremB}
The key to proving this result (which appears as \cref{T:CategoricalHoweDualityI,T:CategoricalHoweDualityII}) is an explicit diagrammatic description of $\pmodSm$ established in \cite{DKM}.    See \cite{QS,ST} for other instances where a Howe duality is established using diagrammatic methods.

There is a certain asymmetry to Theorems~A,~B~and~C: while the action of $\fp (n)$ on $S(Y_m \otimes V_n)$ is easy to define via the natural representation of $\fp(n)$, the definition of the action of $\fp(m)$ is rather more subtle.  We give explicit formulas for this action in \cref{SS:UpnActiononA}.

Of course there is an obvious analogue where the roles of $\fp(m)$ and $\fp(n)$ are reversed.  The superalgebra $S(V_m \otimes Y_n)$ has a natural action by $\fp(m)$ along with a less obvious action by $\fp (n)$ (see \cref{SS:CategoricalHoweDualityII}).   In \cref{T:IntertwiningActionsCategorical,T:IntertwiningActionsClassical} we define intertwining maps which relate the two constructions.  Not only do the intertwiners balance the asymmetry, they are also key to the double centralizer portions of Theorems~A,~B,~and~C.

Another crucial ingredient is the presentation of $U(\fp (m))$ by generators and relations given in \cref{D:Tdef}.  That the superalgebra defined in this way is indeed isomorphic to $U(\fp (m))$ is established in \cref{T:OrderedKostantMonomialsBasisandIsomorphism}.  We also obtain a Kostant $\Z$-form for $U(\fp (m))$.  We believe this presentation and $\Z$-form are of independent interest.

\subsection{Further Questions}\label{SS:FurtherQuestions}   Since the actions of $U(\fp (m))$ and $U(\fp (n))$ on $S(Y_{m} \otimes V_{n})$ commute we can view $S(Y_{m} \otimes V_{n})$ as a $U(\fp (m)) \otimes U(\fp (n))$-supermodule.  Unlike many other examples of Howe duality, it is not semisimple and it would be interesting to describe its indecomposable summands.  In a different direction, Ahmed, Grantcharov, and Guay defined a quantized enveloping superalgebra of type $P$ using the FRT formalism \cite{AGG}. As those authors point out, it is an open question to provide a Drinfeld-Jimbo-type presentation of this quantized enveloping superalgebra which, presumably, could be chosen to be a quantization of  \cref{D:Tdef}.  It would also be interesting to provide a quantized version of the diagrammatics of \cite{DKM} and of the Howe duality presented here.   Finally, the Kostant $\Z$-form for $U(\fp (m))$ given in  \cref{T:OrderedKostantMonomialsBasisandIsomorphism} allows for base change to positive characteristic.   This opens the door to a variety of interesting questions about modular representation theory in type $P$.

\subsection{Acknowledgements} We thank Inna Entova-Aizenbud for her suggestion to also include the Lie superalgebra $\tfp(m)$ and, indeed, it proved convenient for several constructions.  Early investigation in this project began while the authors were participating in MSRI's Spring 2018 program titled ``Group representation theory and its applications''. We thank MSRI for its hospitality.  Initial calculations and final revisions by the second author were done while he enjoyed the inspiring environment of the Pacheco Canyon.

\subsection{ArXiv Version}\label{SS:ArxivVersion}  In order to keep the paper from becoming overlong we chose to relegate several of the more lengthy and routine calculations to the \texttt{arXiv} version of the paper.  The reader interested in seeing these additional details can download the \LaTeX~source file from the \texttt{arXiv}.  Near the beginning of the file is a toggle which allows one to compile the paper with these calculations included.

\section{The Lie Superalgebra of Type P}
\subsection{Preliminaries}\label{SS:prelims} Throughout $\k$ will be an algebraically closed field of characteristic zero. The general setting of this paper is that of mathematical objects which are $\Z_{2}=\Z/2\Z$-graded (or ``super'').  To establish nomenclature, we say an element has \emph{parity} $r$ if it is homogenous and of degree $r \in \Z_{2}$.  We view $\k$ as an associative superalgebra concentrated in parity $\0$.  For short, we sometimes call an element \emph{even} (resp.\ \emph{odd}) if it has parity $\0$ (resp.\ $\1$).  Given a $\Z_{2}$-graded $\k$-vector space, $W = W_{\0}\oplus W_{\1}$, and a homogeneous element $w \in W$, we write $\p{w} \in \Z_{2}$ for the parity of $w$.

We assume the reader is familiar with the language of superspaces, associative superalgebras, Lie superalgebras, supermodules, etc.  Unless otherwise stated, all vector spaces will have a $\Z_{2}$-grading. We should emphasize that we allow for all morphisms, not just parity preserving ones.  The morphism spaces have a $\Z_{2}$-grading given by declaring $f: V \to W$ to have parity $r \in \Z_{2}$ if $f(V_{s}) \subseteq W_{s+r}$ for $s\in \Z_{2}$.  In particular, all categories will be assumed to be $\k$-linear supercategories; that is, categories enriched over the supercategory of $\k$-superspaces,  $\sVec$.  See \cite{BE} for details on supercategories and related topics.

For brevity we frequently drop the prefix ``super''.

\subsection{Definition of \texorpdfstring{$\fp (n)$}{p(n)}} \label{SS:DefofP}
  Let $I = I_{n|n}$ denote the ordered set $I = \{1,\dotsc, n, -1, \dotsc, -n\}$.   Fix a parity function $|\cdot | : I \to \Z_2$  by declaring $|i| = \0$ if $i > 0$, and $|i| = \1$ if $i < 0$.   

Let $V_{n}=V_{\0}\oplus V_{\1 }$ be a $\k$-superspace with fixed  homogeneous basis $\{v_i \mid i \in I\}$ with parity given by $\p{v_i} = \p{i}$.    Let $\gl (V_{n})$ denote the $\k$-vector space of all $\k$-linear endomorphisms of $V_{n}$, which is equipped with a $\Z_2$-grading given in the above discussion.   Then $\gl (V_n)$ is the general linear Lie superalgebra with Lie bracket given by the supercommutator:
\[
[f,g] = f \circ g - (-1)^{\p{f}\p{g}}g \circ f,
\] for all homogeneous $f,g \in \gl(V_{n})$.  With respect to our fixed homogenous basis of $V_{n}$ we may identify $\gl (V_{n})$ with
\begin{equation*} 
\gl(n|n) = \left\{ \left.
\begin{bmatrix}
A & B \\ C & D
\end{bmatrix} \: \right| \text{  $A$, $B$, $C$, and $D$ are $n \times n$ matrices with entries from $\k$}
\right\}.
\end{equation*} Regarding the $\Z_{2}$-grading,  $\gl (n|n)_{\0}$ (resp., $\gl (n|n)_{\1}$) is the subspace of all such matrices with $B$ and $C$ (resp., $A$ and $D$) equal to zero.

The superspace $V_{n}$ may be equipped with an odd, non-degenerate, supersymmetric bilinear form given on the basis by
\[
(v_i, v_j) = \delta_{i, -j}.
\]  By definition, Lie superalgebra $\fp\left(V_{n} \right)$ is the Lie subsuperalgebra of linear maps whose homogeneous elements preserve this bilinear form:
 \begin{equation} \label{E:DefofP}
\fp\left(V_{n} \right)_{r} = \left\{ f \in \gl(V_n)_{r} \mid (f(x), y) + (-1)^{\p{x} \p{f}} (x, f(y)) = 0 \text{ for all homogeneous $x, y \in V_n$} \right\},
\end{equation} for $r \in \Z_{2}$.

With respect to the fixed basis for $V_n$ the condition in \cref{E:DefofP} identifies $\fp(V_n)$  with the subalgebra $\fp(n) \subset \gl(n|n)$ given by
\begin{equation} \label{E:FormofP}
\fp(n) = \left\{
\begin{bmatrix}
A & B \\ C & -A^T
\end{bmatrix} 
\right\},
\end{equation}
where $A, B,$ and $C$ are $n \times n$ matrices, $B$ is symmetric, $C$ is skew-symmetric, and $A^T$ is the usual transpose of $A$.

\subsection{Integral Grading}

Note that $\fg = \fp (n)$ admits a $\Z$-grading 
\begin{equation} \label{E:ZGrading}
\fg = \fg_{-1} \oplus \fg_0 \oplus \fg_1,
\end{equation} where $\fg_{1}$ is the subspace of all matrices of the form \cref{E:FormofP} with $A=C=0$, the subspace $\fg_{0}$  consists of all matrices with $B = C = 0$ , and $\fg_{-1}$ is the subspace of all matrices with $A=B=0$.   Therefore, $\fg_{0} = \fg_{\0}$ is the even subspace of $\fg$, and $\fg_{\1} = \fg_{-1} \oplus \fg_1$ is the odd subspace.  Then $\fg_{\0}$ is a Lie algebra which is isomorphic to $\gl (n)$ and the decomposition in \cref{E:ZGrading} is as $\fg_{\0}$-modules under the adjoint action of $\fg_0$ on $\fg$.  Furthermore,
$\fg_{1} \cong  S^{2}\left(V_{0} \right)$ and $\fg_{-1} \cong \Lambda^{2}\left(V_{0} \right)^{*}$, 
where $V_{0}$ denotes the natural $n$-dimensional module for $\gl(n)$.

\subsection{Root Combinatorics}\label{SS:RootCombinatorics}

Let $\fh$ denote the Cartan subalgebra of $\fp (n)$ consisting of diagonal matrices in \cref{E:FormofP}.  Let $\fb$ denote the Borel subalgebra of $\fp (n)$ consisting of the matrices in \cref{E:FormofP} where $C = 0$ and $A$ is upper-triangular. Then $\fh_{\0}=\fh $ and $\fb_{\0}$ define Cartan and Borel subalgebras for $\fg_{\0}=\gl (n)$.

For $1 \leq i \leq n$ let $\tupepsilon_{i} \in \fh^{*}$ denote the linear functional which picks off the $(i,i)$-diagonal entry.  By definition the weight lattice for $\fp (n)$ is 
\[
X(T)=X(T_{n}) = \bigoplus_{i=1}^{n} \mathbb{Z}\tupepsilon_{i}.
\] 
Define a symmetric, nondegenerate bilinear form on $X(T)$ by $(\tupepsilon_{i}, \tupepsilon_{j})=\delta_{i,j}$ for all $1 \leq i,j \leq n$.  We set $X(T_n)_{\geq 0} = \bigoplus_{i=1}^{n} \mathbb{Z}_{\geq 0} \tupepsilon_{i}$.

With respect to the adjoint action of $\fh$ on $\fg$ there is a decomposition into root spaces which we now describe.  The even roots are the elements of $\Phi_0 = \Phi_{\0} = \left\{\tupalpha_{i,j} = \tupepsilon_i - \tupepsilon_j \mid 1 \leq i \neq j \leq n \right\} $, the positive even roots are $\Phi_{0}^+ = \{\tupalpha_{i,j} \mid 1 \leq i < j \leq n\}$, and the negative even roots are $\Phi_0^- = -\Phi_{0}^+$.  Set $\Phi_{-1} =  \left\{\tupbeta_{i,j} =  \tupepsilon_i + \tupepsilon_j \mid 1 \leq i \leq j \leq n  \right\} $, and $\Phi_1=  \left\{ \tupgamma_{i,j} = -\left(\tupepsilon_i + \tupepsilon_j \right) \mid 1 \leq i < j \leq n \right\}$.  The odd roots are $\Phi_{\1}  = \Phi_{-1} \sqcup \Phi_1$ and the set of all roots is $\Phi = \Phi_{\0} \sqcup  \Phi_{\1}$.  Finally, for short write  $\tupalpha_{i} = \tupalpha_{i,i+1}$, $\tupbeta_{i} = \tupbeta_{i,i+1}$, and $\tupgamma_{i} = \tupgamma_{i,i+1}$ for  $1 \leq i \leq n-1$.

For $i,j \in I_{n|n}$, we write $E_{i,j} \in \gl (n|n)$ for the matrix unit with a one in the $(i,j)$-position and zero elsewhere.  Using these define the following root vectors of $\fp (n)$:
\begin{center}
\begin{table}[h]
\begin{tabular}{|c|c|c|}
\hline
Root Vector          &  Root         & Allowed Indices                  \\ \hline \hline
$a_{i,j} := E_{i,j} - E_{-j,-i}$                                                                                                                                               & $\tupalpha_{i,j} = \tupepsilon_i - \tupepsilon_j$     & $1 \leq i \neq j \leq n$ \\ \hline
$b_{i,j} := 2^{-\delta_{i,j}} \left( E_{i,-j} +E_{j, -i} \right)$ & $\tupbeta_{i,j} =  \tupepsilon_i + \tupepsilon_j $ & $1 \leq i \leq j \leq n$ \\ \hline
$c_{i,j} := E_{-i, j} - E_{-j, i}   $                                                                                                                  & $\tupgamma_{i,j} = -(\tupepsilon_i + \tupepsilon_j)$     & $1 \leq i < j \leq n$    \\ \hline
\end{tabular}  .
\end{table}
\end{center} 
For $1 \leq i \leq n$ write $a_{i,i} = E_{i,i} - E_{-i,-i}$ and note these form a basis of Cartan subalgebra $\fh$. These along with the vectors in the table provide a basis of vectors for $\fp (n)$ with the vectors $a_{i,j}$ forming a basis for $\fg_{\0}$, the vectors $b_{i,j}$ forming a basis for $\fg_{1}$, and the vectors $c_{i,j}$ forming a basis for $\fg_{-1}$.

For the second and third rows of the table it is convenient to allow for $i \geq j$.  Consequently, for $i>j$ we adopt the convention that
$$
b_{i,j} = b_{j,i}, \quad \tupbeta_{i,j} = \tupbeta_{j,i}, \quad c_{i,j} = -c_{j,i},  \quad \text{and} \quad  \tupgamma_{i,j} = \tupgamma_{j,i}.
$$
For $i=j$ set $c_{i,i} = 0$ and $\tupgamma_{i,i} = 0$.

\subsection{Commutator Formulas} 
The following commutators can be calculated directly from the definitions.  We we record them here for later use.
\begin{lemma} \label{L:Commutators}The distinguished basis vectors for $\fp(n)$ satisfy the following commutator relations: 
\begin{align*}
[a_{i,j}, a_{p,q}] &= \delta_{p,j}a_{i,q}-\delta_{q,i}a_{p,j}; \\
[b_{i,j}, b_{p,q}] &= 0; \\
[c_{i,j}, c_{p,q}] &= 0; \\
 [b_{i,j}, c_{p,q}] &= \begin{cases}
                  - \delta_{j,q}a_{i,p}+\delta_{j,p}a_{i,q}-\delta_{i,q}a_{j,p}+\delta_{i,p}a_{j,q}, & i < j; \\
                   - \delta_{i,q}a_{i,p} + \delta_{i,p}a_{i,q}, & i=j;
\end{cases}  \\
[ c_{p,q}, a_{i,j}] &= \delta_{i,p}c_{j,q} - \delta_{i,q}c_{j,p}; \\  
[a_{i,j}, b_{p,q}] &=  \begin{cases}  \delta_{j,p} 2^{\delta_{i,q}} b_{i,q} + \delta_{j,q} 2^{\delta_{i,p}} b_{i,p}, & \text{if $p \neq q$};\\
                                    \delta_{j,p} 2^{\delta_{i,p}} b_{i,p}, & \text{if $p=q$}.
\end{cases} 
\end{align*}
\end{lemma}

\section{Presentation of the Enveloping Superalgebra \texorpdfstring{$U(\fp(n))$}{Utildepn}}\label{S:presentationofUp}

\subsection{The Superalgebra  \texorpdfstring{$\cR$}{R}}\label{SS:Tdef}  We next introduce an associative superalgebra $\cR=\cR_{n}$ by generators and relations using the root combinatorics established in \cref{SS:RootCombinatorics}.  The main goal of this section will be to prove that $\cR \cong U(\fp (n))$ and thereby obtain a presentation for the enveloping superalgebra of $\fp (n)$.

In what follows,  whenever $x, y$ are homogeneous elements of an associative superalgebra, we define 
\[
[x,y] = xy - (-1)^{\p{x} \p{y}}yx.
\]
The bracket is skew supersymmetric and satisfies the graded analogue of the Jacobi identity:
\[
[x,y] = -(-1)^{\p{x} \p{y}} [y,x] \quad \text{and} \quad
[x,[y,x]] = [[x,y], z] + (-1)^{\p{x} \p{y}} [ y, [x,z]],
\]
where these identities hold for all homogeneous $x,y,z$.

\begin{definition}\label{D:Tdef}  Let $\cR=\cR_{n}$ be the associative $\k$-superalgebra with even generators $\cE_{1}, \dotsc , \cE_{n-1}$, $\cF_{1}, \dotsc , \cF_{n-1}$, and $\cH_{1}, \dotsc , \cH_{n}$, and odd generators $\cB_{1,1}, \cB_{1}, \dotsc , \cB_{n-1}$ and $\cC_{1}, \dotsc ,\cC_{n-1}$.  

The relations are as follows.  Unless otherwise indicated, the relations hold for all allowable values of the subscripts. 

\begin{enumerate}[label=({R}{\arabic*})]
\item  \label[Relation]{LL:hcommuterelation}
\begin{equation}
\label{E:hhcommutator} [\cH_{i}, \cH_{j}]=0;
\end{equation}
\item  \label[Relation]{LL:RootVectors}
\begin{align}
\label{E:hecommutator} [\cH_{i}, \cE_{j}] &= (\tupepsilon_{i}, \tupalpha_{j})\cE_{j},\\
\label{E:hfcommutator} [\cH_{i}, \cF_{j}] &= -(\tupepsilon_{i}, \tupalpha_{j})\cF_{j};
\end{align}
\item  \label[Relation]{LL:efrelation}
\begin{equation}
\label{E:efcommutator} [\cE_{i}, \cF_{j}] = \delta_{i,j}(\cH_{i}-\cH_{i+1});
\end{equation}
\item  \label[Relation]{LL:eSerrerelations}
\begin{align}
\label{E:eecommutator} [\cE_{i},\cE_{j}] &= 0 \text{ if $j \neq i \pm 1$}, \\
\label{E:eSerreRelation} [\cE_{i},[\cE_{i}, \cE_{j}]] &=0 \text{ if  $j = i \pm 1 $};
\end{align}
\item  \label[Relation]{LL:fSerrerelations}
\begin{align}
\label{E:ffcommutator}  [\cF_{i},\cF_{j}] &= 0 \text{ if $j \neq i \pm 1$}, \\
\label{E:fSerreRelation} [\cF_{i},[\cF_{i},\cF_{j}]] &=0 \text{ if  $j = i \pm 1 $};
\end{align}
\item  \label[Relation]{LL:bbandccrelations}
\begin{align}
\label{E:bbcommutator} [\cB_{i}, \cB_{j}]&=0,\\
\label{E:cccommutator} [\cC_{i}, \cC_{j}]&=0;
\end{align}
\item  \label[Relation]{LL:bcRootvectors}
\begin{align}
\label{E:hbcommutator} [\cH_{i}, \cB_{j}] &= (\tupepsilon_{i}, \tupbeta_{j})\cB_{j},\\
\label{E:hccommutator} [\cH_{i}, \cC_{j}] &= (\tupepsilon_{i}, \tupgamma_{j})\cC_{j}; 
\end{align}
\item  \label[Relation]{LL:bcrelations}
\begin{equation*} [\cB_{i}, \cC_{j}] = \begin{cases} -\cH_{i}+\cH_{i+1}, & j=i;\\
                                             [\cF_{i-1}, \cF_{i}], & j=i-1; \\
                                             [\cE_{i}, \cE_{i+1}], & j=i+1; \\
                                              0, & \text{else;}
\end{cases} 
\end{equation*}
\item  \label[Relation]{LL:beandbfrelations}
\begin{align}
\label{E:bfcommutator} [\cB_{i}, \cF_{j}] &= 0 \text{ if $j \neq i, i+1$}, \\
\label{E:becommutator} [\cB_{i}, \cE_{j}] &= 0 \text{ if $j \neq i, i-1$},\\
\label{E:becommutatorequalsbfcommutator1} [\cB_{i}, \cF_{i}] &= [\cB_{i+1},\cE_{i+1}], \\
\label{E:becommutatorequalsbfcommutator2} [\cB_{i}, \cF_{i+1}] &= [\cB_{i+1},\cE_{i}];
\end{align} 
\item  \label[Relation]{LL:cefrelations}
\begin{align}
\label{E:bccommutator} [\cF_{j}, \cC_{i}] &=0 \text{ if $j\neq i-1$}, \\
\label{E:fccommutator} [\cE_{j}, \cC_{i}] &=0 \text{ if $j\neq i+1$}, \\
\label{E:eccommutator} [\cF_{i}, \cC_{i+1}] &= [\cE_{i+1}, \cC_{i}];
\end{align} 
\item   \label[Relation]{LL:bSerreRelations}
\begin{align}
\label{E:beeSerreRelations1}[[\cB_{i}, \cF_{i}], \cF_{j}] &= \begin{cases} 2 \cB_{i+1}, & j=i+1; \\
                                         0, & \text{ else};
\end{cases} \\
\label{E:befSerreRelations}[[\cB_1, \cF_1], \cE_1] &=  2\cB_{1},  \\
\label{E:beeSerreRelations3} [[\cB_1, \cE_1], \cE_j] &= 0; 
\end{align}
\item  \label[Relation]{LL:cSerrerelations}
\begin{equation*}
[[\cC_{i+1}, \cF_{i}], \cF_{j}] = \begin{cases} \cC_{i}, & j=i+1; \\
                                           0, & j \neq i-1, i+1;
\end{cases}\end{equation*}

\item\label[Relation]{LL:b11isacommutator}
\begin{equation} \label{E:b11ecommutator}
2\cB_{1,1} = [\cE_1, \cB_1];
\end{equation}

\item \label[Relation]{LL:b11relations}
\begin{align}
\label{E:B11RootVector}
[\cH_{i}, \cB_{1,1}] &= (\tupepsilon_i, \tupbeta_{1,1}) \cB_{1,1}, \\
\label{E:B11B11Commutator}
\cB_{1,1}^2 &= 0.
\end{align}
\end{enumerate}
\end{definition}

\begin{remark}  This is not a minimal list of generators and relations.  For example, when $n > 1$  $\cR$  could alternatively be defined with generators $\cE_i$, $\cH_j$, $\cF_i$, $\cB_i$, and $\cC_i$ (where $1 \leq i \leq n-1$ and $1 \leq j \leq n$) satisfying \cref{LL:hcommuterelation} -- \cref{LL:cSerrerelations} with the generator $\cB_{1,1}$ \textit{defined} by \cref{E:b11ecommutator}. In this case \cref{LL:b11relations} is easily seen to be a consequence of the others.  However, when $n = 1$  $\cR_{1}$ is generated by the element $\cH_1$ and $\cB_{1,1}$ subject to the relations \cref{LL:b11relations}.  Omitting $\cB_{1,1}$ in this case would define a proper subsuperalgebra of $\cR_{1}$. 
\end{remark}

It is straightforward to verify using the commutator formulas in \cref{L:Commutators} that there is a superalgebra homomorphism
\begin{equation} \label{E:RhoHomomorphism}
\rho : \cR \to U(\fp(n))
\end{equation}
given on the generators by
\[
\rho(\cE_{i}) = a_{i,i+1}, \quad \rho(\cF_i) = a_{i+1,i}, \quad \rho(\cH_i) = a_{i,i},  \quad \rho(\cB_{1,1}) = b_{1,1}, \quad \rho(\cB_i) = b_{i,i+1},  \quad \rho(\cC_i) = c_{i,i+1}.
\]
The goal of the remainder of this section is to prove $\rho$ is an isomorphism.

\subsection{Additional Relations}
The following relations can be derived from the defining relations of $\cR$.  \begin{lemma} \label{LL:bierelations}
The following relations hold in $\cR$:
\begin{align} \label{LL:bierelations1}
[[\cB_{i}, \cF_{i}], \cE_{j}] &= \begin{cases} 2\cB_{i}, & j=i; \\
                                         0, & \text{else};
\end{cases} \\
\label{LL:bierelations2}
[[\cB_i, \cE_i],\cE_j] &=
\begin{cases} 2 \cB_{i-1}, & j = i-1; \\
                  0, & \text{else}; \end{cases}\\
\label{LL:bierelations3}
[[\cB_{i}, \cE_{i}],\cF_{j}] &=
\begin{cases} 2 \cB_{i}, & j = i; \\
              0, & \text{else}. \end{cases}
\end{align}
\end{lemma}

\begin{proof}  For \cref{LL:bierelations1}, we first prove $[[\cB_{i}, \cF_{i}], \cE_{i}]=2\cB_{i}$ for all $i$.  Since the statement holds when $i=1$ by \cref{E:befSerreRelations}, we may assume $i >1$.  In which case,
\begin{align*}
[[\cB_{i}, \cF_{i}], \cE_{i}] &= [[\cB_{i}, \cE_{i}], \cF_{i}] \\
                       &= [[\cB_{i-1}, \cF_{i-1}], \cF_{i}] \\
                       &= 2\cB_{i},
\end{align*} where the first equality is by the Jacobi identity along with \cref{LL:efrelation,LL:bcRootvectors}, the second equality is by \cref{E:becommutatorequalsbfcommutator1}, and the third is by \cref{E:beeSerreRelations1}.

We next verify $[[\cB_{i}, \cF_{i}], \cE_{j}]=0$ for all $j \neq i$.  The Jacobi identity and \cref{LL:efrelation}  yields 
\[
[[\cB_{i}, \cF_{i}], \cE_{j}] = [[\cB_{i}, \cE_{j}], \cF_{i}].
\]  By \cref{E:bfcommutator}  this is zero unless $j = i$ or $j=i-1$.  Thus, only the case when $j=i-1$ remains.  In which case,
\begin{align*}
[[\cB_{i}, \cF_{i}], \cE_{i-1}] & = [[\cB_{i}, \cE_{i-1}], \cF_{i}] \\
                          & = [[\cB_{i-1}, \cF_{i}], \cF_{i}] \\
                          & = 0,
\end{align*} where the first equality is as before, the second is by \cref{LL:beandbfrelations}, and the third is by \cref{E:beeSerreRelations1}.

To prove \cref{LL:bierelations2}, first note that the case $i = 1$ is \cref{E:beeSerreRelations3}.  Now assume $i > 1$.  In this case, we use the third part of \cref{LL:beandbfrelations} to obtain the equality
$$
[[\cB_{i}, \cE_{i}], \cE_j] = [[\cB_{i-1}, \cF_{i-1}],  \cE_j].
$$
This along with \cref{LL:bierelations1} implies \cref{LL:bierelations2}.

To prove \cref{LL:bierelations3}, first note that the case of $j=i$ follows from \cref{LL:bierelations3} once one applies the Jacobi identity, \cref{LL:efrelation}, and \cref{LL:bcRootvectors}. The case of $j \neq i, i+1$ follows by an application of the Jacobi identity and \cref{E:bfcommutator}.  Thus the only case remaining is $j=i+1$:
\begin{align*}
[[\cE_{i}, \cB_{i}], \cF_{i+1}]&= [\cE_{i}, [\cB_{i}, \cF_{i+1}]] \\
                      &= [\cE_{i}, [\cB_{i+1}, \cE_{i}]] \\
                      &= \frac{1}{2} [\cE_{i}, [[[\cB_{i}, \cF_{i}], \cF_{i+1}], \cE_{i}]]\\
                      &= \frac{1}{2} [[[\cE_{i}, [\cB_{i}, \cE_{i}]], \cF_{i},], \cF_{i+1}]\\
                      &= 0.
\end{align*}  This calculation follows from an application of the Jacobi identity, \cref{E:becommutatorequalsbfcommutator2}, \cref{E:beeSerreRelations1}, repeated uses of the Jacobi identity, and \cref{LL:bierelations2}.
\end{proof}

\subsection{\texorpdfstring{$\cR$}{R} as a \texorpdfstring{$\gl  (n)$}{gl(n)}-Module}   

Note that the results of this subsection are trivial in the case when $n = 1$ so we only consider $n > 1$ in what follows.

\subsubsection{\texorpdfstring{$U(\gl (n))$}{U(gl(n))} as a subalgebra of \texorpdfstring{$\cR$}{Rn}} \label{SSS:superalgebramaps}
Observe that the purely even superalgebra generated by the elements $\cE_{1}, \dotsc ,\cE_{n-1}$, $\cF_{1}, \dotsc , \cF_{n-1}$, and $\cH_{1}, \dotsc , \cH_{n}$ subject to the first five sets of relations in \cref{D:Tdef} is a well-known presentation for the enveloping algebra $U(\gl (n))$ (e.g., see \cite{DG}).  Consequently, if $\cR^{0}$ denotes the subsuperalgebra of $\cR$ generated by these elements, then there is a surjective superalgebra map
\[
\psi : U(\gl (n)) \to \cR^{0}.
\]  
Let $\rho_0$ denote the restriction of $\rho : \cR \to U(\fp(n))$ to $\cR^{0}$.  It is clear that $\rho_0$ has image contained in $U(\fg_{\0})$.  After identifying $\gl(n)$ with $\fg_{\0}$ via isomorphism $A \mapsto \begin{bmatrix} A & 0 \\ 0 & -A^T \end{bmatrix}$, the composite
\[
U(\gl (n)) \xrightarrow{\psi } \cR^{0} \xrightarrow{\rho_{0}} U(\gl (n))
\] is the identity map and it follows that $\psi$ and $\rho_{0}$ are superalgebra isomorphisms. 
 We henceforth identify $U(\gl (n))$ with both $\cR^{0}$ and as a subalgebra of $U(\fp (n))$ via these maps.
 
\subsubsection{\texorpdfstring{$\gl (n)$}{gl(n)}-module structure on \texorpdfstring{$\cR$}{R}}

We can make $\cR$ and $U(\fp(n))$ into $\gl(n)$-modules as follows.  For $x \in \gl(n)$, $a \in R$, and $y \in U(\fp(n))$, we define:
$$
x.a = [\psi(x), a] \quad \text{ and } \quad x.y = [\rho(\psi(x)), y].
$$
With these definitions, the map $\rho : R \to U(\fp(n))$ is a $\gl(n)$-module homomorphism.   Since $\gl(n)$ is identified as a subspace of both $\cR$ and $U(\fp(n))$ via the maps $\psi$ and $\rho \circ \psi$, we omit the maps when writing these actions.

\subsubsection{Some \texorpdfstring{$\gl (n)$}{gl(n)}-submodules of \texorpdfstring{$\cR$}{R}}  For $\tuplambda \in X(T_{n})$ write $M_0(\tuplambda)$ for the Verma module for $U(\gl(n))$ with highest weight $\tuplambda$ (with respect to the Cartan subalgebra $\fh_{\0}$ and Borel subalgebra $\fb_{\0}$ chosen in \cref{SS:RootCombinatorics}),  and write $L_0(\tuplambda)$ for its simple quotient.  In particular the natural module $V_0$ for $U(\gl(n))$ satisfies $V_0 = L_0(\tupepsilon_1)$.  

\begin{lemma}\label{L:Cembedding} 
Let $\cR^{-1}$ denote the $\gl(n)$-submodule of $\cR$ generated by $\cC_{n-1}$. Then 
\[
\cR^{-1} \cong L_{0}(\tupgamma_{n-1}) \cong \Lambda^{2}\left(V_{0} \right)^*
\] as $\gl (n)$-modules.
Moreover, $\cC_i \in \cR^{-1}$ for all $1 \leq i \leq n-1$ and the restriction of $\rho$ to $\cR^{-1}$ induces a $\gl(n)$-module isomorphism $\rho_{-1} : \cR^{-1} \to \fg_{-1}$. 
\end{lemma}

\begin{proof} First, note that $\cR^{-1}$ is a nonzero subspace of $\cR_{\1}$ since $\rho(\cC_{n-1}) = c_{n-1,n} \neq 0$.  Second, by \cref{LL:bcRootvectors} $\cC_{n-1}$ has weight $\tupgamma_{n-1}$ and by \cref{LL:cefrelations} $\cC_{n-1}$ is annihilated by $\cE_{1}, \dotsc , \cE_{n-1}$.  Since $\cC_{n-1}$ is a highest weight vector of highest weight $\tupgamma_{n-1}$, there is a $\gl (n)$-module homomorphism $M_{0}(\tupgamma_{n-1}) \to \cR$ with image equal to $\cR^{-1}$.  

We claim that
\[
\cF_{i}^{(\tupgamma_{n-1}, \tupalpha_{i})+1}.\cC_{n-1} =0
\] for $i=1, \dotsc , n-1$.  More explicitly, this means the following relations hold in $\cR$:

\[
0 = \begin{cases} [\cF_{i}, \cC_{n-1}], & i\neq n-2; \\
                  [\cF_{n-2},[\cF_{n-2}, \cC_{n-1}]], &  i=n-2;\\
\end{cases}
\]  The first case follows from \cref{LL:cefrelations} while the second case follows from \cref{LL:cSerrerelations}.

By \cite[Theorem 21.4]{Humphreys} the previous claim implies the kernel of the map $M_{0}(\tupgamma_{n-1})  \to \cR^{-1}$ contains the unique maximal submodule of $M_{0}(\tupgamma_{n-1})$.  Therefore there is an induced surjective  $\gl(n)$-module homomorphism $L_{0}(\tupgamma_{n-1})  \to \cR^{-1}$.  Simplicity of $L_{0}(\tupgamma_{n-1})$ implies  this map is an isomorphism.

The fact each $\cC_{i}$ is an element of $\cR^{-1}$ can be deduced from the fact that $\cC_{n-1} \in \cR^{-1}$ along with  \cref{LL:cSerrerelations}, which demonstrates that $\cC_{i-1} \in \cR^{-1}$ whenever $\cC_{i} \in \cR^{-1}$.    Because $\rho(\cC_{n-1}) = c_{n-1,n} \in \fg_{-1} \subseteq U(\fp (n))$ and $\rho$ is a $\gl (n)$-module homomorphism, the image of the restriction of $\rho$ to $\cR^{-1}$ is contained in the $\gl(n)$-module $\fg_{-1}$.   Because $\fg_{-1}$ is also simple, we conclude that the restriction, $\rho_{-1} : \cR^{-1} \to \fg_{-1}$, is an isomorphism of $\gl(n)$-modules.
\end{proof}

\begin{lemma}\label{L:Bembedding} Let $\cR^{1}$ denote the $\gl (n)$-submodule of $\cR$ generated by $\cB_{1,1}$.   Then 
\[
\cR_1 \cong L_{0}(\tupbeta_{1,1}) \cong S^{2}\left(V_{0} \right)
\]as $\gl (n)$-modules.
  Furthermore, $\cB_i \in \cR^{1}$ for all $1 \leq i \leq n-1$, and the restriction of $\rho$  to $\cR^{1}$ induces a $\gl(n)$-module isomorphism  $\rho_1 : \cR^{1} \to \fg_{1}$.
  \end{lemma}

\begin{proof}  \cref{LL:b11relations} shows that $\cB_{1,1}$ has weight $\tupbeta_{1,1}$.  Combining \cref{LL:bierelations2} with \cref{LL:b11isacommutator} yields
$$
\cE_i.\cB_{1,1} = [\cE_{i}, [\cE_{1}, \cB_1]]/2 = 0
$$
for all $1 \leq i \leq n-1$, so $\cB_{1,1}$ is a highest weight vector.  Therefore there is a surjective $\gl(n)$-module homomorphism $M_0(\tupbeta_{1,1}) \to \cR^{1}$.

We next claim 
\begin{equation*}
\cF_{i}^{-(\tupbeta_{1,1}, \tupalpha_i)+1}.\cB_{1,1} = 0
\end{equation*}
for $i=1, \dotsc , n-1$.  That is,
\[
0 = \begin{cases} [\cF_{i}, \cB_{1,1}], & i=2, \dotsc , n-1; \\
                  [\cF_{1},[\cF_{1},[\cF_{1}, \cB_{1,1}]]], &  i=1.\\
\end{cases}
\]  The first case follows by replacing $\cB_{1,1}$ with $[\cE_1, \cB_1]/2$ using \cref{E:b11ecommutator} and then applying \cref{LL:bierelations3}.  To verify the second case, first note that \cref{LL:bierelations3} implies
\begin{equation} \label{E:FEBRelation}
[\cF_1, [\cE_1, \cB_1]] = 2\cB_1.
\end{equation}  Applying \cref{E:b11ecommutator}, \cref{E:FEBRelation}, and \cref{E:beeSerreRelations1} yields the following:
\begin{align*}
[\cF_{1}, [\cF_{1}, [\cF_{1}, \cB_{1,1}] ]]
&= [\cF_1,[\cF_1,[\cF_1, [\cE_1, \cB_1]]]] /2 \\
&= [\cF_1, [\cF_1, \cB_1] ]\\
 &= 0.
\end{align*}
By \cite[Theorem 21.4]{Humphreys} the previous claim implies the kernel of the surjective module homomorphism $M_{0}(\tupbeta_{1,1}) \to \cR^{1}$ contains the unique maximal submodule of $M_{0}(\tupbeta_{1,1})$ and so induces a surjective $\gl(n)$-module homomorphism $L_{0}(\tupbeta_{1,1}) \to \cR_1$, which is an isomorphism by the simplicity of $L_{0}(\tupbeta_{1,1})$.

Next we demonstrate $\cB_{i} \in \cR^{1}$ for all $i=1, \dotsc , n-1$.  From \cref{E:FEBRelation}, we have $\cF_1.\cB_{1,1} = \cB_1 \in \cR^{1}$.  This along with the first part of \cref{LL:bSerreRelations} shows $\cB_i$ is an element of $\cR^{1}$ for all $i$. Finally, because $\cR^{1}$ and $\fg_{1}$ are both simple $\gl(n)$-modules and $\rho(\cB_{1,1}) = b_{1,1} \in \fg_1$, it follows that $\rho$ restricts to a non-zero homomorphism $\rho_{1}: \cR^{1} \to \fg_1$ and this must be an isomorphism.  \end{proof}

\subsection{Root Vectors in \texorpdfstring{$\cR$}{R}}
\subsubsection{Definition of root vectors in \texorpdfstring{$\cR$}{R}}
Recall the algebra isomorphism $\rho_0 : \cR^{0} \to U(\fg_{\0})$.  For all $1 \leq i,j \leq n$ there exists a unique element $\cA_{i,j} \in \cR^{0}$ satisfying $\rho_{0}(\cA_{i,j}) = a_{i,j} \in \fg_{\0}$.  In particular, we have $\cE_{i} = \cA_{i,i+1}$, $\cH_{i} = \cA_{i,i}$, $\cF_i = \cA_{i+1,i}$, and $\{\cA_{i,j} \mid 1 \leq i ,j \leq n\}$ is a basis for $\gl(n) \subset \cR^{0}$.

For the root $\tupalpha = \tupalpha_{i,j}= \tupepsilon_i - \tupepsilon_j \in \Phi_0$ set $\cX_{\tupalpha} = \cA_{i,j}$.  Then
$$[\cH_{r}, \cX_{\tupalpha}] = (\tupepsilon_r, \tupalpha) \cX_{\tupalpha},$$
for $1 \leq r \leq n$.  This follows from an application of the isomorphism $\rho_0$.  

Similarly, using the $\gl(n)$-module isomorphism $\rho_1 : \cR^{1} \to \fg_1$ we define $\cB_{i,j} \in \cR^{1}$ for $1 \leq i \leq j \leq n$ by $\rho_1(\cB_{i,j}) = b_{i,j}$.  In particular,  $\cB_{i} = \cB_{i,i+1}$ for all $1 \leq i \leq n-1$.  The vectors $\cB_{i,j}$ form a basis of weight vectors for $\cR^{1}$, where $\cB_{i,j}$ has weight $\tupbeta_{i,j} = \tupepsilon_i + \tupepsilon_j \in \Phi_1$.  For the root $\tupbeta = \tupbeta_{i,j}$ set $\cX_{\tupbeta} = \cB_{i,j}$.  These vectors satisfy the property that
$$
[\cH_{r}, \cX_{\tupbeta} ] = (\tupepsilon_r, \tupbeta) \cX_{\tupbeta},
$$ for $1 \leq r \leq n$. For $i <  j$ it is convenient to write $\cB_{j,i} = \cB_{i,j}$.  

Similarly, using the $\gl(n)$-module isomorphism $\rho_{-1}:\cR^{-1} \to \fg_{-1}$ we define $\cC_{i,j} \in \cR^{-1}$ for $ 1 \leq i < j \leq n$ by $\rho_{-1}(\cC_{i,j}) = c_{i,j}$.  In particular, $\cC_{i} = \cC_{i,i+1}$ for all $1 \leq i \leq n-1$.    For any $\tupgamma = -\tupepsilon_i - \tupepsilon_j \in \Phi_{-1}$ set $X_{\tupgamma} = \cC_{i,j}$ . These vectors satisfy the property that
$$
[\cH_r, \cX_{\tupgamma}] = (\tupepsilon_r, \tupgamma) \cX_{\tupgamma},
$$ for $1 \leq r \leq n$.  If $i < j$, it is convenient to write $\cC_{j,i} =-\cC_{i,j}$ and to set $\cC_{i,i} = 0$.  

\begin{lemma}\label{L:rhosurjective}  The superalgebra homomorphism 
\[
\rho: \cR_{n} \to U(\fp (n))
\] is surjective.

\end{lemma}
\begin{proof}
This follows immediately from the preceding discussion since $\rho$ maps onto a set of basis vectors for $\fp(n)$ and these are known to generate $U(\fp (n))$.
\end{proof}

\subsubsection{Commutators I}\label{SSS:evenandevenrootvectors} In the remainder of this subsection we demonstrate that any commutator of root vectors in $\cR$ is either a scalar multiple of another root vector, or a linear combination of elements of the form $\cH_i$.  These calculations will be the key to establishing a PBW-like basis theorem for $\cR$, and deducing that $\rho : \cR \to U(\fp(n))$ is an isomorphism.

Let $1 \leq i,j,p,q \leq n$ with $i \neq j$ and $p \neq q$. Given $\tupalpha = \tupepsilon_{i}-\tupepsilon_{j}$ and $\tupalpha' = \tupepsilon_{p}-\tupepsilon_{q}$ in $\Phi_{0}$, set 
\[
c_{\tupalpha, \tupalpha'}=\begin{cases} 1, & j=p, i \neq q;\\
                                 -1, & j\neq p, i = q.
\end{cases}
\]  For $\tupalpha = \tupepsilon_{i}-\tupepsilon_{j} \in \Phi_{0}$, set 
\[
\cH_{\tupalpha}= \cH_{i} - \cH_{j}.
\]

Using this notation relations \cref{LL:hcommuterelation,LL:RootVectors,LL:efrelation,LL:eSerrerelations,LL:fSerrerelations} in \cref{D:Tdef} imply the following commutator formula for the even root vectors of $\cR$.  See \cite[Equation (5.9)]{DG} for details.
\begin{lemma}\label{l2}
Given $\tupalpha, \tupalpha' \in \Phi_{0}$,

\begin{equation*}
 [\cX_\tupalpha, \cX_{\tupalpha'}]=\begin{cases}
             \cH_\tupalpha,& \text{if $\tupalpha+\tupalpha'=0$};\\
            c_{\tupalpha,\tupalpha'}\cX_{\tupalpha+\tupalpha'},& \text{if $\tupalpha+\tupalpha' \in \Phi_{0}$};\\
            0,& \text{otherwise}.
            \end{cases}
\end{equation*}
\end{lemma}

\subsubsection{Commutators II}\label{SSS:evenandoddrootvectors1}

Let $1 \leq i,j \leq  n$ with $ i \neq j$, and let $1 \leq p < q \leq m$.  Given $\tupalpha = \tupepsilon_{i}-\tupepsilon_{j} \in \Phi_{0}$ and $\tupgamma = - \tupepsilon_{p} - \tupepsilon_{q} \in \Phi_{-1}$, set 
\[
c_{\tupalpha, \tupgamma} = \begin{cases} 1, & i=p \text{ and } j > q;\\
1, & i=q \text{ and } j < p;\\
                                    -1, & i =p \text{ and } j < q;\\
                                     -1, & i = q \text{ and } j > p.
\end{cases}
\]

\begin{lemma}\label{L:AandCcommutator}  Given $\tupalpha \in \Phi_{0}$ and $\tupgamma \in \Phi_{-1}$, 
\[
[\cX_{\tupalpha}, \cX_{\tupgamma}] = \begin{cases} c_{\tupalpha, \tupgamma} \cX_{\tupalpha+\tupgamma}, & \text{if $\tupalpha + \tupgamma \in \Phi$};\\
                                                                  0, & \text{else}.
\end{cases}
\]
\end{lemma}

\begin{proof} 
Because $\cX_{\tupalpha} = \cA_{i,j}$ is identified with an element of $\gl(n) \subset \cR_0$ and $\cX_{\tupgamma} = \cC_{p,q} \in \cR^{-1}$, the commutator $[\cX_{\tupalpha}, \cX_{\tupgamma}]$ is the action of an element of $\gl(n)$ on the module $\cR^{-1}$.  In particular, this implies the commutator $[\cX_{\tupalpha}, \cX_{\tupgamma}]$ lies in $\cR^{-1}$.  Since $\rho$ is an algebra homomorphism, we use \cref{L:Commutators} to calculate 
\[
\rho \left( [\cX_{\tupalpha}, \cX_{\tupgamma}] \right) = \left[\rho (\cX_{\tupalpha}), \rho (\cX_{\tupgamma}) \right] = \left[a_{i,j}, c_{p,q} \right] = -\delta_{i,p}c_{j,q} + \delta_{i,q}c_{j,p} = -\delta_{i,p} \rho(\cC_{j,q}) + \delta_{i,q} \rho(\cC_{j,p}).
\]   
Because $\rho$ restricts to a linear isomorphism between $\cR^{-1}$ and $\fg_{-1}$, it follows that $[\cX_{\tupalpha}, \cX_{\tupgamma}]$ is zero unless $i = p$ or $q$.  That is, if and only if $\tupalpha + \tupgamma \in \Phi$.  In this case the last expression above is precisely $c_{\tupalpha, \tupgamma} \cX_{\tupalpha + \tupgamma}$.    \end{proof}

\subsubsection{Commutators III}\label{SSS:evenandoddrootvectors2}

Suppose that $\tupalpha = \tupepsilon_i - \tupepsilon_j \in \Phi_0$, and $\tupbeta = \tupepsilon_p + \tupepsilon_q \in \Phi_1$.  Set
\[
c_{\tupalpha,\tupbeta} = 2^{\delta_{i,p} + \delta_{i,q}}.
\]
The following can be proven by the same argument which was used to establish \cref{L:AandCcommutator}.

\begin{lemma}\label{L:AandBcommutator}  Given $\tupalpha \in \Phi_{0}$ and $\tupbeta \in \Phi_{-1}$, 
\[
[\cX_{\tupalpha}, \cX_{\tupbeta}] = \begin{cases}  c_{\tupalpha, \tupbeta} \cX_{\tupalpha+\tupbeta}, & \text{if $\tupbeta + \tupalpha \in \Phi$};\\
                                                                  0, & \text{else}.
\end{cases}
\]
\end{lemma}

\subsubsection{Commutators IV}\label{SSS:oddandoddrootvectors1}

\begin{lemma}\label{L:cccommutator}  For all $\tupgamma_{1}, \tupgamma_{2} \in \Phi_{-1}$ we have 
\[
[\cX_{\tupgamma_{1}}, \cX_{\tupgamma_{2}}]=0.
\] 
In particular, $\cX_{\tupgamma}^{2}=0$ for all $\tupgamma \in \Phi_{-1}$. 
\end{lemma}

\begin{proof}

 Recall, $\cR^{-1}$ is the $\gl (n)$-submodule of $\cR$ generated by $\cC_{n-1}$.  There is a $\gl (n)$-module homomorphism
\[ 
\sigma_{-1}: \cR^{-1} \otimes \cR^{-1} \to \cR
\] given by $\cC \otimes \cC' \mapsto [\cC, \cC'] = \cC \cC' + \cC' \cC$.  We claim this map is identically zero.   This will imply the first claim of the lemma since $\cX_{\tupgamma_1}, \cX_{\tupgamma_2} \in \cR^{-1}$.   In particular, it follows that $2 X_{\tupgamma}^2 = \sigma(X_{\tupgamma} \otimes X_{\tupgamma}) = 0$.  

Because the map $\sigma_{-1}$ is symmetric in the non-super sense, it suffices to demonstrate $\sigma_{-1}$ vanishes on the summand $S^2(\cR^{-1})$, the non-super symmetric square of the $\gl(n)$-module $\cR^{-1}$.  
From \cref{L:Cembedding}, there is a $\gl(n)$-module isomorphism $\cR^{-1} \cong  \Lambda^2(V_0)^*$.  We have the following $\gl (n)$-module isomorphisms:
$$
S^2(\cR^{-1})  \cong S^2(\Lambda^2(V_0))^* \cong L_0(2 \tupepsilon_1 + 2 \tupepsilon_2)^* \oplus L_{0}(\tupepsilon_1+\tupepsilon_2+\tupepsilon_3+\tupepsilon_4)^{*} \cong L_0(2 \tupgamma_{n-1}) \oplus L_{0}( \tupgamma_{n-3} +  \tupgamma_{n-1}),
$$ where the second isomorphism follows from \cite[Corollary 5.7.6]{GoodmanWallach}.  Note, the second summand only appears if $n \geq 4$.

By weight considerations one sees that the summand of $S^2(\cR^{-1})$ isomorphic to $L_0(2 \tupgamma_{n-1})$  is generated by the highest weight vector $\cC_{n-1} \otimes \cC_{n-1}$.  By \cref{LL:bbandccrelations}, it follows that $\sigma_{-1}(\cC_{n-1} \otimes \cC_{n-1})=0$.  Consequently, $\sigma_{-1}$ is identically zero on this summand.

Similarly, weight considerations show that all highest weight vectors of the second summand lie in the subspace spanned by the three vectors:
\begin{align}
&\cC_{n-3} \otimes \cC_{n-1} + \cC_{n-1} \otimes \cC_{n-3}, \label{vector1}\\  
&\cC_{n-2,n} \otimes \cC_{n-3, n-1} + \cC_{n-3,n-1} \otimes \cC_{n-2,n},  \label{vector2} \\
 &\cC_{n-2} \otimes \cC_{{n-3,n}} +  \cC_{{n-3,n}} \otimes \cC_{n-2}. \label{vector3}
\end{align}
The homomorphism
 $\sigma_{-1}$ is zero on the first vector by \cref{LL:bbandccrelations}.  Once we demonstrate that $\sigma_{-1}$ vanishes on the vectors in \cref{vector2} and \cref{vector3}, it follows that $\sigma_{-1}$ is identically zero on the summand of $S^2(\cR^{-1})$ isomorphic to $L(\tupgamma_{n-1} + \tupgamma_{n-3})$ and hence $\sigma_{-1}$ is identically zero on $\cR^{-1} \otimes \cR^{-1}$, as claimed.
 
We begin by establishing some relations for the action of the $\cF_{i}$ on the vectors $\cC_{j}$.    First,
 $$
 \rho([\cF_{n-2}, \cC_{n-1}]) = [a_{n-1, n-2}, c_{n-1,n}] = -c_{n-2, n} = \rho(-\cC_{n-2, n}),
 $$
 where the second equality follows from \cref{L:Commutators}.
 Because $[\cF_{n-2}, \cC_{n-1}] = \cF_{n-2}.\cC_{n-1}$ is an element of $\cR^{-1}$, and the restriction of $\rho$  to  $\cR^{-1}$ is a linear isomorphism, this implies that $\cF_{n-2}. \cC_{n-1}= - \cC_{n-2, n}$.  Similar techniques demonstrate that
 $
 \cF_{n-3} .\cC_{n-2} = -\cC_{n-3, n-1}.
 $
In addition,  \cref{LL:cefrelations} implies that $\cF_{n-2} . \cC_{n-2} = \cF_{n-3} .\cC_{n-1} = 0$, while \cref{LL:cSerrerelations} says that $\cF_{n-2}\cF_{n-3}.\cC_{n-2} = \cC_{n-3}$.  Using the coproduct to act on tensor products of $\gl(n)$-modules and simplifying using the above relations demonstrates that
  $$
\cF_{n-2} \cF_{n-3} .\left( \cC_{n-2} \otimes \cC_{n-1} \right) =  \cC_{n-3} \otimes \cC_{n-1} +  \cC_{n-3, n-1} \otimes \cC_{n-2,n}.
 $$
 Rewriting this equation and applying $\sigma_{-1}$ gives:
 \begin{align*}
 \sigma_{-1}\left( \cC_{n-3, n-1} \otimes \cC_{n-2,n} \right) &= \sigma_{-1}\left( \cF_{n-2} \cF_{n-3} .\left( \cC_{n-2} \otimes \cC_{n-1} \right)\right) - \sigma_{-1}\left(  \cC_{n-3} \otimes \cC_{n-1}\right) \\ 
&=  \cF_{n-2} \cF_{n-3} . \sigma_{-1} \left( \cC_{n-2} \otimes \cC_{n-1} \right) - \sigma_{-1}\left(  \cC_{n-3} \otimes \cC_{n-1}\right) \\ &= 0,
\end{align*}
where the second equality follows from applying \cref{LL:bbandccrelations} to both terms.  Because $\sigma_{-1}$ is symmetric, we conclude $\sigma_{-1}$ vanishes on the vector in \cref{vector2}. 
 
Showing that $\sigma_{-1}$ vanishes on the vector in \cref{vector3} can be done using similar techniques.  In particular, one demonstrates that:
\[
 \cC_{n-2} \otimes \cC_{n-3, n}  = \cF_{n-3} \cF_{n-2} .\left( \cC_{n-2} \otimes \cC_{n-1} \right) - \cC_{n-3,n-1} \otimes  \cC_{n-2,n},
\]
hence
\begin{align*}
 \sigma_{-1} \left( \cC_{n-2} \otimes \cC_{n-3, n}  \right)  &= \cF_{n-3} \cF_{n-2} . \sigma_{-1} \left( \cC_{n-2} \otimes \cC_{n-1} \right) - \sigma_{-1} \left(\cC_{n-3,n-1} \otimes  \cC_{n-2,n}\right) .
\end{align*}
The first term on the right-hand side of the above equation is zero by \cref{LL:bbandccrelations}, while the second term is zero by our work in the previous paragraph. Hence, $\sigma_{-1} \left( \cC_{n-2} \otimes \cC_{n-3, n}  \right) = 0$, and symmetry of $\sigma_{-1}$ implies  that $\sigma_{-1}$ vanishes on the vector in \cref{vector3}.
\end{proof}

\subsubsection{Commutators V}\label{SSS:oddandoddrootvectors2}

\begin{lemma}\label{L:bbcommutator}  For all $\tupbeta_1, \tupbeta_2 \in \Phi_{1}$ we have 
\[
[\cX_{\tupbeta_1}, \cX_{\tupbeta_2}]=0.
\]
In particular, $\cX_{\tupbeta}^{2}=0$ for all $\tupbeta \in \Phi_{1}$. 
\end{lemma}

\begin{proof}
The proof is analogous to that of \cref{L:cccommutator}, but uses the $\gl(n)$-module homomorphism 
\[
\sigma_{1} : \cR^1 \otimes \cR^1 \to \cR
\]
given by $\cB \otimes \cB' \mapsto [\cB, \cB'] = \cB \cB' + \cB' \cB$ .  As in the proof of \cref{L:cccommutator}, the claims of the lemma are immediate once $\sigma_{1}$ is proven to be identically zero.  This can be done by noticing that $\sigma_1$ is symmetric (in the non-super sense), so it suffices to verify that $\sigma_1$ vanishes on the summand $S^2(\cR^1) \subset \cR^1 \otimes \cR^1$.    Because $\cR^1$ is isomorphic to the $\gl(n)$-module $S^2(V_0)$, there are isomorphisms of $\gl(n)$-modules
\begin{equation}\label{E:LRTprime2Tprime2}
S^2(\cR_1) \cong S^2(S^2(V_0)) \cong L_0(4 \tupepsilon_1) \oplus L_0(2 \tupepsilon_1 + 2 \tupepsilon_2),
\end{equation}
where the second isomorphism follows from \cite[Corollary 5.7.4]{GoodmanWallach}.   Equipped with this decomposition, the rest of the proof can be done by modifying the proof of \cref{L:cccommutator} to this setting.  Use the toggle in the \texttt{arXiv} version to see the details.

\begin{answer}

Consider the summand of $S^2(\cR^1)$ isomorphic to $L_0(4 \tupepsilon_1)$.    By weight considerations, this summand has highest weight vector $\cB_{1,1} \otimes \cB_{1,1}$.  Since $\sigma_{1} (\cB_{1,1} \otimes \cB_{1,1})  = [\cB_{1,1}, \cB_{1,1}] = 2\cB_{1,1}^2 = 0$ by \cref{E:B11B11Commutator}, it follows that $\sigma_1$ is zero on this summand.

Next, consider the summand isomorphic to  $L_0(2 \tupepsilon_1 + 2 \tupepsilon_2)$. By weight considerations the highest weight vector in this summand must be a linear combination of the vectors $\cB_1 \otimes \cB_{1}$ and $\cB_{1,1} \otimes \cB_{2,2} + \cB_{2,2} \otimes \cB_{1,1}$.  Therefore it suffices to show $\sigma_{1}$ vanishes on these vectors.    The fact that $\sigma_1$ vanishes on $\cB_{1} \otimes \cB_{1}$ follows immediately from \cref{LL:bbandccrelations}. 

To show $\sigma_{1}\left(\cB_{1,1} \otimes \cB_{2,2} + \cB_{2,2} \otimes \cB_{1,1} \right)=0$ it suffices to show $\sigma_{1} (\cB_{1,1} \otimes \cB_{2,2} )=\sigma_{1}(\cB_{2,2} \otimes \cB_{1,1})=0$. To do so, we establish several useful identities.  First, note that 
\[
\rho([\cE_2, \cB_2]) = [a_{2,3}, b_{2,3}] = 2 b_{2,2}=2\rho_1(\cB_{2,2}),
\]
where the middle equality follows from \cref{L:Commutators}.  Since $\rho$ is an isomorphism on $R_{1}$ this equality can be formulated in terms of the action of $\gl (n)$ on $R_{1}$ as $\cE_2. \cB_2 = 2 \cB_{2,2}$.  Second, observe that  \cref{E:b11ecommutator,LL:bierelations2,E:becommutator} can be restated in terms of the action of $\gl(n)$ on $R_{1}$ as $\cE_{1}.\cB_1 = 2\cB_{1,1}$
 $(\cE_1 \cE_2). \cB_2 = 2\cB_1$, and $\cE_{2}.\cB_1 = 0$, respectively.

Combining these relations with the action of $\gl(n)$ on $R^1 \otimes R^1$ via the coproduct, we can simplify:
$$
\cE_{1} \cE_{2} . \left( \cB_1 \otimes \cB_2 \right) = 4 \cB_{1,1} \otimes \cB_{2,2} + 2 \cB_{1} \otimes \cB_{1}.
$$
From \cref{LL:bbandccrelations} we have $\sigma_{1}(\cB_1 \otimes \cB_2) = 0$, and because $\sigma_1$ is a $\gl(n)$-homomorphism, it follows that $\sigma_1$ vanishes on the left-hand side of the above equation.  Again from \cref{LL:bbandccrelations}, $\sigma_1$ also vanishes on the second term of the right-hand side of the equation above, and we conclude that $\sigma_1$ vanishes on $\cB_{1,1} \otimes \cB_{2,2}$, as claimed.

Finally, since $\cX_{\tupbeta}$ is odd we have $2\cX_{\tupbeta}^{2}=[\cX_{\tupbeta}, \cX_{\tupbeta}]=0$, which implies the last claim of the lemma.
\end{answer}
\end{proof}

\subsubsection{Commutators VI}\label{SSS:oddandoddrootvectors3}

\begin{lemma}\label{L:BCcommutatorisinA} Given $\tupbeta \in \Phi_{-1}$ and $\tupgamma \in \Phi_{1}$, we have 
\[
[\cX_{\tupbeta}, \cX_{\tupgamma}] \in \cR^{0}.
\]
\end{lemma}

\begin{proof} 
As in \cref{L:Cembedding,L:Bembedding}, let $\cR^{-1} \cong \Lambda^2(V_0)^*$ and $\cR^{1} \cong S^2(V_0)$ denote the $\gl (n)$-submodules of $\cR$ generated by  $\cC_{n-1}$ and $\cB_{1,1}$, respectively.  There is a map of $\gl (n)$-modules given by 
\[
\sigma_{2} : \cR^{1} \otimes \cR^{-1} \to \cR  
\] given by $\sigma_{2} (\cB \otimes \cC) = [\cB,\cC]= \cB\cC+\cC\cB$.  We claim the image of $\sigma_{2}$ lies in $\cR_0$, from which the lemma follows.

By the Littlewood-Richardson rule 
\begin{equation}\label{E:AnotherDecomposition}
\cR^{1} \otimes \cR^{-1} \cong L_{0}(\tupepsilon_{1}-\tupepsilon_{n}) \oplus L_{0}(2 \tupepsilon_1 - \tupepsilon_{n-1} - \tupepsilon_n),
\end{equation}
where the second summand is omitted if $n = 2$. Since $\sigma_{2}$ is a $\gl (n)$-module homomorphism it suffices to consider $\sigma_{2}$ on each summand.

We first consider the case when $n > 2$.  Since $\cB_{1,1} \otimes \cC_{n-1} \in \cR^{1} \otimes \cR^{-1}$ is a highest weight vector for $\gl (n)$ of weight $2 \tupepsilon_1 - \tupepsilon_{n-1} - \tupepsilon_n$, it generates the summand isomorphic to  $L_{0}(2\tupepsilon_{1} -\tupepsilon_{n-1} - \tupepsilon_{n})$.  We have:
$$
2\sigma_{2}(\cB_{1,1}\otimes \cC_{n-1}) = [[\cE_1, \cB_1], \cC_{n-1}]  = -[[\cE_{1},\cC_{n-1}], \cB_{1}] + [\cE_{1}, [\cB_1, \cC_{n-1}]]  = [\cE_{1}, [\cB_{1},\cC_{n-1}]].
$$
where the first equality is from \cref{LL:b11isacommutator}, the second is the Jacobi identity, and the third follows from \cref{E:fccommutator}.  If $n >3$, the last expression is zero by \cref{LL:bcrelations}.  If $n = 3$, then we have:
$$
[\cE_1, [\cB_1, \cC_2]] = [\cE_{1}, [\cE_1, \cE_2]] = 0,
$$
where the first equality follows from \cref{LL:bcrelations}, and the second from \cref{LL:eSerrerelations}.  In either case, we see that $\sigma_{2}(\cB_{1,1} \otimes \cC_{n-1})$ is zero and, hence, $\sigma_{2}$ vanishes on the summand isomorphic to $L_0(2 \tupepsilon_1 - \tupepsilon_{n-1} - \tupepsilon_n)$.  This implies that the image of $\sigma_2$ is equal the restriction of $\sigma_2$ to the simple summand isomorphic to $L_0(\tupepsilon_1 - \tupepsilon_n)$.  In particular, the image of $\sigma_2$ is either simple or zero.

 \cref{LL:bcrelations} shows $\sigma_{2} (\cB_{1}\otimes \cC_{1})= \cH_2 - \cH_1 \in \cR^0$.  This is non-zero, because $\rho(\cH_2 - \cH_1) = a_{2,2} - a_{1,1}$ is a non-zero element of $U(\fp(n))$.  Hence, the image of $\sigma_{2}$ is a simple $\gl (n)$-module which intersects nontrivially with the module $\cR^{0}$.  Therefore the image of $\sigma_{2}$ is completely contained in $\cR^0$, as desired.

When $n=1$ the result holds trivially.
When $n = 2$, $\cR^{1} \otimes \cR^{-1} \cong L_0(\tupepsilon_1 - \tupepsilon_n)$ and the arguments from the previous paragraph can be readily adapted.  We leave the details to the reader.
\end{proof}

Given $1 \leq i < j \leq n$ and $1 \leq p < q \leq n$ and $\tupbeta = \tupepsilon_{i} + \tupepsilon_{j}, \tupgamma = -\tupepsilon_{p}-\tupepsilon_{q} \in \Phi_{1}$ define
\[
c_{\tupbeta, \tupgamma} = \begin{cases} -1, & i = q \text{ or } j = q; \\
                                  1, & i = p \text{ or } j = p.
\end{cases}
\]

\begin{lemma}\label{L:BCcommutator}  Let $\tupbeta = \tupepsilon_{i} + \tupepsilon_{j} \in \Phi_{1}$ and $\tupgamma = -\tupepsilon_{p}-\tupepsilon_{q} \in \Phi_{-1}$.  Then 
\[
[\cX_{\tupbeta}, \cX_{\tupgamma}] = \begin{cases}  -\cH_{\tupepsilon_{i}-\tupepsilon_{j}}, & \tupbeta + \tupgamma = 0; \\
                                         c_{\tupbeta, \tupgamma}\cX_{\tupbeta + \tupgamma}, & \text{if $\tupbeta+\tupgamma \in \Phi$};\\
                                         0, & \text{otherwise}.
\end{cases}
\]
\end{lemma}

\begin{proof}   By \cref{L:BCcommutatorisinA}, $[\cX_{\tupbeta}, \cX_{\tupgamma}]$ lies in $\cR_0$.  By the discussion in \cref{SSS:superalgebramaps} the map $\rho_{0}: \cR^{0} \to U(\fg_0)$ is an algebra isomorphism.  Thus it suffices to compute $$\rho_{0}\left([\cX_{\tupbeta}, \cX_{\tupgamma}] \right) = \left[\rho (\cX_{\tupbeta}), \rho (\cX_{\tupgamma}) \right] = \left[b_{i,j}, c_{p,q} \right].$$  The claim then follows from the commutator formulas in \cref{L:Commutators}.   
\end{proof}

\subsection{Divided Powers in \texorpdfstring{\cR}{R}}\label{SS:dividedpowers}

For $\cX \in \cR$ and $k\in \Z_{\geq 0}$, we define the \emph{$k$th divided power of $\cX$} to be
\[
\cX^{(k)}=\frac{\cX^k}{k!} \in  \cR.
\] 
For all $a,b \geq 0$ and $\cX \in \cR$ we have
\begin{equation}
 \cX^{(a)}\cX^{(b)}=\binom{a+b}{a}\cX^{(a+b)}.\label{eq9}
\end{equation}

Given $\tupalpha = \tupepsilon_{i}-\tupepsilon_{j} \in \Phi_{0}$, recall that 
\[
\cH_{\tupalpha} = \cH_{i}-\cH_{j}.
\]
For $i=1, \dotsc , n$ and $k\geq 0$ define
\[
\binom {\cH_{i}}{k}=\frac{\cH_i(\cH_i-1)\cdots (\cH_i-k+1)}{k!} \in \cR.
\] By definition,
\[
\binom{\cH_i}{0}=1.
\]  

Given $\bsym \delta \in \Phi_{\pm 1}$ or $\tupalpha \in \Phi_{0}$, since $\cX_{\bsym \delta}$ and $\cX_{\tupalpha}$ are root vectors a straightforward computation shows 
\begin{align}\label{E:binomialandrootcommutator}
\binom{\cH_{i}}{s}\cX_{\bsym \delta} &= \cX_{\bsym \delta}\binom{\cH_{i}}{s} + \binom{(\tupepsilon_{i}, \bsym \delta)}{s},\\
\binom{\cH_{i}}{s}\cX_{\tupalpha}^{(r)} &= \cX_{\tupalpha}^{(r)}\binom{\cH_{i}}{s} + \binom{(\tupepsilon_{i}, r\tupalpha)}{s}.
\end{align}
There is no relation for the divided power $\cX^{(r)}_{\bsym \delta}$ for $r > 1$ because $\cX_{\bsym \delta}^{(r)} = 0$  by \cref{L:cccommutator} and \cref{L:bbcommutator}.

\subsection{Commutator Formulas for Divided Powers}\label{SS:CommutatorFormulas}

We now record the commutator formulas for the divided powers of the root vectors in $\cR$.  

First we consider divided powers of root vectors for two roots in $\Phi_{0}$. 
\begin{lemma}\label{L:Acommutators}  Let $\tupalpha,\,\tupalpha'\in \Phi_{0}$ and $r,s > 0$.
\begin{equation*}
\cX_\tupalpha^{(r)} \cX_{\tupalpha'}^{(s)}=\begin{cases}
             \cX_{\tupalpha'}^{(s)}\cX_\tupalpha^{(r)}+\displaystyle{\sum_{j=1}^{\min(r,s)}\cX_{\tupalpha'}^{(s-j)}\binom {\cH_\tupalpha-r-s+2j}{j}
 \cX_\tupalpha^{(r-j)}},& \text{if $\tupalpha+\tupalpha'=0$};\\
           \cX_{\tupalpha'}^{(s)} \cX_\tupalpha^{(r)}+\displaystyle{\sum_{j=1}^{\min(r,s)}c_{\tupalpha,\tupalpha'}^j \cX_{\tupalpha'}^{(s-j)}\cX_{\tupalpha+\tupalpha'}^{(j)}
 \cX_\tupalpha^{(r-j)}},& \text{if $\tupalpha+\tupalpha' \in \Phi$};\\
            \cX_{\tupalpha'}^{(s)} \cX_\tupalpha^{(r)}, & \text{otherwise}.
            \end{cases}\label{eq3}
\end{equation*}
\end{lemma}

\begin{proof} As \eqref{eq3} involves purely even root vectors, it follows from the classical case (see \cite[(5.11a)-(5.11c)]{DG}). 
\end{proof}

Next, consider a divided power of an even root vector with an odd root vector.  A straightforward induction on $r$ proves the following formulas.

\begin{lemma}\label{L:AandBCCdividedpowercommutator}  Let $\tupalpha \in \Phi_{0}$, $\tupbeta \in \Phi_{-1}$, $\tupgamma \in \Phi_{1}$, and let $r > 0$.  Then,
\begin{align*}
\cX_{\tupbeta}\cX_{\tupalpha}^{(r)} &= \begin{cases} \cX_{\tupalpha}^{(r)}\cX_{\tupbeta} + \cX_{\tupalpha}^{(r-1)}\cX_{\tupbeta+\tupalpha}, &\text{if $\tupbeta +\tupalpha \in \Phi$;} \\
0, & \text{otherwise,}
\end{cases}\\
\cX_{\tupalpha}^{(r)}\cX_{\tupgamma} &= \begin{cases} \cX_{\tupgamma}\cX_{\tupalpha}^{(r)} +c_{\tupalpha, \tupgamma} \cX_{\tupalpha+\tupgamma}\cX_{\tupalpha}^{(r-1)}, &\text{if $\tupalpha +\tupgamma \in \Phi$;} \\
0, & \text{otherwise.}
\end{cases}
\end{align*}

\end{lemma}

Finally, we record the cases of two odd root vectors.

\begin{lemma}\label{L:BBandCCcommutator}  Let $\tupbeta, \tupbeta' \in \Phi_{-1}$ and $\tupgamma, \tupgamma' \in \Phi_{1}$.  Then
\begin{align*}
\cX_{\tupbeta}\cX_{\tupbeta'} &= - \cX_{\tupbeta'}\cX_{\tupbeta},\\
\cX_{\tupgamma}\cX_{\tupgamma'} &= - \cX_{\tupgamma'}\cX_{\tupgamma}.
\end{align*}

\end{lemma}

\begin{lemma}\label{L:BandCcommutator}  Let $\tupbeta = -\tupepsilon_{i}-\tupepsilon_{j} \in \Phi_{-1}$ and $\tupgamma = \tupepsilon_{p}+\tupepsilon_{q} \in \Phi_{1}$.  Then 
\[
\cX_{\tupbeta}\cX_{\tupgamma} = \begin{cases} -\cX_{\tupgamma}\cX_{\tupbeta}+ \cH_{\tupepsilon_{i}-\tupepsilon_{j}}, & \tupbeta + \tupgamma = 0; \\
                                       -\cX_{\tupgamma}\cX_{\tupbeta}+  c_{\tupbeta, \tupgamma}\cX_{\tupbeta + \tupgamma}, & \text{if $\tupbeta+\tupgamma \in \Phi$};\\
                                       -\cX_{\tupgamma}\cX_{\tupbeta}, & \text{otherwise}.
\end{cases}
\]
\end{lemma}

\subsection{Kostant Monomials}\label{SS:KostantMonomials}  A \emph{Kostant monomial} of $\cR$ is an element which is an arbitrary product of finitely many elements from among:
\[ \cX_{\tupdelta}^{d},\;\; \cX_{\tupalpha}^{(r)},\;\; \binom{\cH_{i}}{s},
\]  where $\tupdelta \in \Phi_{\pm 1}$ and $d \in \left\{0,1 \right\}$; $\tupalpha \in \Phi_{0}$ and $r \in \Z_{\geq 0}$; and $1 \leq i \leq n$ and $s \in \Z_{\geq 0}$.  Let $\cR_{\ZZ}$ denote the $\ZZ$-span of the Kostant monomials in $\cR$.  Since the set of Kostant monomials is closed under multiplication, $\cR_{\ZZ}$ is a $\ZZ$-subalgebra of $\cR$.

Fix total orders on $\Phi_{-1}$, $\Phi^{-}_{0}$, $\Phi^{+}_{0}$, and $\Phi_{1}$.  We define an \emph{ordered Kostant monomial} of $\cR$ to be an element of the form
\[
\prod_{\tupbeta \in \Phi_{-1}} \cX_{\tupbeta}^{d_{\tupbeta}} \prod_{\tupalpha \in \Phi^{-}_{0}} \cX_{\tupalpha}^{(r_{\tupalpha})} \prod_{i=1}^{n} \binom{\cH_{i}}{s_{i}} \prod_{\tupalpha' \in \Phi^{+}_{0}} \cX_{\tupalpha'}^{(r_{\tupalpha'})} \prod_{\tupgamma \in \Phi_{1}} \cX_{\tupgamma}^{d_{\tupgamma}},
\] where $d_{\tupbeta}, d_{\tupgamma} \in \left\{0,1 \right\}$, $r_{\tupalpha}, r_{\tupalpha'} \in  \Z_{\geq 0}$, $s_{i} \in \Z_{\geq 0}$, and all products are taken in the order of their respective index sets.

\begin{proposition}\label{P:OrderedKostantMonomialsSpan}  Every Kostant monomial in $\cR$ can be written as a $\ZZ$-linear combination of ordered Kostant monomials.  In particular, the ordered Kostant monomials provide a $\ZZ$-spanning set for $\cR_{\ZZ}$ and a $\k$-spanning set for $\cR$.
\end{proposition}
{}
\begin{proof} Using the commutator formulas from \cref{SS:CommutatorFormulas} one can rewrite any Kostant monomial as a $\ZZ$-linear combination of Kostant monomials of the form 
\begin{equation}\label{E:partialKostantMonomial}
\prod_{\tupbeta \in \Phi_{-1}}\cX^{d_{\tupbeta}}_{\tupbeta} A \prod_{\tupgamma \in \Phi_{1}} \cX_{\tupgamma}^{d_{\tupgamma}},
\end{equation}
where the products are taken in the order which was fixed on the index sets $\Phi_{\pm 1}$, $d_{\tupbeta}, d_{\tupgamma} \in \{0,1 \}$ for all $\tupbeta, \tupgamma$, and where $A$ is a Kostant monomial which is a product of elements from among 
\[\cX_{\tupalpha}^{(r)} \;\; \text{and} \;\; \binom{\cH_{i}}{s},
\] where $\tupalpha \in \Phi_{0}$, $r \in \Z_{\geq 0}$,  $1 \leq i \leq n$, and $s \in \Z_{\geq 0}$.  Because $A$ is a Kostant monomial in the sense of \cite{DG} for $U(\gl (n))$, it can be written as an integral linear combination of ordered Kostant monomials of the form 
\[
\prod_{\tupalpha \in \Phi^{-}_{0}} \cX_{\tupalpha}^{(r_{\tupalpha})} \prod_{j=1}^{n} \binom{\cH_{i}}{s_{i}} \prod_{\tupalpha' \in \Phi^{+}_{0}} \cX_{\tupalpha'}^{(r_{\tupalpha'})},
\] where $r_{\tupalpha}, r_{\tupalpha'} \in \Z_{\geq 0}$ and $s_{i} \in \Z_{\geq 0}$.  Substituting this into \cref{E:partialKostantMonomial} shows an arbitrary Kostant monomial can be rewritten as claimed.  In particular, $\cR_{\ZZ}$ is spanned by the ordered Kostant monomials.

 Since the generators used to define $\cR$ are Kostant monomials and any product of Kostant monomials is a Kostant monomial, it follows that the $\k$-span of the Kostant monomials equals $\cR$.  From the previous paragraph this implies the $\k$-span of the ordered Kostant monomials equals $\cR$.
\end{proof}

\begin{theorem}\label{T:OrderedKostantMonomialsBasisandIsomorphism}  The set of ordered Kostant monomials is a $\ZZ$-basis for $\cR_{\ZZ}$, a $\k$-basis for $\cR$, and the  map 
\[
\rho: \cR \to U(\fp (n))
\] is an isomorphism of $\k$-superalgebras.
\end{theorem}

\begin{proof} Recall that $\rho: \cR \to U(\fp (n))$ defines a surjective algebra homomorphism.  By the PBW theorem for $U(\fp (n))$, applying $\rho$ to the ordered Kostant monomials yields a linearly independent set in $U(\fp (n))$ which implies the ordered Kostant monomials are $\ZZ$- and $\k$-linearly independent in $\cR_{\ZZ}$ and $\cR$, respectively.  Therefore the ordered Kostant monomials form a basis. Since $\rho$ takes a basis to a basis and preserves parity it is a superalgebra isomorphism.
\end{proof}

Going forward will identify $\cR$ and $U(\fp (n))$ via $\rho$.  In particular, we view \cref{D:Tdef} as a presentation of $U(\fp (n))$ and $\cR_{\ZZ}$ as a Kostant $\ZZ$-form for $U(\fp (n))$. 

\subsection{The Category \texorpdfstring{$\bU(\fp(n))$}{Udotp(n)}}\label{SS:Udot}  We next introduce a categorical version of Lusztig's idempotent form for $U(\fp (n))$.

When writing compositions of morphisms it is convenient to write them as products (e.g., $fg = f \circ g$).   To lighten notation we use the same name for morphisms between different objects and leave the objects implicit.  When the objects need to be specified, we do this by precomposing with the identity morphism.  As an example of these conventions using the notation established below, the morphism $\cC_{j} \cE_{i} 1_{\tuplambda}$ denotes the composition of $\cE_{i} : \tuplambda \to \tuplambda + \tupalpha_i$ with $\cC_{j} : \tuplambda + \tupalpha_i \to \tuplambda + \tupalpha_{i}+\tupgamma_{j}$.  Moreover, if $f$ is a family of morphisms of the same name which can be composed and $k \geq 1$, then  $f^{(k)}$ denotes the ``divided power'' morphism $\frac{f^{k}}{k!}$.  Commutators are similarly interpreted; for example, $[\cB_{1}, \cC_{j}]1_{\tuplambda} = \cB_{i}\cC_{j}1_{\tuplambda} + \cC_{j}\cB_{i}1_{\tuplambda}: \tuplambda \to \tuplambda + \tupbeta_{i}+\tupgamma_{j}$.

The following definition uses the notation established in \cref{SS:RootCombinatorics}.
\begin{definition}\label{D:Udot}
Let $\bU(\fp(n))$ be the $\k$-linear supercategory with set of objects $X(T_{n})$.  For every $\tuplambda \in X(T_{n})$ and $1 \leq i \leq n-1$ the generating morphisms are:
\begin{align*}
\cE_{i} &: \tuplambda \to \tuplambda+\tupalpha_{i},\\
\cF_{i} &: \tuplambda \to \tuplambda-\tupalpha_{i},\\
\cB_{1,1} &: \tuplambda \to \tuplambda + \tupbeta_{1,1},\\
\cB_{i} &: \tuplambda \to \tuplambda+\tupbeta_{i},\\
\cC_{i} &: \tuplambda \to \tuplambda+\tupgamma_{i},
\end{align*} where the $\Z_{2}$-grading on morphisms is given by declaring $\cE_{1}, \dotsc , \cE_{n-1}$, $\cF_{1}, \dotsc , \cF_{n-1}$ to be even and $\cB_{1,1}, \cB_{1}, \dotsc , \cB_{n-1}$, $\cC_{1}, \dotsc , \cC_{n-1}$ to be odd. We write $1_{\tuplambda}: \tuplambda \to \tuplambda$ for the identity morphism on $\tuplambda$.  

The morphisms in $\bU(\fp(n))$ are subject to the following relations for all objects $\tuplambda = \sum_{i=1}^{n} \lambda_{i}\tupepsilon_i \in X(T_{n})$ and all $1 \leq i,j \leq n-1$:
\begin{enumerate}[label=({U}{\arabic*})]
\item  \label[Relation]{LL:Uefrelation}
\begin{equation}
\label{E:Uefcommutator} [\cE_{i}, \cF_{j}]1_{\tuplambda}  = \delta_{i,j}(\tupalpha_{i}, \tuplambda)1_{\tuplambda}; 
\end{equation}
\item  \label[Relation]{LL:UeSerrerelations}
\begin{align}
\label{E:Ueecommutator} [\cE_{i},\cE_{j}]1_{\tuplambda} &= 0 \text{ if $j \neq i \pm 1$}, \\
\label{E:UeSerreRelation} [\cE_{i},[\cE_{i}, \cE_{j}]]1_{\tuplambda} &=0 \text{ if  $j = i \pm 1 $};
\end{align}
\item  \label[Relation]{LL:UfSerrerelations}
\begin{align}
\label{E:Uffcommutator}  [\cF_{i},\cF_{j}] 1_{\tuplambda} &= 0 \text{ if $j \neq i \pm 1$}, \\
\label{E:UfSerreRelation} [\cF_{i},[\cF_{i},\cF_{j}]] 1_{\tuplambda} &=0 \text{ if  $j = i \pm 1 $};
\end{align}
\item  \label[Relation]{LL:Ubbandccrelations}
\begin{align}
\label{E:Ubbcommutator} [\cB_{i}, \cB_{j}] 1_{\tuplambda}&=0,\\
\label{E:Ucccommutator} [\cC_{i}, \cC_{j}] 1_{\tuplambda} &=0;
\end{align}
\item  \label[Relation]{LL:Ubcrelations}
\begin{equation*} [\cB_{i}, \cC_{j}] 1_{\tuplambda} = \begin{cases} -(\tupalpha_{i}, \tuplambda)1_{\tuplambda}; & j=i;\\ 
                                             [\cF_{i-1}, \cF_{i}] 1_{\tuplambda}, & j=i-1; \\
                                             [\cE_{i}, \cE_{i+1}] 1_{\tuplambda} , & j=i+1; \\
                                              0, & \text{else;}
\end{cases} 
\end{equation*}
\item  \label[Relation]{LL:Ubeandbfrelations}
\begin{align}
\label{E:Ubfcommutator} [\cB_{i}, \cF_{j}]1_{\tuplambda} &= 0 \text{ if $j \neq i, i+1$}, \\
\label{E:Ubecommutator} [\cB_{i}, \cE_{j}] 1_{\tuplambda} &= 0 \text{ if $j \neq i, i-1$},\\
\label{E:Ubecommutatorequalsbfcommutator1} [\cB_{i}, \cF_{i}] 1_{\tuplambda} &= [\cB_{i+1},\cE_{i+1}]1_{\tuplambda}, \\
\label{E:Ubecommutatorequalsbfcommutator2} [\cB_{i}, \cF_{i+1}]1_{\tuplambda}  &= [\cB_{i+1},\cE_{i}] 1_{\tuplambda} ;
\end{align} 
\item  \label[Relation]{LL:Ucefrelations}
\begin{align}
\label{E:Ubccommutator} [\cF_{j}, \cC_{i}] 1_{\tuplambda}&=0 \text{ if $j\neq i-1$}, \\
\label{E:Ufccommutator} [\cE_{j}, \cC_{i}] 1_{\tuplambda}&=0 \text{ if $j\neq i+1$}, \\
\label{E:Ueccommutator} [\cF_{i}, \cC_{i+1}]1_{\tuplambda} &= [\cE_{i+1}, \cC_{i}]1_{\tuplambda};
\end{align} 
\item   \label[Relation]{LL:UbSerreRelations}
\begin{align}
\label{E:UbeeSerreRelations1}[[\cB_{i}, \cF_{i}], \cF_{j}] 1_{\tuplambda} &= \begin{cases} 2 \cB_{i+1}1_{\tuplambda}, & j=i+1; \\
                                         0, & \text{ else};
\end{cases} \\
\label{E:UbefSerreRelations}[[\cB_1, \cF_1], \cE_1]1_{\tuplambda} &=  2\cB_{1}1_{\tuplambda},  \\
\label{E:UbeeSerreRelations3} [[\cB_1, \cE_1], \cE_j] 1_{\tuplambda}&= 0, 
\end{align} 
\item \label[Relation]{LL:UcSerrerelations}
\begin{equation*}
[[\cC_{i+1}, \cF_{i}], \cF_{j}]1_{\tuplambda}  = \begin{cases} \cC_{i} 1_{\tuplambda} , & j=i+1; \\
                                           0, & j \neq i-1, i+1;
\end{cases}\end{equation*}
\item\label[Relation]{LL:Ub11isacommutator}
\begin{equation} \label{E:Ub11ecommutator}
2 \cB_{1,1}1_{\tuplambda} =  [\cE_1, \cB_1] 1_{\tuplambda};
\end{equation}
\item \label[Relation]{LL:Ub11relations}
\begin{align}
\label[Relation]{E:UB11B11Commutator}
\cB_{1,1}^2 1_{\tuplambda}&= 0.
\end{align}
\end{enumerate}
\end{definition}

A representation of $\bU(\fp(n))$ is a superfunctor $\mathcal{W} : \bU(\fp(n)) \to \sVec$, where $\sVec$ is the supercategory of $\k$-superspaces.   To such a functor we may associate a superspace $W = \bigoplus_{\tuplambda \in X(T_n)} W_{\tuplambda}$, where $W_\tuplambda := \mathcal{W}(\tuplambda)$.

Let
\[
\dot{U}(\fp(n)) =\bigoplus_{\tuplambda,\tupmu\in X(T_{n})} \Hom_{\bU(\fp(n))}(\tuplambda, \tupmu).
\]  Then $\dot{U}(\fp(n))$ is a locally unital $\k$-superalgebra with product given by composition of morphisms when their composition is defined and zero otherwise.  There is also a distinguished family of orthogonal idempotents $\left\{1_{\tuplambda} \mid \tuplambda \in X(T_{n}) \right\}$.  There is an equivalence between representations of the supercategory $\bU(\fp(n))$ and locally unital modules for the superalgebra $\dot{U}(\fp(n))$ which takes the functor $\mathcal{W}$ to the superspace $W$.  In particular, the idempotent $1_{\tuplambda}$ acts by $1_{\tuplambda}W = W_\tuplambda = \mathcal{W}(\tuplambda)$.  In turn, thanks to the presentation of $U(\fp(n))$ given in \cref{T:OrderedKostantMonomialsBasisandIsomorphism}, modules for $\dot{U}(\fp(n))$ are equivalent to weight modules for $U(\fp (n))$ whose weights lie in $X(T_{n})$.   Going forward we will freely go between these notions.

\section{Another Lie Superalgebra of Type \texorpdfstring{$P$}{P}}\label{S:Ptilde}

\subsection{The Lie Superalgebra \texorpdfstring{$\tfp (n)$}{tildep(n)}}\label{SS:Ptilde} We now introduce another Lie superalgebra of type $P$.  While it is isomorphic to $\fp (n)$, it will be useful in what follows.

Let $W_n = \k^{n|n}$ be the superspace with homogeneous basis $\{w_{i} \mid i \in I_{n|n}\}$ and parity given by $\p{w_{i}} = \p{i}$.  Equip $W_{n}$ with a non-degenerate, odd,  skew-supersymmetric bilinear form $\lrangle{w_i, w_j} = (-1)^{\p{i}} \delta_{i,-j}$.  Let $\tfp(W_{n})$ denote the corresponding Lie superalgebra of linear maps which leave this bilinear form invariant.  That is, $\tfp (W_{n}) $ is defined as in \cref{E:DefofP}  using the bilinear form $\langle \cdot, \cdot \rangle $.   It can be regarded as a subalgebra of $\gl(n|n)$ using the chosen basis of $W_n$ and is given by
\begin{equation} \label{E:FormOfTildeP}
\tfp(n) = \left\{ \begin{bmatrix}
A & C \\ B & -A^T
\end{bmatrix}\right\} \subset \gl(n|n), 
\end{equation}
where $A, B$, and $C$ are $n \times n$ matrices with entries in $\k$, $A^T$ is the transpose of $A$, the matrix $B$ is symmetric, and $C$ is skew-symmetric.

\subsection{Root Combinatorics for \texorpdfstring{$\tfp (n)$}{tildep(n)}} Fix the Cartan subalgebra $\fh$ to be the diagonal matrices in $\tfp(n)$.  Since this is  also the Cartan subalgebra of $\fp(n)$, we abuse notation by giving it the same name and reusing the notation for the weight lattice $X(T_{n})$ defined in \cref{SS:RootCombinatorics}.  The roots and root vectors for $\tfp(n)$ are:
\begin{center}
\begin{table}[h]
\begin{tabular}{|c|c|c|}
\hline
Root Vector          &  Root         & Allowed Indices                  \\ \hline \hline
$\ta_{i,j} := E_{i,j} - E_{-j,-i}$                                                                                                                                               & $\talpha_{i,j} = \tupepsilon_i - \tupepsilon_j$     & $1 \leq i \neq j \leq n$ \\ \hline
$\tb_{i,j} := 2^{-\delta_{i,j}} \left( E_{-i,j} +E_{-j, i} \right)$ & $\tbeta_{i,j} =  -(\tupepsilon_i + \tupepsilon_j) $ & $1 \leq i \leq j \leq n$ \\ \hline
$\tc_{i,j} := E_{i, -j} - E_{j, -i}   $                                                                                                                  & $\tgamma_{i,j} = \tupepsilon_i + \tupepsilon_j$     & $1 \leq i < j \leq n$    \\ \hline
\end{tabular}  .
\end{table}
\end{center} 
Here the elements  $\ta_{i,j}$ have even parity while $\tb_{i,j}$ and $\tc_{i,j}$ have odd parity.  For $1 \leq i \leq n$, write $\ta_{i,i} = E_{i,i} - E_{-i, -i}$.   For $i < j$, define $\tb_{j,i} = \tb_{j,i}$ and $\tc_{j,i} = -\tc_{i,j}$.  

\subsection{The Chevalley Isomorphism}\label{SS:Chevalleyisomorphism} 

The Chevalley automorphism $\tau : \gl(n|n) \to \gl(n|n)$ is given on the basis of matrix units by 
\[
\tau(E_{i,j}) = (-1)^{\p{j}(\p{i} + \p{j}) + 1} E_{j,i}.
\]
It is easily checked that the restriction of $\tau$ to $\fp(n) \subset \gl(n|n)$ induces an isomorphism $\fp(n) \rightarrow \tfp(n)$ which we call the \emph{Chevalley isomorphism}. In terms of the distinguished bases for $\fp(n)$ and $\tfp(n)$ the isomorphism is given by 
\begin{equation} \label{E:TauOnGenerators}
\tau(a_{i,j}) = -\ta_{j,i}, \quad \tau(b_{i,j}) = \tb_{i,j}, \quad \tau(c_{i,j}) = \tc_{i,j}.
\end{equation} 

Twisting the action on modules by the inverse of the Chevalley isomorphism induces an isomorphism of monoidal supercategories $\mathcal{T} = \mathcal{T}_n : \fp(n) \text{-mod} \to \tfp(n) \text{-mod}$.  Explicitly, if $M$ is a $\fp(n)$-module, then $\mathcal{T}(M)$ is the same underlying superspace viewed as a $\tfp(n)$-module via the action $X.m = \tau^{-1} (X)m$ for all $X \in \tfp(n)$ and $m\in M$. In particular, one has  $\mathcal{T}_n(V_n) \cong W_n^*$ and, more generally, $\mathcal{T}_{n}\left(S^{d}(V_{n}) \right) \cong S^{d}(W_{n})^{*}$ for all $d \geq 0$.

There is an anti-isomorphism of Lie superalgebras 
\[
- \tau : \fp(n) \to \tfp(n)
\] which induces an anti-isomorphism of associative superalgebras
\[
s : U(\fp(n)) \to U(\tfp(n)).
\]
Recall that being an anti-isomorphism of associative superalgebras means that $s$ is a linear isomorphism which satisfies $s(XY) = (-1)^{\p{X}\p{Y}} s(Y) s(X)$ for all homogeneous $X, Y \in U(\fp(n))$.
\subsection{A Presentation for \texorpdfstring{$U(\tfp (n))$}{U(tildep(n))}} \label{SS:Tildedpresentation}

Using the Chevalley isomorphism one can transport the presentation for $U(\fp (n))$ given earlier to obtain a presentation for $U(\tfp (n))$,  As we do not need it, we choose to omit it.  It can be found in the \texttt{ArXiv} version of the paper as explained in \cref{SS:ArxivVersion}.  See \cref{D:TildeUdot} for the categorical version.

\begin{answer}
\begin{definition}
Let $\tilde{\cR} = \tilde{\cR}_n$ denote the superalgebra with even generators $\tcE_{1},\dotsc ,\tcE_{n-1}$, $\tcF_{1}, \dots ,\tcF_{n-1}$, and $\tcH_1, \dotsc ,\tcH_{n-1}$, and odd generators $\tcB_{1,1}, \tcB_{1}, \dotsc , \tcB_{n-1}$, $\tcC_{1}, \dotsc, \tcC_{n-1}$.  These generators are subject to the following relations: 

\begin{enumerate}[label=({$\tilde{\text{R}}$}{\arabic*})]
\item  \label[Relation]{LL:Tildehcommuterelation}
\begin{equation}
\label{E:Tildehhcommutator} [\tcH_{i}, \tcH_{j}]=0;
\end{equation}
\item  \label[Relation]{LL:TildeRootVectors}
\begin{align}
\label{E:Tildehecommutator} [\tcH_{i}, \tcE_{j}] &= (\tupepsilon_{i}, \talpha_{j})\tcE_{j},\\
\label{E:Tildehfcommutator} [\tcH_{i}, \tcF_{j}] &= -(\tupepsilon_{i}, \talpha_{j})\tcF_{j};
\end{align}
\item  \label[Relation]{LL:Tildeefrelation}
\begin{equation}
\label{E:Tildeefcommutator} [\tcE_{i}, \tcF_{j}] = \delta_{i,j}(\tcH_{i}-\tcH_{i+1});
\end{equation}
\item  \label[Relation]{LL:TildeeSerrerelations}
\begin{align}
\label{E:Tildeeecommutator} [\tcE_{i},\tcE_{j}] &= 0 \text{ if $j \neq i \pm 1$}, \\
\label{E:TildeeSerreRelation} [\tcE_{i},[\tcE_{i}, \tcE_{j}]] &=0 \text{ if  $j = i \pm 1 $};
\end{align}
\item  \label[Relation]{LL:TildefSerrerelations}
\begin{align}
\label{E:Tildeffcommutator}  [\tcF_{i},\tcF_{j}] &= 0 \text{ if $j \neq i \pm 1$}, \\
\label{E:TildefSerreRelation} [\tcF_{i},[\tcF_{i}, \tcF_{j}]] &=0 \text{ if  $j = i \pm 1 $};
\end{align}
\item  \label[Relation]{LL:Tildebbandccrelations}
\begin{align}
\label{E:Tildebbcommutator} [\tcB_{i}, \tcB_{j}]&=0,\\
\label{E:Tildecccommutator} [\tcC_{i}, \tcC_{j}]&=0;
\end{align}
\item  \label[Relation]{LL:TildebcRootvectors}
\begin{align}
\label{E:Tildehbcommutator} [\tcH_{i}, \tcB_{j}] &= (\tupepsilon_{i}, \tbeta_{j})\tcB_{j},\\
\label{E:Tildehccommutator} [\tcH_{i}, \tcC_{j}] &= (\tupepsilon_{i}, \tgamma_{j})\tcC_{j}; 
\end{align}
\item  \label[Relation]{LL:Tildebcrelations}
\begin{equation*} [\tcB_{i}, \tcC_{j}] = \begin{cases} \tcH_{i}-\tcH_{i+1}, & j=i;\\
                                             [\tcE_{i-1}, \tcE_{i}], & j=i-1; \\
                                             [\tcF_{i}, \tcF_{i+1}], & j=i+1; \\
                                              0, & \text{else;}
\end{cases} 
\end{equation*}
\item  \label[Relation]{TildeLL:beandbfrelations}
\begin{align}
\label{E:Tildebecommutator} [\tcB_{i}, \tcE_{j}] &= 0 \text{ if $j \neq i, i+1$}, \\
\label{E:Tildebfcommutator} [\tcB_{i}, \tcF_{j}] &= 0 \text{ if $j \neq i, i-1$},\\
\label{E:Tildebecommutatorequalsbfcommutator1} [\tcB_{i}, \tcE_{i}] &= [\tcB_{i+1},\tcF_{i+1}], \\
\label{E:Tildebecommutatorequalsbfcommutator2} [\tcB_{i}, \tcE_{i+1}] &= [\tcB_{i+1},\tcF_{i}];
\end{align} 
\item  \label[Relation]{LL:Tildecefrelations}
\begin{align}
\label{E:Tildebccommutator} [\tcE_{j}, \tcC_{i}] &=0 \text{ if $j\neq i-1$}, \\
\label{E:Tildefccommutator} [\tcF_{j}, \tcC_{i}] &=0 \text{ if $j\neq i+1$}, \\
\label{E:Tildeeccommutator} [\tcE_{i}, \tcC_{i+1}] &= [\tcF_{i+1}, \tcC_{i}];
\end{align} 
\item   \label[Relation]{LL:TildebSerreRelations}
\begin{align}
\label{E:TildebeeSerreRelations1}[[\tcB_{i}, \tcE_{i}], \tcE_{j}] &= \begin{cases} 2 \tcB_{i+1}, & j=i+1; \\
                                         0, & \text{ else};
\end{cases}\\
\label{E:TildebefSerreRelations}[[\tcB_{1}, \tcE_{1}], \tcF_{1}] &=  2\tcB_{1},  \\
\label{E:TildebeeSerreRelations3} [[\tcB_{1}, \tcF_1], \tcF_j] & = 0,
                                         0, & j \neq i+1, i-1;
\end{align} 
\item  \label[Relation]{LL:TildecSerrerelations}
\begin{equation*}
[[\tcC_{i+1}, \tcE_{i}], \tcE_{j}] = \begin{cases} \tcC_{i}, & j=i+1; \\
                                           0, & j \neq i-1, i+1;
\end{cases}
\end{equation*}
\item\label[Relation]{LL:Tildeb11isacommutator}
\begin{equation} \label{E:TildeUb11ecommutator}
2\tcB_{1,1} =  [\tcB_1, \tcF_1];
\end{equation}

\item \label[Relation]{LL:Tildeb11relations}
\begin{align}
\label{E:TildeB11RootVector}
[\tcH_{i}, \tcB_{1,1}] &= (\tupepsilon_i, \tbeta_{1,1}) \tcB_{1,1},  \\
\label[Relation]{E:TildeB11B11Commutator}
\tcB_{1,1}^2 &= 0.
\end{align}
\end{enumerate}

\end{definition}

\begin{theorem}\label{T:TildePPresentation}
For every $n \geq 1$, there is an isomorphism of superalgebras $\tilde{\rho} : \widetilde{\cR}_n \to U(\tfp(n))$ given on even generators by:
\[
\tilde{\rho}(\tcE_i) = \ta_{i,i+1}, \quad \tilde{\rho}(\tcF_i) = \ta_{i+1,i}, \quad \tilde{\rho}(\tcH_i) = \ta_{i,i},
\]
and on odd generators by
\[
\tilde{\rho}({\tcB_{1,1}}) = \tb_{1,1}, \quad
\tilde{\rho}(\tcB_i) = \tb_{i,i+1}, \quad \tilde{\rho}(\tcC_i) = \tc_{i,i+1}.
\]  
\end{theorem}

\begin{proof}
 There is an algebra isomorphism $\psi: \tilde{\cR}_n \to \cR_n$ given on generators by 
\[
\psi(\tcE_i) = -\cF_i \quad \psi(\tcF_i) = -\cE_i, \quad \psi(\tcH_i)= -\cH_i, \quad \psi(\tcB_{1,1}) = \cB_{1,1}, \quad \psi(\tcB_i) = \cB_i, \quad \psi(\tcC_i) = \cC_i. 
\]
Write $\tau : U(\fp(n)) \to U(\tfp(n))$ for the algebra isomorphism induced by the Lie superalgebra isomorphism $\fp(n) \to \tfp(n)$ of the same name.  On the generators of $\tilde{\cR}_n$, the composition $ \tau \circ \rho \circ \psi : \tilde{\cR}_n \to U(\tfp(n))$ agrees with the map $\tilde{\rho}$ described in the lemma.  Hence, $\tilde{\rho}$ gives a well-defined algebra homomorphism, and makes the following diagram commute: 
\[\begin{tikzcd}
\cR_n  \arrow{d}{\rho}  &  \widetilde{\cR}_n \arrow{d}{\tilde{\rho}} \arrow{l}{\psi} \\
U(\fp(n)) \arrow{r}{\tau} & U(\tfp(n)),
\end{tikzcd}
\]
The fact that $\tilde{\rho}$ is an isomorphism follows from the fact that the other three maps are isomorphisms.
\end{proof}
On generators the map $s$ is given by:
$$
s(\cE_i) = \tcF_i, \quad s(\cF_i) = \tcE_i, \quad s(\cH_i) = \tcH_i, \quad s(\cB_{1,1}) = -\tcB_{1,1}, \quad s(\cB_i) = - \tcB_i, \quad s(\cC_i) = -\tcC_i.  
$$
\end{answer}

\subsection{The Category \texorpdfstring{$\bU (\tfp (n))$}{Udot(tildep(n))}} \label{Udottilde}
As in \cref{SS:Udot} a supercategory can be associated to $U(\tfp(n))$.  We refer to \cref{SS:Udot} for notational conventions.  

\begin{definition} \label{D:TildeUdot}
Let $\bU(\tfp(n))$ denote the $\k$-linear supercategory with objects $X(T_n)$.  For every $\tuplambda \in X(T_n)$ and $1 \leq i \leq n-1$, the generating morphisms are:
\begin{align*}
\tcE_i &: \tuplambda \to \tuplambda + \talpha_i \\
\tcF_i &: \tuplambda \to \tuplambda - \talpha_i \\
\tcB_{1,1} &: \tuplambda \to \tuplambda + \tbeta_{1,1} \\
\tcB_i &: \tuplambda \to \tuplambda + \tbeta_i \\
\tcC_i &: \tuplambda + \tuplambda + \tgamma_i 
\end{align*}
where the $\Z_2$-grading on the morphisms is such that $\tcE_1 ,...,\tcE_{n-1}$ and $\tcF_{1}, ... ,\tcF_{n-1}$ are declared to be even, and the generators $
\tcB_{1,1}, \tcB_{1},...,\tcB_{n-1}$, $\tcC_1,...,\tcC_{n-1}$ are declared to be odd.  The generating morphisms satisfy the following relations for all valid indices $i$ and $j$, and all $\tuplambda \in X(T_n)$.

\begin{enumerate}[label=($\widetilde{\text{{U}}}${\arabic*})]
\item  \label[Relation]{LL:UTildeefrelation}
\begin{equation}
\label{E:UTildeefcommutator} [\tcE_{i}, \tcF_{j}] 1_{\tuplambda} = \delta_{i,j}(\tupalpha_{i}, \tuplambda)1_{\tuplambda}; 
\end{equation}
\item  \label[Relation]{LL:UTildeeSerrerelations}
\begin{align}
\label{E:UTildeeecommutator} [\tcE_{i},\tcE_{j}] 1_{\tuplambda} &= 0 \text{ if $j \neq i \pm 1$}, \\
\label{E:UTildeeSerreRelation} [\tcE_{i},[\tcE_{i}, \tcE_{j}]] 1_{\tuplambda} &=0 \text{ if  $j = i \pm 1 $};
\end{align}
\item  \label[Relation]{LL:UTildefSerrerelations}
\begin{align}
\label{E:UTildeffcommutator}  [\tcF_{i},\tcF_{j}]1_{\tuplambda} &= 0 \text{ if $j \neq i \pm 1$}, \\
\label{E:UTildefSerreRelation} [\tcF_{i},[\tcF_{i}, \tcF_{j}]] 1_{\tuplambda} &=0 \text{ if  $j = i \pm 1 $};
\end{align}
\item  \label[Relation]{LL:UTildebbandccrelations}
\begin{align}
\label{E:UTildebbcommutator} [\tcB_{i}, \tcB_{j}]1_{\tuplambda}&=0,\\
\label{E:UTildecccommutator} [\tcC_{i}, \tcC_{j}]1_{\tuplambda}&=0;
\end{align}
\item  \label[Relation]{LL:UTildebcrelations}
\begin{equation*} [\tcB_{i}, \tcC_{j}]1_{\tuplambda} = \begin{cases} (\tupalpha_{i}, \tuplambda)1_{\tuplambda}, & j=i;\\ 
                                             [\tcE_{i-1}, \tcE_{i}]1_{\tuplambda}, & j=i-1; \\
                                             [\tcF_{i}, \tcF_{i+1}]1_{\tuplambda}, & j=i+1; \\
                                              0, & \text{else;}
\end{cases} 
\end{equation*}
\item  \label[Relation]{TildeLL:Ubeandbfrelations}
\begin{align}
\label{E:UTildebecommutator} [\tcB_{i}, \tcE_{j}] 1_{\tuplambda} &= 0 \text{ if $j \neq i, i+1$}, \\
\label{E:UTildebfcommutator} [\tcB_{i}, \tcF_{j}] 1_{\tuplambda} &= 0 \text{ if $j \neq i, i-1$},\\
\label{E:UTildebecommutatorequalsbfcommutator1} [\tcB_{i}, \tcE_{i}]1_{\tuplambda} &= [\tcB_{i+1},\tcF_{i+1}]1_{\tuplambda}, \\
\label{E:UTildebecommutatorequalsbfcommutator2} [\tcB_{i}, \tcE_{i+1}]1_{\tuplambda} &= [\tcB_{i+1},\tcF_{i}]1_{\tuplambda};
\end{align} 
\item  \label[Relation]{LL:UTildecefrelations}
\begin{align}
\label{E:UTildebccommutator} [\tcE_{j}, \tcC_{i}] 1_{\tuplambda}&=0 \text{ if $j\neq i-1$}, \\
\label{E:UTildefccommutator} [\tcF_{j}, \tcC_{i}]1_{\tuplambda} &=0 \text{ if $j\neq i+1$}, \\
\label{E:UTildeeccommutator} [\tcE_{i}, \tcC_{i+1}]1_{\tuplambda} &= [\tcF_{i+1}, \tcC_{i}]1_{\tuplambda} ;
\end{align} 
\item   \label[Relation]{LL:UTildebSerreRelations}
\begin{align}
\label{E:UTildebeeSerreRelations1}[[\tcB_{i}, \tcE_{i}], \tcE_{j}] 1_{\tuplambda}&= \begin{cases} 2 \tcB_{i+1}1_{\tuplambda} , & j=i+1; \\
                                         0, & \text{ else};
\end{cases} \\
\label{E:UTildebefSerreRelations}[[\tcB_{1}, \tcE_{1}], \tcF_{1}]1_{\tuplambda} &=  2\tcB_{1} 1_{\tuplambda} ,  \\
\label{E:UTildebeeSerreRelations3} [[\tcB_{1}, \tcF_1], \tcF_j] 1_{\tuplambda} & = 0, 
\end{align} 
\item  \label[Relation]{LL:UTildecSerrerelations}
\begin{equation*}
[[\tcC_{i+1}, \tcE_{i}], \tcE_{j}] 1_{\tuplambda} = \begin{cases} \tcC_{i} 1_{\tuplambda}, & j=i+1; \\
                                           0, & j \neq i-1, i+1;
\end{cases}
\end{equation*}

\item\label[Relation]{LL:UTildeb11isacommutator}
\begin{equation} \label{E:TildeUb11ecommutator}
2\tcB_{1,1}1_{\tuplambda} =  [\tcB_{1,1}, \tcF_1] 1_{\tuplambda};
\end{equation}
\item \label[Relation]{LL:UTildeb11relations}
\begin{align}
\label[Relation]{E:UTildeB11B11Commutator}
\tcB_{1,1}^2 1_{\tuplambda}&= 0
\end{align}
\end{enumerate}
\end{definition}

As in \cref{SS:Udot} there is an associated locally unital superalgebra $\dot{U}(\tfp (n))$ and a three way dictionary between representations of $\bU (\tfp (n))$, modules for $\dot{U}(\tfp (n))$, and weight modules for $U(\tfp (n))$ whose weights lie in $X(T_{n})$.

The Chevalley isomorphism $\tau : U(\fp(n)) \to U(\tfp(n))$ has covariant functor analogue,
\[
T : \bU(\fp(n)) \to \bU(\tfp(n)).
\]
It is given on objects by $T(\tuplambda) = -\tuplambda$ for all $\tuplambda \in X(T_n)$.  On morphisms $T$ is determined by
$$
T(\cE_i 1_{\tuplambda}) = -\tcF_i1_{-\tuplambda}, \; T(\cF_i 1_{\tuplambda}) = -\tcE_i 1_{-\tuplambda}, \; T(\cB_{1,1}1_{\tuplambda}) = \tcB_{1,1}1_{-\tuplambda}, \; T(\cB_i) =  \tcB_i1_{-\tuplambda}, \; T(\cC_i1_{\tuplambda}) = \tcC_i 1_{-\tuplambda}. 
$$

The anti-isomorphism $s : U(\fp(n)) \to U(\tfp(n))$ has contravariant functor analogue:
\[
\Psi : \bU(\fp(n)) \to \bU(\tfp(n)).
\]
It is given on objects by $\Psi(\tuplambda) = \tuplambda$ for all $\tuplambda \in X(T_n)$.  On morphisms $\Psi$ is determined by
$$
\Psi(\cE_i 1_{\tuplambda}) = \tcF_i1_{\tuplambda}, \;\; \Psi(\cF_i 1_{\tuplambda}) = \tcE_i 1_{\tuplambda}, \;\; \Psi(\cB_{1,1}1_{\tuplambda}) = -\tcB_{1,1}1_{\tuplambda}, \;\; \Psi(\cB_i) = - \tcB_i1_{\tuplambda}, \;\; \Psi(\cC_i1_{\tuplambda}) = -\tcC_i 1_{\tuplambda}. 
$$
We remark that $\Psi$ is contravariant in the graded sense: $\Psi(X \circ Y) = (-1)^{|X||Y|} \Psi(Y) \circ \Psi(X)$ for all composable homogeneous morphisms $X$ and $Y$.

\section{Representations of \texorpdfstring{$\fp (n)$}{p(n)} and Webs}\label{S:p(n)repsandwebs}

\subsection{Module Categories}

Define $\pmodSS$ to be the monoidal supercategory of $\fp(n)$-modules generated by the modules $S^{d}(V_n)$ and $S^{d}(V_n)^*$, for $d \geq 0$, where $S^{0}(V_{n}) = \k$ is the trivial module by definition.  That is, the objects of $\pmodSS$ are the $\fp(n)$-modules obtained by taking a finite tensor product of supersymmetric powers of $V_n$ and their duals (which are isomorphic to exterior powers).  The morphisms in $\pmodSS$ are all $\fp(n)$-module homomorphisms (i.e., we do not assume morphisms preserve the $\Z_{2}$-grading).  That is, if $M, N$ are $\fp (n)$-modules, then a homogeneous linear map $f: M \to N$ is a module homomorphism if $f(x.m) = (-1)^{\p{f}\p{x}}x.f(m)$ for all homogeneous $x \in \fp (n)$ and $m \in M$.  A general module homomorphism is the sum of such maps.  Let $\pmodS$ and $\pmodSD$ denote the full monoidal subcategories of $\pmodSS$ generated by the modules of the form $S^{d}(V_n)$ and $S^{d}(V_n)^*$, for $d\geq 0$, respectively.  Finally, for $m \geq 0$ write $\pmodSSm$ for the full subcategory of $\pmodSS$ whose objects are tensor products of exactly $m$ modules of the form $S^{d}(V_n)$ and $S^{d}(V_n)^*$ for $d \geq 0$. Since the trivial module is allowed this is the full subcategory of all modules which are isomorphic to a tensor product of $m$ or less symmetric powers of $V_n$ and their duals.   Define $\pmodSm$ and $\pmodSDm$ similarly.

There are obvious analogues of the above categories for $\tfp (n)$.  The only difference is that we write $W_n$ for its natural module.

\subsection{Webs for \texorpdfstring{$\fp (n)$}{p(n)}}\label{SS:websforP}

In \cite{DKM} the authors introduced a monoidal category,  $\pWeb_{\uparrow \downarrow}$, of oriented webs of type $P$.  It is diagrammatically defined as a strict monoidal supercategory by generators and relations.  The generating objects in $\pWeb_{\uparrow \downarrow}$ are the symbols $\uparrow_{a}$ and  $\downarrow_{a }$, where $a  \in \Z_{\geq 0}$.   Horizontal concatenation is the monoidal (tensor) product, a general object in $\pWeb_{\uparrow \downarrow}$  is a word in these symbols, and $\uparrow_{0}=\downarrow_{0}$ is the tensor identity.   The morphisms in $\pWeb_{\uparrow \downarrow}$ are given by generators and relations with the $\Z_{2}$-grading on the morphism spaces induced by the declared parity of the generating morphisms.  See \cite{DKM} for details.

The following result is a restatement of  \cite[Theorems 6.5.1, 6.8.2]{DKM}.

\begin{theorem} \label{T:WebsTheorem} For every $n \geq 1$ there is a full, essentially surjective, monoidal functor
\[
G_{\uparrow \downarrow} : \pWeb_{\uparrow \downarrow} \to \pmodSS,
\]
given on generating objects by $G_{\uparrow \downarrow}(\uparrow_{a }) = S^{a }(V_n)$ and $G_{\uparrow \downarrow}(\downarrow_{a }) = S^{a }(V_n)^*$ for every $a  \geq 0$.
\end{theorem}

 Let $\pWeb_{\uparrow}$ and $\pWeb_{\downarrow}$ denote the full subcategories of $\pWeb_{\uparrow \downarrow}$ whose objects are words in the symbols $\left\{ \uparrow_{a } \mid a  \geq 0 \right\}$ and $\left\{ \downarrow_{a } \mid a  \geq 0 \right\}$, respectively. 

For $m \geq 0$, let $\pWeb_{\uparrow, m}$ denote the full subcategory of $\pWeb_{\uparrow}$ whose objects are words of length $m$ in the symbols $\left\{ \uparrow_{a} \mid a \geq 0 \right\}$.
Define $\pWeb_{\downarrow, m}$ as a full subcategory of $\pWeb_{\downarrow}$ similarly.  An arbitrary object in $\pWeb_{\uparrow,m}$ can be written in the form $\uparrow_{\bba} := \uparrow_{a_1} \cdots \uparrow_{a_m}$, where $\bba = (a_{1}, \dotsc , a_{m}) \in \Z_{\geq 0}^m$.  The set $X(T_m)_{\geq 0}$ also indexes the objects of $\pWeb_{\uparrow,m}$  via the map
\[
\tuplambda =  \sum_{i = 1}^m \lambda_i \tupepsilon_i  \mapsto \uparrow_{\tuplambda} = \uparrow_{\lambda_{1}} \dotsb \uparrow_{\lambda_{m}}.
\]  Similarly, $X(T_m)_{\geq 0}$ can be used to index  the objects of $\pWeb_{\downarrow,m}$.

Restricting the functor $G_{\uparrow \downarrow}$ yields full, essentially surjective functors
\begin{align*}
G_{\uparrow,m} &: \pWeb_{\uparrow,m} \to \pmodSm, \\
G_{\downarrow, m} &: \pWeb_{\downarrow, m} \to \fp(m)\text{-mod}_{\mathcal{S}^{*},m},
\end{align*}  for every $m \geq 1$.

Recall from \cref{SS:Chevalleyisomorphism} the equivalence of monoidal categories $\mathcal{T} : \fp(n)\text{-mod} \to \tfp(n)\text{-mod}$ which satisfies $\mathcal{T}(S^d(V_n)) \cong S^d(W_n)^*$  and $\mathcal{T}(S^d(V_n)^*) \cong S^d(W_n)$ for all $d \geq 0$. By restriction $\mathcal{T}$ induces an equivalence
\[
\mathcal{T} : \fp(n)\text{-mod}_{\mathcal{S^*},m} \to  \tfp(n)\text{-mod}_{\mathcal{S},m}.
\] Moreover, the composition
\[
\mathcal{T} \circ G_{\downarrow,m} : \pWeb_{\downarrow,m} \to \tfp(n) \text{-mod}_{\mathcal{S},m}
\] is full and essentially surjective.

As noted in \cite[Section 4.4]{DKM}, there is a contravariant ``reflection'' isomorphism
\[
\operatorname{refl} : \pWeb_{\uparrow} \overset{\sim}{\to} \pWeb_{\downarrow}
\]
given on generating objects by $\uparrow_{a} \mapsto \downarrow_{a}$ and, up to a sign, on morphisms by reflecting diagrams across a horizontal axis.  The functor $\operatorname{refl}$ also restricts to a contravariant isomorphism 
\[
\operatorname{refl} : \pWeb_{\uparrow, m} \overset{\sim}{\to} \pWeb_{\downarrow, m}.
\]

\section{Categorical and Finitary Howe Dualities}

\subsection{The Functor   \texorpdfstring{$\bU(\tfp (m)) \to \pmodSm$}{dotU(tildep(n)) to p(n)-modSm}}\label{SS:CategoricalHoweDuality} Let $\bU(\tfp(m))_{\geq 0}$ denote the quotient category obtained from $\bU(\tfp(m))$ by setting the identity morphism $1_{\tuplambda}$ to zero whenever $\tuplambda \notin X(T_m)_{\geq 0}$.   As we next explain, this category is related to $\pWeb_{\uparrow, m}$.

In \cite[Section 3.6, Corollary 3.6.4]{DKM} certain morphisms in $\pWeb_{\uparrow,m}$ are introduced and shown to generate it as a $\k$-linear category; that is, using $\k$-linear combinations and compositions but not the monoidal structure.  For $1 \leq i \leq m-1$, and $\tuplambda \in X(T_m)_{\geq 0}$, these generating morphisms are:
  \begin{align*}
  e_{[i,i+1], \tuplambda}^{(1)} :& \uparrow_\tuplambda \to \uparrow_{\tuplambda + \talpha_i}, \\
  f_{[i,i+1], \tuplambda}^{(1)} :& \uparrow_\tuplambda \to \uparrow_{\tuplambda - \talpha_i}, \\
  b_{[1], \tuplambda}^{(1)} :& \uparrow_\tuplambda \to \uparrow_{\tuplambda - \tbeta_{1,1}}, \\
  b_{[i,i+1], \tuplambda} :& \uparrow_\tuplambda \to \uparrow_{\tuplambda + \tbeta_i}, \\ 
  c_{[i,i+1], \tuplambda} :& \uparrow_\tuplambda \to \uparrow_{\tuplambda + \tgamma_{i}}.
  \end{align*} 
 \begin{lemma}\label{L:FunctorToWebs}
For every $m \geq 1$ there is an essentially surjective, full functor
\begin{align*}
H_m:\dot{\mathbf{U}}(\tilde{\mathfrak{p}}(m))_{\geq 0} \to \pWeb_{\uparrow, m}
\end{align*}
given on objects by 
\[
\tuplambda \mapsto \uparrow_{\tuplambda} = \uparrow_{\lambda_{1}}\dotsb \uparrow_{\lambda_{m}}
\]
for all $\tuplambda = \sum_{i=1}^{m} \lambda_{i}\tupepsilon_{i} \in X(T_m)_{\geq 0}$.   On generating morphisms the functor is given by
\begin{align*}
\tcE_i 1_{\tuplambda}  &\mapsto e_{[i,i+1], \tuplambda}^{(1)}, \\
\tcF_i  1_{\tuplambda} &\mapsto f_{[i,i+1],\tuplambda} ^{(1)}, \\
\tcB_{1,1} 1_{\tuplambda} &\mapsto b_{[1], \tuplambda}, \\
\tcB_i 1_{\tuplambda} &\mapsto b_{[i,i+1], \tuplambda}, \\
\tcC_i 1_{\tuplambda} &\mapsto c_{[i,i+1], \tuplambda}.
\end{align*}
\end{lemma}
\begin{proof} The fact that $H_m$ is well-defined follows by comparing the relations among the generators for $\bU(\tfp(m))$ given in \cref{D:TildeUdot} with the relations among the  morphisms in $\pWeb_{\uparrow,m}$ calculated in \cite[Lemma 3.7.1]{DKM}.  The fullness of $H_m$ follows from the fact the generating morphisms of $\pWeb_{\uparrow, m} $ are in the image of $H_m$.  The fact the functor is essentially surjective is clear.
\end{proof}

Let $\operatorname{pr}:\bU(\tfp(m)) \to \bU(\tfp(m))_{\geq 0}$ be the canonical quotient functor.  The next theorem follows immediately from the previous results.

\begin{theorem} \label{T:CategoricalHoweDualityI}  For every $m, n \geq 1$, the composition 
\[
G_{\uparrow,m} \circ H_{m} \circ \operatorname{pr} : \bU(\tfp(m)) \to \pmodSm
\]
is an essentially surjective, full functor.  
\end{theorem}

\begin{remark}  To obtain Theorem~D in the introduction from this result one simply needs to precompose this functor with the Chevalley functor $T: \bU (\fp (m)) \to \bU (\tfp (m))$ given at the end of \cref{Udottilde}.
\end{remark}

\begin{remark} In \cite{ST}, analogues of the functor from \cref{T:CategoricalHoweDualityI} are constructed in the quantum type BCD setting.  The authors of that paper refer to these as \textit{Howe functors.}
\end{remark}

Let $\dot{U}(\fp (n))$ and $\dot{U}(\tfp (m))$ be the locally unital superalgebras defined by the supercategories $\bUdot (\fp (n))$ and $\bUdot (\tfp(m))$, respectively (see \cref{SS:Udot}). Let 
\begin{equation*}
\mathcal{A} = \mathcal{A}_{m,n}  = S(V_n)^{\otimes m}.
\end{equation*} There is an action of $\fp (n)$  on $\mathcal{A}$ coming from the natural action on $V_{n}$.  In particular, $\mathcal{A}$ is a weight module for $U(\fp (n))$ with weights lying in the weight lattice $X(T_{n})$ for $\fp(n)$ so is a module for $\dot{U}(\fp (n))$.

On the other hand, for any $\tuplambda = \sum_{i = 1}^m \lambda_i \tupepsilon_i \in X(T_m)$ set $\mathcal{A}^{\tuplambda} = G_{\uparrow,m} \circ H_{m} \circ \operatorname{pr}(\tuplambda)$. Then, 
$$
\mathcal{A}^{\tuplambda} = \begin{cases} S^{\tuplambda}(V_n) = S^{\lambda_1}(V_n) \otimes \cdots \otimes S^{\lambda_m}(V_n), 
& \tuplambda = \sum_{i = 1}^n \lambda_i \tupepsilon_i \in X(T_m)_{\geq 0}; \\ 0, & \text{else}. 
\end{cases}
$$ We may identify $\mathcal{A}$ with the direct sum
\begin{equation}\label{E:Adecomp}
\mathcal{A} = \bigoplus_{\tuplambda \in X(T_{m})_{\geq 0}} S^{\lambda_{1}}(V_{n}) \otimes \dotsb \otimes S^{\lambda_{m}}(V_{n}) = \bigoplus_{\tuplambda \in X(T_{m})} \mathcal{A}^{\tuplambda}.
\end{equation}
From this it follows that $\mathcal{A}$ is also a module for $\dot{U}(\tfp(m))$  by the discussion at the end of \cref{SS:Udot}, where $\mathcal{A}^{\tuplambda}$ is the $\tuplambda$-weight space.  Explicitly, for each $\tuplambda, \tupmu \in X(T_m)$, the action of any element $x \in 1_{\tupmu} \dot{U}(\tfp(m) 1_{\tuplambda} = \Hom_{\bU(\tfp(m))}(\tuplambda, \tupmu)$ on $\mathcal{A}$ is given by the composite
$$
\mathcal{A} \to \mathcal{A}^{\tuplambda} \xrightarrow{\left(G_{\uparrow,m} \circ H_m \circ \operatorname{pr} \right)(x)}\mathcal{A}^{\tupmu} \to \mathcal{A},
$$
where the first arrow is the projection of $\mathcal{A}$ onto $\mathcal{A}^{\tuplambda}$ 
and the last arrow is the inclusion of $\mathcal{A}^{\tupmu}$ into $\mathcal{A}$.  

Recall from the introduction that for a superspace $V$ we write $\End^{\fin}_{\k}(V)$ for the superalgebra of linear maps $f : V \to V $ whose image is finite-dimensional. 

\begin{theorem}\label{T:FinitaryHoweDualityI} For every $m, n \geq 1$ there are commuting actions of $\dot{U}(\tfp(m))$ and $\dot{U}(\fp(n))$ on 
\[
\mathcal{A}_{m,n}=S(V_n)^{\otimes m},
\]
where $\dot{U}(\fp(n))$ has the natural action.

Furthermore, these actions induce superalgebra homomorphisms
$$
\begin{tikzcd}[row sep = small, column sep = small]
\dot{U}(\tfp(m)) \arrow{dr}{\tilde{\varrho}_{ }} &  & \dot{U}(\fp(n)) \arrow{dl}[swap]{\varrho_{ }} \\
 & \End^{\fin}_{\k}(\mathcal{A}_{m,n})  & 
\end{tikzcd}
$$ and the image of $\tilde{\varrho}_{ }$ is the full centralizer of the image of $\varrho_{ }$.  

\end{theorem}

\begin{proof} 
First we investigate the $\dot{U}(\fp(n))$ action on $\mathcal{A} = \mathcal{A}_{m,n}$.  The action of  $\dot{U}(\fp(n))$ on $\mathcal{A}$ induces a superalgebra map $\varrho_{ } : \dot{U}(\fp(n)) \to \End_{\k}\left( \mathcal{A}\right)$, where for each $\tupnu \in X(T_n)$, the linear map $\varrho_{ }(1_{\tupnu})$ is the projection of $\mathcal{A}$ onto the $\tupnu$-weight space $\mathcal{A}_{\tupnu}$.  Since $\dot{U}(\fp(n)) = \bigoplus_{\tupnu \in X(T_n)} \dot{U}(\fp(n)) 1_{\tupnu}$, the action of an element of $\dot{U}(\fp (n))$ is zero on all but finitely many weight spaces of $\mathcal{A}$.  Since each $\fp(n)$-weight space of $\mathcal{A}$ is finite-dimensional the image of $\varrho_{ }$ lies in $\End^\fin_{\k}(\mathcal{A})$.  

Next, the action of $\dot{U}(\tfp(m))$ on $\mathcal{A}$ induces a superalgebra homomorphism $\tilde{\varrho_{}} : \dot{U}(\tfp(m)) \to \End_{\k}(\mathcal{A})$.  Because the weight spaces $\mathcal{A}^{\tuplambda}$ are finite-dimensional, similar arguments to those used in the previous paragraph imply that the image of $\tilde{\varrho_{}}$ is also contained in $\End^\fin_{\k}(\mathcal{A})$.

We now show that the image of $\tilde{\varrho}_{ }$ is precisely the centralizer of the image of $\varrho_{ }$.  First, from the discussion before the statement of the theorem every $x \in \dot U(\tfp(m))$ acts on $\mathcal{A}$ via a $\fp(n)$-module homomorphism.  That is, the image of $\tilde{\varrho}$ lies in the centralizer of the image of $\varrho$.  

Next, assume that $f \in \End_{\k}^{\fin}\left( \mathcal{A}\right)$ centralizes the image of $\varrho_{ }$ and so the map $f$ is a $\fp(n)$-module endomorphism of $\mathcal{A} =\bigoplus_{\tuplambda\in X(T_m)_{\geq 0}} S^{\tuplambda}(V_n)$.  Furthermore,  the assumption that the image of $f$ is finite-dimensional implies one can choose finitely many weights $\tupmu_1,...,\tupmu_{\ell}$ so that the image of $f$ is contained in $\bigoplus_{i = 1}^{\ell} S^{\tupmu_i}(V_n)$ and finitely many weights $\tuplambda_{1},...,\tuplambda_k \in X(T_m)_{\geq 0}$ so that if $f\left( S^{\tuplambda}(V_n)\right) \neq 0$ then $\tuplambda = \tuplambda_i$ for some $1 \leq i \leq k$.   That is, $f$ restricts to an element of
\begin{align*}
\Hom_{\fp(n)}\left( \bigoplus_{i= 1}^{k} S^{\tuplambda_i}(V_n), \bigoplus_{j = 1}^{\ell} S^{\tupmu_j}(V_n)\right) &= \bigoplus_{i= 1}^{k} \bigoplus_{j = 1}^{\ell} \Hom_{\fp(n)}( S^{\tuplambda_i}(V_n),  S^{\tupmu_j}(V_n)) \\  &\subset \bigoplus_{\tuplambda, \tupmu \in X(T_m)_{\geq 0}}  \Hom_{\fp(n)}( S^{\tuplambda}(V_n),  S^{\tupmu}(V_n)).
\end{align*} 
By extending linearly, the full functor $G_{\uparrow,m} \circ H_m \circ \operatorname{pr}$ induces a surjective map:
$$
\dot{U}(\tfp(m)) = \bigoplus_{\tuplambda, \tupmu \in X(T_m) } \Hom_{\bU(\tfp(m))}(\tuplambda, \tupmu) \to \bigoplus_{\tuplambda, \tupmu \in X(T_m)_{\geq 0} } \Hom_{\fp(n)}(S^{\tuplambda}(V_n), S^{\tupmu}(V_n)).
$$
It follows that $f$ is in the image of this map.  Furthermore, when the codomain of this map is viewed as a subspace $\End_{\k}(\mathcal{A})$ the map above is precisely the superalgebra homomorphism $\tilde{\varrho}$,  hence $f$ is in the image of $\tilde{\varrho}$.  The claim follows.
\end{proof}

\subsection{The \texorpdfstring{$\dot{U}(\tfp (m))$}{dotU(tildep(n))} Action on \texorpdfstring{$\mathcal{A}$}{A}}\label{SS:UpnActiononA}  While the action of $\dot{U}(\fp (n))$ on $\mathcal{A}$ is the natural one, the action of $\dot{U}(\tfp (m))$ is less obvious.  We next give explicit formulas for this action.

Recall our fixed homogeneous basis $\{v_{i} \mid i \in I_{n|n} \}$ for $V_{n}$.  Then $S(V_n)$ has a basis given by all ordered monomials of the form 
\[
\left\{ \prod_{\ell = 1}^n v_\ell^{d_\ell} v_{-\ell}^{\varepsilon_\ell} \mid  d_{\ell} \in \Z_{\geq 0} \text{ and }\varepsilon_{\ell} \in \{0, 1\} \right\}.
\]   This basis along with the identification \cref{E:Adecomp} shows $\mathcal{A}$ has a homogeneous basis given by the set
\begin{equation*}
\left\{ \bigotimes_{k = 1}^{m} \prod_{\ell = 1}^n v_{\ell} ^{d_{k, \ell}} v_{-\ell}^{\varepsilon_{k,\ell}}  \mid  d_{k,\ell} \in \Z_{\geq 0} \text{ and }\varepsilon_{k,\ell} \in \{0,1\}       \right\}.
\end{equation*}

Let $\preceq$ denote the lexicographic order on  $I_{m,n}:=\{1,...,m\} \times \{1,...,n\}$. For $\varepsilon = (\varepsilon_{k, \ell})_{(k,\ell) \in I_{m,n}}$  with $\varepsilon_{k, \ell} \in \Z_{2}$, and $(i,t) \in I_{m,n}$, let
\begin{align*}
y_i(\varepsilon, t) &= \sum_{(k,\ell) \prec (i, t)} \varepsilon_{k,\ell}. \end{align*}

For any for $\cY \in \{ \cE_{i}  , \cF_{i}, \cB_{1,1}, \cB_{i}, \cC_{i} \}$ and any $\tuplambda \in X(T_m)$, the generator $\tcY 1_{\tuplambda}$ of $\dot{U}(\tfp(m))$ determines a linear endomorphism of the space $\mathcal{A} = \bigoplus_{\tuplambda \in X(T_m)} S^{\tuplambda}(V_n)$.  Because each $\tcY 1_{\tuplambda}$ acts as zero on $S^{\tupmu}(V_n)$ whenever $\tupmu \neq \tuplambda$, the infinite sum $\tcY = \sum_{\tuplambda \in X(T_m)} \tcY 1_{\tuplambda}$ determines a well-defined linear endomorphism of $\mathcal{A}$.   We describe the action of this operator in the following proposition.  By precomposing $\tcY$ with the distinguished idempotents of $\dot{U}(\tfp (m))$ one can recover the action of the corresponding generator of $\dot{U}(\tfp(m))$.

\begin{proposition}\label{P:pmtildeactiononA} For $\cY \in \{\cE_i, \cF_i, \cB_{1,1}, \cB_i, \cC_i\}$ the linear endomorphism $\tcY$ of $\mathcal{A}$ is given by:
\begin{align*}
\tcE_{i}  . \bigotimes_{k = 1}^{m} \prod_{\ell = 1}^n v_{\ell} ^{d_{k, \ell}} v_{-\ell}^{\varepsilon_{k,\ell}} &= 
\sum_{t = 1}^n d_{i+1, t} \bigotimes_{k = 1}^{m}  \prod_{\ell = 1}^n v_{\ell} ^{d_{k, \ell} + \delta_{t, \ell}(\delta_{k,i} - \delta_{k, i+1})} v_{-\ell}^{\varepsilon_{k,\ell}} \\
&\hspace{0.5in} +  \sum_{t = 1}^n \varepsilon_{i+1, t} (-1)^{y_{i+1}(\varepsilon, t) - {y_{i}(\varepsilon, t)}} \bigotimes_{k= 1}^{m} \prod_{\ell = 1}^n v_{\ell} ^{d_{k, \ell}} v_{-\ell}^{\varepsilon_{k, \ell} + \delta_{t, \ell}(\delta_{k,i} - \delta_{k, i+1} )}, \\
\tcF_{i} . \bigotimes_{k = 1}^{m} \prod_{\ell = 1}^n v_{\ell} ^{d_{k, \ell}} v_{-\ell}^{\varepsilon_{k,\ell}} &= 
\sum_{t = 1}^n d_{i, t} \bigotimes_{k = 1}^{m}  \prod_{\ell = 1}^n v_{\ell} ^{d_{k, \ell} + \delta_{t, \ell}(\delta_{k,i+1} - \delta_{k, i})} v_{-\ell}^{\varepsilon_{k,\ell}} \\
&\hspace{0.5in} +  \sum_{t = 1}^n \varepsilon_{i, t} (-1)^{y_{i+1}(\varepsilon, t) - {y_{i}(\varepsilon, t)}} \bigotimes_{k= 1}^{m} \prod_{\ell = 1}^n v_{\ell} ^{d_{k, \ell}} v_{-\ell}^{\varepsilon_{k, \ell} + \delta_{t, \ell}(\delta_{k,i+1} - \delta_{k, i} )}, \\
\tcB_{1,1} .\bigotimes_{k = 1}^{m} \prod_{\ell = 1}^n v_{\ell} ^{d_{k, \ell}} v_{-\ell}^{\varepsilon_{k,\ell}} &= \sum_{t = 1}^n
(-1)^{y_1(\varepsilon,t)} d_{1,t} \varepsilon_{1,t}  \bigotimes_{k = 1}^m \prod_{\ell = 1}^n v_\ell^{d_{k,\ell} - \delta_{k, 1}\delta_{\ell, t}}v_{-\ell}^{\varepsilon_{k,\ell} - \delta_{k,1}\delta_{\ell,t}}, \\
\tcB_i. \bigotimes_{k = 1}^{m} \prod_{\ell = 1}^n v_{\ell} ^{d_{k, \ell}} v_{-\ell}^{\varepsilon_{k,\ell}} &=
\sum_{t=1}^n 
(-1)^{y_i(\varepsilon, t)   + 1}
d_{i+1, t}\varepsilon_{i,t}
\bigotimes_{k=1}^n \prod_{\ell=1}^n
v_\ell^{d_{k\ell} - \delta_{t, \ell}\delta_{k,i+1}} v_{-\ell}^{\varepsilon_{k \ell} - \delta_{t,\ell}\delta_{k,i}},\\
&\hspace{0.5in} +
\sum_{t=1}^n
(-1)^{y_{i+1}{(\varepsilon_i, t) }}
d_{i, t}\varepsilon_{i+1,t}
\bigotimes_{k=1}^n
\prod_{\ell=1}^n v_\ell^{d_{k\ell} - \delta_{t, \ell}\delta_{k,i}} v_{-\ell}^{\varepsilon_{k \ell} - \delta_{t,\ell}\delta_{k,i+1}}, \\
\tcC_i . \bigotimes_{k = 1}^{m} \prod_{\ell = 1}^n v_{\ell} ^{d_{k, \ell}} v_{-\ell}^{\varepsilon_{k,\ell}} &=
\sum_{t = 1}^n
(1-\varepsilon_{i+1,t})
(-1)^{y_{i+1}(\varepsilon, t)}
\bigotimes_{k = 1}^m \prod_{\ell =1}^n
v_{\ell}^{d_{k\ell} + \delta_{t,\ell}\delta_{k, i}} v_{-\ell}^{\varepsilon_{k \ell} + \delta_{t \ell} \delta_{k, i+1}}, \\
&\hspace{0.5in} +
\sum_{t = 1}^n
(1-\varepsilon_{i,t})
(-1)^{y_i(\varepsilon,t) +1}
\bigotimes_{k= 1}^m \prod_{\ell =1}^n
v_{\ell}^{d_{k\ell} + \delta_{t,\ell}\delta_{k, i+1}} v_{-\ell}^{\varepsilon_{k \ell} + \delta_{t \ell} \delta_{k, i}}.
\end{align*} 

For each $\tuplambda = \sum_{i=1}^{m}\lambda_{i}\tupepsilon_{i} \in X(T_{m})$ the weight idempotent $1_{\tuplambda}$ acts on $\mathcal{A}$ by projecting onto the summand $S^{\tuplambda}\left(V_{n} \right).$
\end{proposition}
\begin{proof}
For $1 \leq j \leq m$, set $\lambda_{j} = \sum_{\ell} d_{j,\ell} + \tupepsilon_{j, \ell}$ and set $\tuplambda = \sum_{j=1}^{m}\lambda_{j}\tupepsilon_{j}$.  Then, 
\[
 \bigotimes_{k = 1}^{m} \prod_{\ell = 1}^n v_{\ell} ^{d_{k, \ell}} v_{-\ell}^{\varepsilon_{k,\ell}} \in S^{\tuplambda}(V_n)
\]
and $\tcY 1_{\tupmu}$ acts on this vector as zero unless $\tupmu = \tuplambda$.  Hence, to describe the action of $\tcY $ on this vector it suffices to describe the action of $\tcY 1_{\tuplambda}$.   We do this for $\tcE_{i}$ and $\tcB_i$, and leave the others to the reader.

Translating through the definitions, for every $\tuplambda \in X(T_m)_{\geq 0}$ and $1 \leq i \leq n-1$, the element $\tcE_i 1_{\tuplambda} $ (resp.\  $\tcB_i1_{\tuplambda}$) of $\dot U(\fp(m))$ will act on a vector in $S^{\tuplambda}(V_n)$ as the morphism $G_{\uparrow,m}(e_{[i,i+1],\tuplambda}^{(1)}) : S^{\tuplambda}(V_n) \to S^{\tuplambda + \tupalpha_i}(V_n)$ (resp.\  $G_{\uparrow}( b_{[i,i+1],\tuplambda}) : S^{\tuplambda}(V_n) \to  S^{\tuplambda + \tgamma_{i}}(V_n)$).   

We now explicitly describe these morphisms.   Following \cite[Section 5.2]{DKM} there is a product $\nabla : S(V_n) \otimes S(V_n) \to S(V_n)$ and coproduct  $\Delta : S(V_n) \to S(V_n) \otimes S(V_n)$ making $S(V_n)$ into a bialgebra in the graded sense.  These maps are given on generators by
$$
\nabla(v_i \otimes v_j) = v_{i} v_{j}, \quad 
\Delta(v_{i}) = v_{i} \otimes 1 + 1 \otimes v_i,
$$
for all $i,j \in I_{n|n}$.  Composing the coproduct with the linear projection $S(V_n) \to V_n$ applied to the first tensor factor, we obtain the \textit{split map}:
\begin{gather*}
\spl : S(V_n) \to V_n \otimes S(V_n),\\
 \prod_{\ell = 1}^n v_{\ell} ^{d_{\ell}} v_{-\ell}^{\varepsilon_{\ell}} \mapsto  
 \sum_{t = 1}^n \left( d_{t} v_{t} \otimes \prod_{\ell = 1}^n v_{\ell} ^{d_{\ell} - \delta_{t, \ell}}v_{-\ell}^{\varepsilon_{\ell}} \right) 
 + \sum_{t = 1}^n (-1)^{\chi(\varepsilon,t)} \left( \varepsilon_{t} v_{-t} \otimes \prod_{\ell = 1}^n v_{\ell} ^{d_{\ell}}v_{-\ell}^{\varepsilon_{\ell} - \delta_{\ell, t}} \right),
\end{gather*}
where $\varepsilon = (\varepsilon_1,...,\varepsilon_n)$ and $\chi(\varepsilon,t) = \sum_{k = 1}^{t-1} \varepsilon_k$.   This sign appears when simplifying the terms of $\Delta\left(\prod_{\ell = 1}^n v_{\ell} ^{d_{\ell}}v_{-\ell}^{\varepsilon_{\ell}} \right)$ which have the form:
$$
\left(\prod_{\ell = 1}^{t-1} (1 \otimes v_{\ell})^{d_\ell} \cdot (1 \otimes v_{-\ell})^{\varepsilon_\ell} \right) \cdot (1 \otimes v_{t}^{d_t} ) \cdot (v_{-t} \otimes 1) \cdot \left( \prod_{\ell = t+1}^n (1 \otimes v_{\ell})^{d_\ell} \cdot (1 \otimes v_{-\ell})^{\varepsilon_\ell} \right).
$$
Note that such terms will only occur when $\varepsilon_t = 1$, which is why we include the factor of $\varepsilon_t v_{-t}$ in our formula for $\spl$.  

By precomposing the product $\nabla$ with the inclusion of $V_n$ into $S(V_n)$ in the second tensor factor, we obtain the merge map:
\begin{gather*}
\mer : S(V_n) \otimes V_n \to S(V_n),\\
\left(\prod_{\ell = 1}^n v_{\ell} ^{d_{\ell}} v_{-\ell}^{\varepsilon_{\ell}}\right) \otimes v_{t} \mapsto 
\prod_{\ell = 1}^n v_{\ell} ^{d_{\ell} + \delta_{\ell,t}} v_{-\ell}^{\varepsilon_{\ell}}, \\
\left(\prod_{\ell = 1}^n v_{\ell} ^{d_{\ell}} v_{-\ell}^{\varepsilon_{\ell}}\right) \otimes v_{-t} 
\mapsto  (-1)^{\psi(\varepsilon,t)} \prod_{\ell = 1}^n v_{\ell} ^{d_{\ell}} v_{-\ell}^{\varepsilon_{\ell} + \delta_{\ell,t} },
\end{gather*}
for all $1 \leq t \leq n$, where $\psi(\varepsilon,t) = \sum_{k = t + 1}^n \varepsilon_{k}$.  In the last line, the sign occurs due to the supercommutativity of $S(V_n)$.  Also, note that the product in the last line will be zero if $\varepsilon_t = 1$, since $v_{-t}^2 = 0$.

Using the definition of the functor $G$ from \cite[Section 6.1]{DKM}, we see that $\tcE_{i}$ acts on $S(V_n)^{\otimes m}$ as the map
$$
\left(\Id_{S(V_n)}\right)^{\otimes (i-1)} \otimes \xi \otimes \left(\Id_{S(V_n)}\right)^{\otimes (m - i - 1)},
$$
where $\xi$ is the composition:
\[
S(V_n) \otimes S(V_n) \xrightarrow{\Id_{S(V_{n})} \otimes \spl}  S(V_n) \otimes V_n \otimes S(V_n) \xrightarrow{\mer \otimes \Id_{S(V_{n})}}  S(V_n) \otimes S(V_n).
\]
It can be checked directly that this results in the formula for the action of $\tcE_{i}$ on the tensor product of monomials in \cref{P:pmtildeactiononA} where the sign $(-1)^{y_{i+1}(\varepsilon, t) - {y_{i}(\varepsilon, t)}}$ in the proposition is the product of the signs occurring in $\mer$ and $\spl$.

To verify the action of $\tcB_i$, first note that composing $\spl$ with the graded flip map $V_n \otimes S(V_n) \to S(V_n) \otimes V_n$ given by $x \otimes y \mapsto (-1)^{\p{x} \p{y}} y \otimes x$ yields a second split map:
\begin{gather*}
\spl' : S(V_n) \to S(V_n) \otimes V_n,\\
\prod_{\ell = 1}^n v_{\ell} ^{d_{\ell}} v_{-\ell}^{\varepsilon_{\ell}} \mapsto  
 \sum_{t = 1}^n             \left( \prod_{\ell = 1}^n v_{\ell} ^{d_{\ell} - \delta_{t, \ell}}v_{-\ell}^{\varepsilon_{\ell}} \right) \otimes d_t v_t 
 + \sum_{t = 1}^n (-1)^{\chi'(\varepsilon,t)} \left(\prod_{\ell = 1}^n v_{\ell} ^{d_{\ell}}v_{-\ell}^{\varepsilon_{\ell} - \delta_{\ell, t}} \right) \otimes \varepsilon_t v_{-t},
\end{gather*}
where now $\chi'(\varepsilon,t) = \sum_{k = t+1}^n \varepsilon_k$.  This is also a $\fp(n)$-module homomorphism.

 Recall that $\fp(n)$ preserves an odd bilinear form $\beta : V_n \otimes V_n \to \k$ given by:
$$
\beta(v_{i} \otimes v_j) = \delta_{i,-j}.
$$
Hence, we obtain a $U(\fp(n))$-module homomorphism
$
\xi' : S(V_n)  \otimes S(V_n) \to S(V_n) \otimes S(V_n) 
$
which is the composition:
\[
S(V_n) \otimes S(V_n) \xrightarrow{\spl' \otimes \spl}  S(V_n) \otimes V_n \otimes V_n \otimes S(V_n) \xrightarrow{\Id_{S(V_{n})} \otimes \beta \otimes \Id_{S(V_{n})}}  S(V_n) \otimes S(V_n).
\]
Explicitly computing $\xi'$ on one of our distinguished basis vectors and omitting the terms where $\Id_{S(V_n)} \otimes \beta \otimes \Id_{S(V_n)}$ will evaluate to zero shows that \(\prod_{\ell = 1}^n v_{\ell}^{d_{1,\ell}} v_{-\ell}^{\varepsilon_{1,\ell}} \otimes \prod_{\ell = 1}^n v_{\ell}^{d_{2,\ell}} v_{-\ell}^{\varepsilon_{2,\ell}}\) is sent to 
\begin{multline*}
\sum_{t = 1}^n d_{1,t} \varepsilon_{2,t} (-1)^{y_{2}(\varepsilon,t)} \left( \prod_{\ell = 1}^n v_{\ell}^{d_{1,\ell}-\delta_{\ell,t}} v_{-\ell}^{\varepsilon_{1,\ell}} \otimes \beta(v_{t} \otimes v_{-t}) \otimes
\prod_{\ell = 1}^n v_{\ell}^{d_{2,\ell}} v_{-\ell}^{\varepsilon_{2,\ell} - \delta_{\ell, t}} \right) \\
+ \sum_{t = 1}^n d_{2,t} \varepsilon_{1,t} (-1)^{y_1(\varepsilon,t)+1}\left( \prod_{\ell = 1}^n v_{\ell}^{d_{1,\ell}} v_{-\ell}^{\varepsilon_{1,\ell}-\delta_{\ell,t}} \otimes \beta(v_{-t} \otimes v_{t}) \otimes
\prod_{\ell = 1}^n v_{\ell}^{d_{2,\ell}- \delta_{\ell, t}} v_{-\ell}^{\varepsilon_{2,\ell} } \right).
\end{multline*}
Simplifying, this equals
\begin{multline*}
\sum_{t = 1}^n d_{1,t} \varepsilon_{2,t} (-1)^{y_{2}(\varepsilon,t)} \left( \prod_{\ell = 1}^n v_{\ell}^{d_{1,\ell}-\delta_{\ell,t}} v_{-\ell}^{\varepsilon_{1,\ell}} \otimes \prod_{\ell = 1}^n v_{\ell}^{d_{2,\ell}} v_{-\ell}^{\varepsilon_{2,\ell} - \delta_{\ell, t}} \right) \\
 + \sum_{t = 1}^n d_{2,t} \varepsilon_{1,t} (-1)^{y_1(\varepsilon,t)+ 1} \left( \prod_{\ell = 1}^n v_{\ell}^{d_{1,\ell}} v_{-\ell}^{\varepsilon_{1,\ell}-\delta_{\ell,t}} \otimes
\prod_{\ell = 1}^n v_{\ell}^{d_{2,\ell}- \delta_{\ell, t}} v_{-\ell}^{\varepsilon_{2,\ell} } \right).
\end{multline*}
Here the sign  $(-1)^{y_2(\varepsilon,t)}$ is the result of multiplying the sign occurring in the formula for the action of $\Id_{S(V_n)} \otimes \spl$ on a basis element with the sign occurring when the odd homomorphism $\beta$ is moved past the tensor factor $\prod_{\ell = 1}^n v_{\ell}^{d_{1,\ell}-\delta_{\ell,t}} v_{-\ell}^{\varepsilon_{1,\ell}}$.  Similarly, the sign of $(-1)^{y_1(\varepsilon,t) + 1}$ is the result of combining the sign occurring in $\spl' \otimes \Id_{S(V_n)}$ with the sign occurring when $\beta$ is moved past the tensor factor $\prod_{\ell = 1}^n v_{\ell}^{d_{1,\ell}} v_{-\ell}^{\varepsilon_{1,\ell}-\delta_{\ell,t}}$.

As with $\tcE_i$, one sees that the element $\tcB_i 1 \in \dot{U}(\tfp(m))$ acts on the $\tfp(n))$-module $S(V_n)^{\otimes m}$ as the odd module homomorphism
$$
\left(\Id_{S(V_n)}\right)^{\otimes (i-1)} \otimes \xi' \otimes \left(\Id_{S(V_n)}\right)^{\otimes (m - i - 1)}.
$$
It can be checked directly using the computation of $\xi'$ given above that this gives the formula for the action of $\tcB_i$ stated in the proposition.
 \end{proof}

\subsection{The Functor  \texorpdfstring{$\bU(\fp (n)) \to \tpmmodSn$}{dotU(p(n)) to tildep(m)-modSn}}\label{SS:CategoricalHoweDualityII}

Recall the contravariant isomorphisms $\Psi : \bU(\fp(n)) \to \bU(\tfp(n)) $, $\operatorname{refl}: \pWeb_{\uparrow, n} \to \pWeb_{\downarrow, n}$, and the covariant isomorphism $\mathcal{T}_{m}: \fp(m) \text{-mod}_{\mathcal{S}^{*},n} \to \tfp(m) \text{-mod}_{\mathcal{S},n}$ from \cref{Udottilde,SS:websforP,SS:Chevalleyisomorphism}, respectively.  Consider the following composition of functors:
\begin{equation}\label{E:compositefunctor}
\begin{tikzcd}[column sep = 8 ex]
\bU(\fp(n))  \arrow{r}{\operatorname{pr} \circ \Psi} &   \bU(\tfp(n))_{\geq 0}  \arrow{r}{H_n} & \pWeb_{\uparrow, n} \arrow{r}{\operatorname{refl}} &  \pWeb_{\downarrow,n}  \arrow{r}{\mathcal{T}_m \circ G_{\downarrow, n}}  & \tfp(m) \text{-mod}_{\mathcal{S},n}.
\end{tikzcd}
\end{equation}
Since $\Psi$ and $\operatorname{refl}$ are both contravariant superfunctors and the others are covariant, the entire composite is covariant.  In addition, since each functor is full and essentially surjective, so is the composite.  That is, we have the following result.

\begin{theorem} \label{T:CategoricalHoweDualityII}  For every $m, n \geq 1$, the composition 
\[
 \mathcal{T}_{m}\circ G_{\uparrow,n} \circ \operatorname{refl} \circ H_{n} \circ \Psi \circ \operatorname{pr}   : \bU(\fp(n)) \to \tpmmodSn
\]
is an essentially surjective, full functor.  
\end{theorem}

Set 
\[
\mathcal{A}'=\mathcal{A}'_{m,n} = S(W_m)^{\otimes n} \cong \bigoplus_{\tupnu \in X(T_n)_{\geq 0}} S^{\tupnu}(W_{m})  = \bigoplus_{\tupnu \in X(T_n)_{\geq 0}}  S^{\nu_{1}}(W_{m}) \otimes \dotsb \otimes S^{\nu_{n}}(W_{n}).
\]
 This is naturally a $U(\tfp(m))$-module with weights lying in $X(T_{m})$ and, hence, is a $\dot{U}(\tfp(m))$-module.

For each $\tupnu = \sum_{i=1}^{n}\nu_{i}\tupepsilon_{i} \in X(T_n)_{\geq 0}$ the composite functor in the previous theorem is given on objects by
\[
\tupnu \mapsto S^{\tupnu}(W_m) = S^{\nu_{1}}(W_{m}) \otimes \dotsb  \otimes S^{\nu_{n}}(W_{m}).
\]
Just as in \cref{T:CategoricalHoweDualityI}, the representation of $\bU(\fp(n))$ given by the functor in \cref{E:compositefunctor} makes $\mathcal{A}'$ into a $\dot{U}(\fp (n))$-module and the action is by $\dot{U}(\tfp(m))$-module homomorphisms.  The argument for \cref{T:FinitaryHoweDualityI} can be applied here to yield the following result.

\begin{theorem}\label{T:FinitaryHoweDualityII} For every $m, n \geq 1$, there are commuting actions of $\dot{U}(\tfp(m))$ and $\dot{U}(\fp(n))$ on 
\[
\mathcal{A'}=S(W_m)^{\otimes n}
\]
where $\dot{U}(\tfp(m))$ has the natural action.

Furthermore, these actions induce superalgebra homomorphisms
$$
\begin{tikzcd}[row sep = small, column sep = small]
\dot{U}(\tfp(m)) \arrow{dr}{\tilde{\varrho}'_{  }} &  & \dot{U}(\fp(n)) \arrow{dl}[swap]{\varrho'_{  }} \\
 & \End^\fin_{\k}(\mathcal{A'}_{m,n})  & 
\end{tikzcd}
$$
and the image of $\varrho'_{  }$ is the full centralizer of the image of $\tilde{\varrho}'_{  }$.

\end{theorem}

There is of course an analogue of \cref{P:pmtildeactiononA} which describes the action of $\dot{U}(\fp (n))$ on $\mathcal{A}'_{m,n}$.  We leave it to the interested reader to work out these formulas for themselves.

\subsection{Intertwining Actions}\label{SS:IntertwiningActionsCategorical}
Comparing \cref{T:CategoricalHoweDualityI,T:FinitaryHoweDualityI}   with \cref{T:CategoricalHoweDualityII,T:FinitaryHoweDualityII} it is natural to ask how the the representations of $\bU(\tfp (m))$ and $\bU (\fp (n))$ are related or, equivalently, how the commuting actions of $\dot{U}(\tfp(m))$ and $\dot{U}(\fp(n))$ on the superspaces $\mathcal{A}_{m,n}=S(V_{n})^{\otimes m}$ and $\mathcal{A}'_{m,n}=S(W_{m})^{\otimes n}$ are related.  Answering this question will also resolve the asymmetry in the roles of $\tfp (m)$  and $\fp (n)$.

For all $n\geq 1$, let $\operatorname{str}_{n} : \gl (n|n) \to \k$ be the supertrace representation given by 
\[
\operatorname{str}_{n}\left(\begin{matrix} A & B \\
                                           C & D
\end{matrix} \right) = \operatorname{tr}(A) - \operatorname{tr}(D),
\] where $\operatorname{tr}: \gl(n) \to \k$ is the ordinary trace function.  This restricts to define a one-dimensional representation of both $\fp (n)$ and $\tfp (n)$ for all $n \geq 1$.  In both cases, the corresponding module is one-dimensional and spanned by a weight vector of weight $2\tupdelta_n$, where $\tupdelta_n := \tupepsilon_1 + \cdots + \tupepsilon_n$.    Given an integer $a$, let $\k_{a\tupdelta_{n}}=\k v_{a\tupdelta_{n}}$ denote the one-dimensional, purely even representation of $\fp (n)$ of weight $a\tupdelta_n$ obtained by scaling the supertrace by $a/2$.  For any $n, m \geq 1$, view $\k_{a\tupdelta_{n}}$ as a $U(\fp (n)) \otimes U(\tfp (m))$-module with trivial $U(\tfp (m))$ action.  Similarly, let $\tilde{\k}_{a\tupdelta_{m}}=\k \tilde{v}_{a\tupdelta_{m}}$ denote the one-dimensional, purely even representation of $\tfp (m)$ of weight $a\tupdelta_m$ with a trivial $U(\fp(n))$ action for any $n \geq 1$.  In particular, $\mathcal{A} \otimes \k_{m\tupdelta_{n}} $ and  $\mathcal{A}' \otimes \tilde{\k}_{n\tupdelta_{m}}$ are still both $\dot{U}(\fp (n))$- and $\dot{U}(\tfp (m))$-modules and the actions still mutually commute.

Define a linear map
\begin{align*}
\phi_{m,n}: S(W_m)^{\otimes n}  \to S(V_n)^{\otimes m} 
\end{align*} on the basis given in \cref{SS:UpnActiononA} 
via the rule
\begin{align*}
\bigotimes_{\ell = 1}^n \prod_{k = 1}^m w_k^{d_{k\ell}} w_{-k}^{\varepsilon_{k \ell}}
\;\;
\mapsto
\;\;
(-1)^{\kappa(\bbep) + \chi(\bbep)}
\bigotimes_{k=1}^m \prod_{\ell=1}^n v_\ell^{d_{k\ell}}v_{-\ell}^{1 - \varepsilon_{k \ell}},
\end{align*}
where
\(d_{k \ell} \in \Z_{\geq 0}\), \(\varepsilon= (\varepsilon_{k, \ell})_{(k,\ell) \in I_{m,n}} \) with $\varepsilon_{k,\ell}\in \{0,1 \}$, and where
\begin{align*}
\kappa(\varepsilon) = \sum_{ \substack{ 1 \leq  k<k' \leq m \\ 1 <\ell' < \ell \leq n  }} 
 (1-\varepsilon_{k \ell})(1-\varepsilon_{k' \ell'}),
 \qquad
 \qquad
 \chi(\varepsilon) = \sum_{\substack{1 \leq k  \leq m \\ 1 \leq \ell \leq n}} ( m\ell + m+ mn + k+1) \varepsilon_{k \ell}.
\end{align*}  Comparing the parity of the input and output vectors shows  the map $\phi_{m,n}$ has parity $\overline{mn}$.   Let $\eta_{m,n}: \tilde{\k}_{n\tupdelta_{m}} \to \k_{m\tupdelta_{n}}$ be the even linear map defined by $\eta_{m,n}(v_{n\tupdelta_{m}}) = v_{m\tupdelta_{n}}$.

The following result gives the precise connection between $\mathcal{A}'_{m,n}$ and $\mathcal{A}_{m,n}$.
\begin{theorem}\label{T:IntertwiningActionsCategorical}  For all $m,n \geq 1$, the map
\begin{align*}
\phi_{m,n} \otimes \eta_{m,n}: \mathcal{A}'_{m,n} \otimes \tilde{\k}_{n\tupdelta_{m}}  \to \mathcal{A}_{m,n} \otimes \k_{m\tupdelta_{n}}
\end{align*} is an isomorphism of $\dot{U}(\fp(n)) \otimes \dot{U}(\tfp (m))$-modules and has parity $\overline{mn}$.
\end{theorem}
\begin{proof}
The map \(\phi_{m,n}\otimes \eta_{m,n} \) is an isomorphism of superspaces since, up to a sign, it is a bijection between bases for each module. Next we claim \(\phi_{m,n} \otimes \eta_{m,n}\) commutes with the action of $\dot{U}(\fp (n)) \otimes \dot{U}(\tfp (m))$.  It suffices to check that it commutes with the generators in \cref{SS:Udot,Udottilde}.

First we consider the action of the weight idempotents of $\dot{U}(\fp (n))$.  By definition, the $\tuplambda= \sum_{i=1}^{n} \lambda_{i}\tupepsilon_{n}$ weight space of $\mathcal{A}'$ is $S^{\tuplambda}(W_{m}) \otimes v_{n\tupdelta_{m}}$.   Using the basis for $W_{m}$ fixed in \cref{SS:Ptilde}, $S^{\tuplambda}(W_{m}) \otimes v_{n\tupdelta_{m}}$ has basis 
\begin{equation}\label{}
\left\{\left. \left(  \bigotimes_{\ell = 1}^{n}\prod_{k = 1}^m w_{k}^{d_{k,\ell}} w_{-k}^{\varepsilon_{k, \ell}}\right) \otimes v_{n\tupdelta_{m}} \;\right|\; \sum_{k=1}^{m} d_{k,\ell}+ \sum_{k=1}^{m} \varepsilon_{k,\ell} = \lambda_{\ell} \text{ for } \ell =1, \dotsc , n \right\}.
\end{equation}  Up to a scalar the map $\phi_{m,n} \otimes \eta_{m,n}$ evaluates on such a vector by
\[
\left(\bigotimes_{\ell = 1}^{n}\prod_{k = 1}^m w_{k}^{d_{k,\ell}} w_{-k}^{\varepsilon_{k, \ell}} \right) \otimes v_{n\tupdelta_{m}} \mapsto \left( \bigotimes_{k = 1}^{m} \prod_{\ell = 1}^{n} v_{\ell }^{d_{k,\ell}} v_{-\ell}^{1-\varepsilon_{k, \ell}}\right)  \otimes v_{m\tupdelta_{n}}.
\]  However, under the natural action by $\fp (n)$ the weight of the latter vector is $\tupmu = \sum_{\ell = 1}^{n} \mu_{\ell}\tupepsilon_{\ell} + m\tupdelta_{n}$ where 
\[
\mu_{\ell}= \sum_{k=1}^{m} d_{k, \ell} - \sum_{k=1}^{m} (1- \varepsilon_{k, \ell}) = \lambda_{\ell}  - m.
\]  That is, $\tupmu = \tuplambda - m\tupdelta_{n} + m\tupdelta_{n}=\tuplambda$, as desired. An identical analysis shows $\phi_{m,n} \otimes \eta_{m,n}$ preserves the action of the weight idempotents for $\dot{U}(\tfp (m))$.

As in the proof of \cref{P:pmtildeactiononA}, for $\cY \in \{\cE_{i}, \cF_{i}, \cB_{1,1}, \cB_{i}, \cC_{i}\}$ we have the operators $\cY=\sum_{\tuplambda \in X(T_{n})}{ \cY}1_{\tuplambda}$ and $\tcY =\sum_{\tuplambda \in X(T_{n})}\tcY 1_{\tuplambda}$.  It suffices to verify that  $\phi_{m,n} \otimes \eta_{m,n}$ commutes with the maps $\cY$ and $\tcY$ for all allowable values of $\cY$.  This can be done on a case-by-case basis.   On the whole this is straightforward but does require some care to assure signs are preserved.  We include details for the actions of \(\tcE_{i}\) and \(\tcB_i\) below as exemplars of the checks which need to be done.  Full details are included in the {\tt arXiv} version of this paper as explained in \cref{SS:ArxivVersion}.  
Also, since these remaining generators act trivially on $\k_{m\tupdelta_{n}}$ and $\k_{n\tupdelta_{n}}$, they and the map $\eta_{m,n}$ will be omitted in what follows.

\noindent \textbf{Action of \(\tcE_i\):}
We have that \(\phi_{m,n}\left( \tcE_i \cdot \bigotimes_{\ell = 1}^n \prod_{k = 1}^m v_k^{d_{k\ell}} v_{-k}^{\varepsilon_{k \ell}} \right)\) is equal to 
\begin{align*}
&
\phi_{m,n} \left( 
\sum_{t=1}^n
d_{i+1,t} \bigotimes_{\ell = 1}^n \prod_{k = 1}^m v_k^{d_{k \ell} + \delta_{t, \ell}( \delta_{i,k} - \delta_{i+1,k})}v_{-k}^{\varepsilon_{k \ell}}
\right)\\
&\hspace{0.5in}
- \phi_{m,n} \left(
\sum_{t=1}^n
\varepsilon_{i t}(1 - \varepsilon_{i+1,t})
\bigotimes_{\ell = 1}^n
\prod_{k=1}^m
v_k^{d_{k\ell}}
v_{-k}^{\varepsilon_{k\ell} + \delta_{t, \ell}(\delta_{i+1,k} - \delta_{i,k})}
\right)\\
&=
(-1)^{\kappa(\bbep) + \chi(\bbep)}
\sum_{t=1}^n d_{i+1,t} \bigotimes_{k=1}^m \prod_{\ell=1}^n
v_\ell^{d_{k \ell} + \delta_{t, \ell}(\delta_{i,k} - \delta_{i+1,k})} v_{-\ell}^{1 - \varepsilon_{k \ell}}\\
&
\hspace{0.5in}
+
\sum_{t=1}^n \varepsilon_{it}(1-\varepsilon_{i+1,t}) (-1)^{\kappa(\bbep^{(t)})+ \chi(\bbep^{(t)})+1}\bigotimes_{k=1}^m \prod_{\ell =1 }^n
v_\ell^{d_{k \ell}} v_{-\ell}^{1 - \varepsilon_{k \ell} - \delta_{t, \ell}(\delta_{i+1,k} - \delta_{i,k})},
\end{align*}
where 
\begin{align*}
\varepsilon^{(t)}_{k \ell} = 
\begin{cases}
\varepsilon_{i+1,t} & \textup{if }k = i, \ell = t;\\
\varepsilon_{i,t} & \textup{if }k = i+1, \ell = t;\\
\varepsilon_{k \ell}  & \textup{otherwise}.
\end{cases}
\end{align*}
We also have that \(\tcE_i \cdot \phi_{m,n}\left( \bigotimes_{\ell = 1}^n \prod_{k = 1}^m v_k^{d_{k\ell}} v_{-k}^{\varepsilon_{k \ell}} \right)\) is equal to
\begin{align*}
&
\tcE_i \cdot (-1)^{\kappa(\bbep) + \chi(\bbep)}
\bigotimes_{k=1}^m \prod_{\ell=1}^n v_\ell^{d_{k\ell}}v_{-\ell}^{1 - \varepsilon_{k \ell}}\\
&=
(-1)^{\kappa(\bbep) + \chi(\bbep)}\sum_{t=1}^n d_{i+1,t} \bigotimes_{k=1}^m \prod_{\ell =1 }^n
v_\ell^{d_{k\ell} + \delta_{t, \ell}( \delta_{k, i} - \delta_{k, i+1})}v_{-\ell}^{1-\varepsilon_{k \ell}}
\\
&+ (-1)^{\kappa(\bbep) + \chi(\bbep)}
\sum_{t=1}^n(1-\varepsilon_{i+1,t}) 
\varepsilon_{i,t}
(-1)^{\beta(\bbep,t)}
\bigotimes_{k=1}^m 
\prod_{\ell=1}^n
v_\ell^{d_{k\ell}} v_{-\ell}^{1- \varepsilon_{k \ell} + \delta_{t,\ell}(\delta_{k,i} - \delta_{k, i+1})},
\end{align*}
where
\begin{align*}
\beta(\bbep,t) = \sum_{\ell< t} (1- \varepsilon_{i+1,\ell}) + \sum_{\ell>t} (1-\varepsilon_{i,\ell}).  
\end{align*}
So equality holds provided
\begin{align*}
(-1)^{\kappa(\bbep^{(t)}) + \chi(\bbep^{(t)})+ 1} = (-1)^{\kappa(\bbep) + \chi(\bbep)+ \beta(\bbep, t)}
\end{align*}
for all \(t\) such that \(\varepsilon_{i+1,t} =0\), \(\varepsilon_{i,t} = 1\). In this situation note that 
\begin{align*}
\kappa(\varepsilon^{(t)})- \kappa(\varepsilon)  &=   \sum_{\ell < t} (1-\varepsilon_{i+1,t})(1-\varepsilon_{i+1,\ell}) + \sum_{\ell > t} (1-\varepsilon_{i \ell})(1-\varepsilon_{it})\\
&\hspace{0.5in}-
\sum_{\ell < t}(1-\varepsilon_{it})(1-\varepsilon_{i+1,\ell})
-\sum_{\ell >t}(1-\varepsilon_{i \ell})(1 -\varepsilon_{i+1,t})\\
&=
\sum_{\ell < t}(1-\varepsilon_{i+1,\ell}) - \sum_{\ell > t}(1 - \varepsilon_{i\ell}).
\end{align*}
This last term has the same parity as \(\beta( \bbep, t)\), so it follows that \(\kappa(\bbep^{(t)})\) has the same parity as \(\kappa(\bbep) + \beta(\bbep, t)\). Thus it remains to show that \(\chi(\bbep^{(t)}) + 1\) has the same parity as \(\chi(\bbep)\). Noting that
\begin{align*}
\chi(\varepsilon^{(t)}) - \chi(\varepsilon) = \varepsilon_{i,t}-\varepsilon_{i+1,t} = 1
\end{align*}
provides the result, so the action of \(\tcE_i\) commutes with the map \(\varphi_{m,n}\).

\noindent \textbf{Action of \(\tcB_i\):}
We have that \(\phi_{m,n}\left( \tcB_i \cdot \bigotimes_{\ell = 1}^n \prod_{k = 1}^m v_k^{d_{k\ell}} v_{-k}^{\varepsilon_{k \ell}} \right)\) is equal to
\begin{align*}
&
\phi_{m,n} \left(
\sum_{t=1}^n 
d_{i+1, t}(1-\varepsilon_{i,t})
(-1)^{\tau(\bbep,i+1,t)}
 \bigotimes_{\ell = 1}^n \prod_{k = 1}^m v_k^{d_{k\ell} - \delta_{t, \ell}\delta_{k,i+1}} v_{-k}^{\varepsilon_{k \ell} + \delta_{t,\ell}\delta_{k,i}}
\right)\\
&\hspace{0.25in}+
\phi_{m,n} \left(
\sum_{t=1}^n 
d_{i, t}(1-\varepsilon_{i+1,t})
(-1)^{\tau(\bbep,i+1,t)}
 \bigotimes_{\ell = 1}^n \prod_{k = 1}^m v_k^{d_{k\ell} - \delta_{t, \ell}\delta_{k,i}} v_{-k}^{\varepsilon_{k \ell} + \delta_{t,\ell}\delta_{k,i+1}}
\right)\\
&=
\sum_{t=1}^n 
(-1)^{\kappa(\varepsilon^{(i,t)}) + \chi(\varepsilon^{(i,t)})+\tau(\varepsilon,i+1,t)}
d_{i+1, t}(1-\varepsilon_{i,t})
\bigotimes_{k=1}^n \prod_{\ell=1}^n
v_\ell^{d_{k\ell} - \delta_{t, \ell}\delta_{k,i+1}} v_{-\ell}^{1- \varepsilon_{k \ell} - \delta_{t,\ell}\delta_{k,i}}\\
&\hspace{0.25in}+
\sum_{t=1}^n
(-1)^{\kappa( \varepsilon^{(i+1,t)}) + \chi(\varepsilon^{(i+1,t)}) + \tau(\varepsilon,i+1,t)}
d_{i, t}(1-\varepsilon_{i+1,t})
\bigotimes_{k=1}^n
\prod_{\ell=1}^n v_\ell^{d_{k\ell} - \delta_{t, \ell}\delta_{k,i}} v_{-\ell}^{1-\varepsilon_{k \ell} - \delta_{t,\ell}\delta_{k,i+1}},
\end{align*}
where
\begin{align*}
\varepsilon^{(j,t)}_{k \ell} = 
\begin{cases}
\varepsilon_{j,t} +1 & \textup{if }k = j, \ell = t;\\
\varepsilon_{k \ell}  & \textup{otherwise}.
\end{cases}
\qquad\qquad
\textup{
and
}
\qquad\qquad
\tau(\varepsilon,j,t) = \sum_{ \substack{ 1 \leq k \leq m \\ 1 \leq r < t} }\varepsilon_{k,r}
+ 
\sum_{1 \leq s < j}  \varepsilon_{ s,t}.
\end{align*}
We also have that \((-1)^{mn}\tcB_i \cdot \phi_{m,n}\left(\bigotimes_{\ell = 1}^n \prod_{k = 1}^m v_k^{d_{k\ell}} v_{-k}^{\varepsilon_{k \ell}} \right)\) is equal to 
\begin{align*}
&
\tcB_i\cdot (-1)^{\kappa(\bbep)+ \chi(\bbep)+mn}\bigotimes_{k=1}^m \prod_{\ell = 1}^n v_\ell^{d_{k\ell}}v_{-\ell}^{1-\varepsilon_{k \ell}}\\
&=
\sum_{t=1}^n 
(-1)^{\kappa(\bbep)+ \chi(\bbep) + \beta(\bbep)+ \gamma_1(\bbep,t)+mn} d_{i,t}(1-\varepsilon_{i+1,t})\bigotimes_{k=1}^m\prod_{\ell=1}^n v_{\ell}^{d_{k\ell} - \delta_{k,i}\delta_{\ell,t}}v_{-\ell}^{1-\varepsilon_{k\ell}-\delta_{k,i+1}\delta_{\ell,t}}\\
&
\hspace{0.5in}+
\sum_{t=1}^n(-1)^{\kappa(\bbep)+ \chi(\bbep) + \beta(\bbep)+ \gamma_2(\bbep,t)+mn+1} d_{i+1,t}(1-\varepsilon_{i,t})\bigotimes_{k=1}^m\prod_{\ell=1}^n v_{\ell}^{d_{k\ell} - \delta_{k,i+1}\delta_{\ell,t}}v_{-\ell}^{1-\varepsilon_{k\ell}-\delta_{k,i}\delta_{\ell,t}},
\end{align*}
where
\begin{align*}
\beta(\bbep) = \sum_{\substack{1 \leq k \leq i  \\ 1 \leq \ell \leq n  }}(1-\varepsilon_{k \ell}),
\qquad
\qquad
\gamma_1(\bbep,t) = \sum_{1 \leq r < t}(1-\varepsilon_{i+1,r}),
\qquad
\qquad
\gamma_2(\bbep,t) = \sum_{t < r \leq n}(1-\varepsilon_{i,r}).
\end{align*}
So equality holds provided
\begin{align}\label{signBeq1}
(-1)^{\kappa( \bbep^{(i+1,t)}) + \chi(\bbep^{(i+1,t)}) + \tau(\bbep,i+1,t)}
&=
(-1)^{\kappa(\bbep)+ \chi(\bbep) + \beta(\bbep)+ \gamma_1(\bbep,t)+mn}
\end{align}
for all \(t\) such that \(\varepsilon_{i+1,t} = 0\), and
\begin{align}\label{signBeq2}
(-1)^{\kappa( \bbep^{(i+1,t)}) + \chi(\bbep^{(i+1,t)}) + \tau(\bbep,i+1,t)}
&=
(-1)^{\kappa(\bbep)+ \chi(\bbep) + \beta(\bbep)+ \gamma_2(\bbep,t)+mn+1} 
\end{align}
for all \(t\) such that \(\varepsilon_{i,t} = 0\).
First, note that in this setting we have
\begin{align*}
\kappa(\bbep) - \kappa(\bbep^{(i+1,t)}) = \sum_{\substack{i+1< k \leq m\\1 \leq \ell < t}}(1-\varepsilon_{k \ell})
+ 
\sum_{\substack{1 \leq k < i+1 \\ t < \ell \leq n}} (1-\varepsilon_{k \ell}).
\end{align*}
and 
\begin{align*}
\chi(\bbep)-\chi(\bbep^{(i+1,t)}) =-mt-m-mn-i - 2.
\end{align*}
Then
\begin{align*}
\beta(\bbep) + \gamma_1(\bbep,t) +\kappa(\bbep) - \kappa(\bbep^{(i+1,t)})&\equiv
\beta(\bbep) -  \hspace{-3mm}\sum_{\substack{1 \leq k < i+1 \\ t < \ell \leq n}} (1-\varepsilon_{k \ell}) + \gamma_1(\bbep,t) 
+ \hspace{-3mm}\sum_{\substack{i+1< k \leq m\\1 \leq \ell < t}}(1-\varepsilon_{k \ell})
 \mod 2\\
&\equiv 
\sum_{\substack{  1 \leq k \leq m \\ 1 \leq \ell < t}} (1-\varepsilon_{k \ell}) + \sum_{\substack{ 1 \leq k < i+1  }}(1- \varepsilon_{k,t}) \mod 2.
\end{align*}
Therefore we have
\begin{align*}
\beta(\bbep) + \gamma_1(\bbep,t) +\kappa(\bbep) - \kappa(\bbep^{(i+1,t)}) + \tau(\bbep,i+1,t)
&\equiv \hspace{-3mm}
\sum_{\substack{  1 \leq k \leq m \\ 1 \leq \ell < t}} 1 + \sum_{\substack{ 1 \leq k < i+1  }}1 \mod 2\\
&\equiv m(t-1) + i \mod 2\\
&\equiv \chi(\bbep)-\chi(\bbep^{(i+1,t)}) +mn \mod 2,
\end{align*}
so (\ref{signBeq1}) holds. Next note
\begin{align*}
\kappa(\bbep) - \kappa(\bbep^{(i,t)}) = \sum_{\substack{i< k \leq m\\1 \leq \ell < t}}(1-\varepsilon_{k \ell})
+ 
\sum_{\substack{1 \leq k < i \\ t < \ell \leq n}} (1-\varepsilon_{k \ell})
\end{align*}
and 
\begin{align*}
\chi(\bbep)-\chi(\bbep^{(i,t)}) =-mt-m-mn-i -1.
\end{align*}
Then 
\begin{align*}
\beta(\bbep) + \gamma_2(\bbep, t)  + \kappa(\bbep) - \kappa(\bbep^{(i,t)})  &\equiv
\beta(\bbep) - \sum_{\substack{1 \leq k < i \\ t < \ell \leq n}} (1-\varepsilon_{k \ell}) - \gamma_2(\bbep,t) +  \hspace{-3mm}\sum_{\substack{i< k \leq m\\1 \leq \ell < t}}(1-\varepsilon_{k \ell}) \mod 2\\
&\equiv 
\sum_{\substack{1 \leq k \leq m \\ 1 \leq \ell < t} }(1- \varepsilon_{k \ell}) + \sum_{1 \leq k \leq i} (1-\varepsilon_{kt}) \mod 2. \\ 
\end{align*}
Therefore we have
\begin{align*}
\beta(\bbep) + \gamma_2(\bbep, t)  + \kappa(\bbep) - \kappa(\bbep^{(i,t)}) + \tau(\bbep, i+1,t) + 1 & \equiv \sum_{\substack{1 \leq k \leq m \\ 1 \leq\ell < t}} 1 + \sum_{1 \leq k \leq i} 1 \mod 2\\
& \equiv m(t-1) + t + i + 1 \mod 2\\
& \equiv \chi(\bbep)-\chi(\bbep^{(i,t)}) +mn \mod 2,
\end{align*}
so (\ref{signBeq2}) holds, and therefore the action of \(\tcB_i\) commutes with the map \(\varphi_{m,n}\).
\begin{answer}

(\(\tcF_i\) action)
We have that \(\phi_{m,n}\left( \tcF_i \cdot \bigotimes_{\ell = 1}^n \prod_{k = 1}^m v_k^{d_{k\ell}} v_{-k}^{\varepsilon_{k \ell}} \right)\) is equal to
\begin{align*}
&
\phi_{m,n} \left( 
\sum_{t=1}^n
d_{i,t} \bigotimes_{\ell = 1}^n \prod_{k = 1}^m v_k^{d_{k \ell} + \delta_{t, \ell}( \delta_{i+1,k} - \delta_{i,k})}v_{-k}^{\varepsilon_{k \ell}}
\right)\\
&
\hspace{0.5in}
- \phi_{m,n} \left(
\sum_{t=1}^n
\varepsilon_{i+1, t} (1-\varepsilon_{i,t})
\bigotimes_{\ell = 1}^n
\prod_{k=1}^m
v_k^{d_{k\ell}}
v_{-k}^{\varepsilon_{k\ell} + \delta_{t, \ell}(\delta_{i,k} - \delta_{i+1,k})}
\right)\\
&=
(-1)^{\kappa(\bbep) + \chi(\bbep)}
\sum_{t=1}^n d_{i,t} \bigotimes_{k=1}^m \prod_{\ell=1}^n
v_\ell^{d_{k \ell} + \delta_{t, \ell}(\delta_{i+1,k} - \delta_{i,k})} v_{-\ell}^{1 - \varepsilon_{k \ell}}\\
&+
\sum_{t=1}^n \varepsilon_{i+1,t}(1-\varepsilon_{i,t}) (-1)^{\kappa(\bbep^{(t)})+ \chi(\bbep^{(t)})+1}\bigotimes_{k=1}^m \prod_{\ell =1 }^n
v_\ell^{d_{k \ell}} v_{-\ell}^{1 - \varepsilon_{k \ell} - \delta_{t, \ell}(\delta_{i,k} - \delta_{i+1,k})},
\end{align*}
where 
\begin{align*}
\varepsilon^{(t)}_{k \ell} = 
\begin{cases}
\varepsilon_{i+1,t} & \textup{if }k = i, \ell = t;\\
\varepsilon_{i,t} & \textup{if }k = i+1, \ell = t;\\
\varepsilon_{k \ell}  & \textup{otherwise}.
\end{cases}
\end{align*}
We also have that \(\tcF_i \cdot \phi_{m,n}\left( \bigotimes_{\ell = 1}^n \prod_{k = 1}^m v_k^{d_{k\ell}} v_{-k}^{\varepsilon_{k \ell}} \right)\) is equal to 
\begin{align*}
&\tcF_i \cdot (-1)^{\kappa(\bbep) + \chi(\bbep)}
\bigotimes_{k=1}^m \prod_{\ell=1}^n v_\ell^{d_{k\ell}}v_{-\ell}^{1 - \varepsilon_{k \ell}}\\
&=
(-1)^{\kappa(\bbep) + \chi(\bbep)}\sum_{t=1}^n d_{i,t} \bigotimes_{k=1}^m \prod_{\ell =1 }^n
v_\ell^{d_{k\ell} + \delta_{t, \ell}( \delta_{k, i+1} - \delta_{k, i})}v_{-\ell}^{1-\varepsilon_{k \ell}}
\\
&
\hspace{0.5in}
+ (-1)^{\kappa(\bbep) + \chi(\bbep)}
\sum_{t=1}^n(1-\varepsilon_{i,t}) 
\varepsilon_{i+1,t}
(-1)^{\beta(\bbep,t)}
\bigotimes_{k=1}^m 
\prod_{\ell=1}^n
v_\ell^{d_{k\ell}} v_{-\ell}^{1- \varepsilon_{k \ell} + \delta_{t,\ell}(\delta_{k,i+1} - \delta_{k, i})}
\end{align*}
where
\begin{align*}
\beta(\bbep,t) = \sum_{\ell> t} (1- \varepsilon_{i,\ell}) + \sum_{\ell<t} (1-\varepsilon_{i+1,\ell})  
\end{align*}
So we must check that 
\begin{align*}
(-1)^{\kappa(\bbep^{(t)}) + \chi(\bbep^{(t)}) + 1} = (-1)^{\kappa(\bbep) + \chi(\bbep) + \beta(\bbep,t)}.
\end{align*}
in the situation where \(\varepsilon_{i,t} = 0\) and \(\varepsilon_{i+1,t} = 1\). As in the \(\tcE_i\) calculation, we have
\begin{align*}
\kappa(\bbep^{(t)}) - \kappa(\bbep) = 
\sum_{\ell < t} (1-\varepsilon_{i+1,\ell}) - \sum_{\ell > t}(1-\varepsilon_{i \ell}),
\end{align*}
which is the same parity as \(\beta(\bbep,t)\). Finally, noting that
\begin{align*}
\chi(\bbep^{(t)}) - \chi(\bbep) = \varepsilon_{i,t}-\varepsilon_{i+1,t} = -1
\end{align*}
provides the result, so the action of \(\tcF_i\) commutes with the map \(\varphi_{m,n}\).

(\(\tcC_i\) action)
We have that \(\phi_{m,n}\left( \tcC_i \cdot \bigotimes_{\ell = 1}^n \prod_{k = 1}^m v_k^{d_{k\ell}} v_{-k}^{\varepsilon_{k \ell}} \right)\) is equal to
\begin{align*}
&
\phi_{m,n}\left( 
\sum_{t = 1}^n
\varepsilon_{i+1,t}
(-1)^{\tau(\bbep, i+1, t)}
\bigotimes_{\ell = 1}^n \prod_{k =1}^m 
v_{k}^{d_{k\ell} + \delta_{t,\ell}\delta_{k, i}} v_{-k}^{\varepsilon_{k \ell} - \delta_{t \ell} \delta_{k, i+1}}
\right)\\
&
\hspace{0.5in}
+
\phi_{m,n}\left( 
\sum_{t = 1}^n
\varepsilon_{i,t}
(-1)^{\tau(\bbep, i, t)+1}
\bigotimes_{\ell = 1}^n \prod_{k =1}^m 
v_{k}^{d_{k\ell} + \delta_{t,\ell}\delta_{k, i+1}} v_{-k}^{\varepsilon_{k \ell} - \delta_{t \ell} \delta_{k, i}}
\right)\\
&=
\sum_{t = 1}^n
\varepsilon_{i+1,t}
(-1)^{\tau(\bbep, i+1, t) + \kappa(\bbep^{(i+1,t)^-}) + \chi(\bbep^{(i+1,t)^-})}
\bigotimes_{k = 1}^m \prod_{\ell =1}^n
v_{\ell}^{d_{k\ell} + \delta_{t,\ell}\delta_{k, i}} v_{-\ell}^{1-\varepsilon_{k \ell} + \delta_{t \ell} \delta_{k, i+1}}
\\
&
\hspace{0.5in}+
\sum_{t = 1}^n
\varepsilon_{i,t}
(-1)^{\tau(\bbep, i, t)+ \kappa(\bbep^{(i,t)^-}) + \chi(\bbep^{(i,t)^-})+1}
\bigotimes_{k= 1}^m \prod_{\ell =1}^n
v_{\ell}^{d_{k\ell} + \delta_{t,\ell}\delta_{k, i+1}} v_{-\ell}^{1-\varepsilon_{k \ell} + \delta_{t \ell} \delta_{k, i}}
\\
\end{align*}
where
\begin{align*}
\varepsilon^{(j,t)^-}_{k \ell} = 
\begin{cases}
\varepsilon_{j,t} -1 & \textup{if }k = j, \ell = t;\\
\varepsilon_{k \ell}  & \textup{otherwise}.
\end{cases}
\qquad\qquad
\textup{
and
}
\qquad\qquad
\tau(\bbep,j,t) = \sum_{ \substack{ 1 \leq k \leq m \\ 1 \leq r < t} }\varepsilon_{k,r}
+ 
\sum_{1 \leq s < j}  \varepsilon_{ s,t}.
\end{align*}
We also have that \((-1)^{mn}\tcC_i \cdot \phi_{m,n}\left(\bigotimes_{\ell = 1}^n \prod_{k = 1}^m v_k^{d_{k\ell}} v_{-k}^{\varepsilon_{k \ell}} \right)\) is equal to
\begin{align*}
&
\tcC_i\cdot (-1)^{\kappa(\bbep)+ \chi(\bbep)+mn}\bigotimes_{k=1}^m \prod_{\ell = 1}^n v_\ell^{d_{k\ell}}v_{-\ell}^{1-\varepsilon_{k \ell}}\\
&=
\sum_{t = 1}^n
\bigotimes_{k=1}^m \prod_{\ell = 1}^n 
(-1)^{\kappa(\bbep)+ \chi(\bbep) + \beta(\bbep)+\gamma_1(\bbep,t)+mn}
\varepsilon_{i+1,t}
v_{\ell}^{d_{k \ell} + \delta_{k,i}\delta_{\ell, t}}v_{-\ell}^{1 - \varepsilon_{k \ell}+\delta_{k,i+1}\delta_{t,\ell}}\\
&
\hspace{0.5in}
+
\sum_{t = 1}^n
\bigotimes_{k=1}^m \prod_{\ell = 1}^n 
(-1)^{\kappa(\bbep)+ \chi(\bbep) + \beta(\bbep)+\gamma_2(\bbep,t)+mn+1}
\varepsilon_{i,t}
v_{\ell}^{d_{k \ell} + \delta_{k,i+1}\delta_{\ell, t}}v_{-\ell}^{1 - \varepsilon_{k \ell}+\delta_{k,i}\delta_{t,\ell}}\\
\end{align*}
where 
\begin{align*}
\beta(\bbep) = \sum_{\substack{1 \leq k \leq i  \\ 1 \leq \ell \leq n  }}(1-\varepsilon_{k \ell}),
\qquad
\qquad
\gamma_1(\bbep,t) = \sum_{1 \leq r < t}(1-\varepsilon_{i+1,r}),
\qquad
\qquad
\gamma_2(\bbep,t) = \sum_{t < r \leq n}(1-\varepsilon_{i,r})
\end{align*}
Now we just need to check that
\begin{align*}
(-1)^{\tau(\bbep, i+1, t) + \kappa(\bbep^{(i+1,t)^-}) + \chi(\bbep^{(i+1,t)^-})} &= 
(-1)^{\kappa(\bbep)+ \chi(\bbep) + \beta(\bbep)+\gamma_1(\bbep,t)+mn}
\end{align*}
for all \(t\) such that \(\varepsilon_{i+1,t} = 1\), and
\begin{align*}
(-1)^{\tau(\bbep, i, t)+ \kappa(\bbep^{(i,t)^-}) + \chi(\bbep^{(i,t)^-})+1}
&=(-1)^{\kappa(\bbep)+ \chi(\bbep) + \beta(\bbep)+\gamma_2(\bbep,t)+mn+1}
\end{align*}
for all \(t\) such that \(\varepsilon_{i,t} =1\).

We attack the first case. Note that 
\begin{align*}
\kappa(\bbep^{(i+1,t)^-})  - \kappa(\bbep) = \sum_{\substack{i+1< k \leq m\\1 \leq \ell < t}}(1-\varepsilon_{k \ell})
+ 
\sum_{\substack{1 \leq k < i+1 \\ t < \ell \leq n}} (1-\varepsilon_{k \ell}).
\end{align*}
and
\begin{align*}
\chi(\bbep^{(i+1,t)^-}) - \chi(\bbep) =-mt-m-mn-i - 2.
\end{align*}
Then the equivalence follows as in the \(\tcB_i\) case.

For the second equality, note:
\begin{align*}
\tau(\bbep, i+1,t) = \tau(\bbep, i, t) +1,
\end{align*}
since \(\varepsilon_{i,t} =1\).
Next note
\begin{align*}
 \kappa(\bbep^{(i,t)^-})- \kappa(\bbep) = \sum_{\substack{i< k \leq m\\1 \leq \ell < t}}(1-\varepsilon_{k \ell})
+ 
\sum_{\substack{1 \leq k < i \\ t < \ell \leq n}} (1-\varepsilon_{k \ell}).
\end{align*}
and
\begin{align*}
\chi(\bbep^{(i,t)^-}) - \chi(\bbep)=-mt-m-mn-i -1.
\end{align*}
Therefore we have
\begin{align*}
\beta(\bbep) + \gamma_2(\bbep, t)  + \kappa(\bbep) - \kappa(\bbep^{(i,t)^-}) + \tau(\bbep, i,t) 
& \equiv \chi(\bbep)-\chi(\bbep^{(i,t)^-}) +mn \mod 2,
\end{align*}
as in the \(\tcB_i\) case, proving the second equality. Therefore the action of \(\tcC_i\) commutes with the map \(\varphi_{m,n}\).

Next, we check that \(\varphi_{m,n}\) commutes with the action of \(\frak{p}_n\). Many of the arguments with respect to preservation of signs are similar to the cases above, so we will be slightly less verbose in our explanations from this point forward.

(\(\cE_i\) action) We have that \(  \phi_{m,n} \cdot \cE_i  \left(
\bigotimes_{\ell = 1}^n \prod_{k = 1}^m v_k^{d_{k\ell}} v_{-k}^{\varepsilon_{k \ell}}
\right)\) is equal to
\begin{align*}
&
  \phi_{m,n} \left( \sum_{t=1}^m   d_{t,i+1}  \bigotimes_{\ell=1}^n \prod_{k=1}^m     v_k^{d_{k\ell} + \delta_{k,t}(  \delta_{\ell,i}  - \delta_{\ell,i+1}     )}v_{-k}^{\varepsilon_{k \ell}}\right)\\
 &\hspace{0.5in}+
  \phi_{m,n} \left( \sum_{t=1}^m  (-1)^{\beta(\bbep,t)} \varepsilon_{t,i+1}(1-\varepsilon_{t,i})  \bigotimes_{\ell=1}^n \prod_{k=1}^m     v_k^{d_{k\ell}    }v_{-k}^{\varepsilon_{k \ell }   + \delta_{k,t}(  \delta_{\ell,i}  - \delta_{\ell,i+1} )  }\right)\\
  &= 
   (-1)^{\kappa(\bbep) + \chi(\bbep)  }    \sum_{t=1}^m   d_{t,i+1}  \bigotimes_{k=1}^m \prod_{\ell=1}^n     v_\ell^{d_{k\ell} + \delta_{k,t}(  \delta_{\ell,i}  - \delta_{\ell,i+1}     )}v_{-\ell}^{1-\varepsilon_{k \ell}}\\
   &\hspace{0.5in}+
   \sum_{t=1}^m  (-1)^{\beta(\bbep,t)  +\kappa(\tilde{\bbep}^{(t)})  + \chi(\tilde{\bbep}^{(t)})  } \varepsilon_{t,i+1}(1-\varepsilon_{t,i})  \bigotimes_{k=1}^m \prod_{\ell=1}^n     v_\ell^{d_{k\ell}    }v_{-\ell}^{1- \varepsilon_{k \ell }   - \delta_{k,t}(  \delta_{\ell,i}  - \delta_{\ell,i+1} )  }
\end{align*}
where
\begin{align*}
\beta(\bbep,t) = \sum_{k<t} \varepsilon_{  k ,i+1} + \sum_{k>t} \varepsilon_{k,i}
\end{align*}
and
\begin{align*}
\tilde{\varepsilon}^{(t)}_{k \ell} = 
\begin{cases}
\varepsilon_{t,i+1} & \textup{if }k = t, \ell = i;\\
\varepsilon_{t,i} & \textup{if }k = t, \ell = i+1;\\
\varepsilon_{k \ell}  & \textup{otherwise}.
\end{cases}
\end{align*}
We also have that \(\cE_i \cdot \phi_{m,n}\left( \bigotimes_{\ell = 1}^n \prod_{k = 1}^m v_k^{d_{k\ell}} v_{-k}^{\varepsilon_{k \ell}} \right)\) is equal to
\begin{align*}
&
\cE_i \cdot (-1)^{\kappa(\bbep) + \chi(\bbep)}
\bigotimes_{k=1}^m \prod_{\ell=1}^n v_\ell^{d_{k\ell}}v_{-\ell}^{1 - \varepsilon_{k \ell}}\\
&=
(-1)^{\kappa(\bbep) + \chi(\bbep)}
\sum_{t=1}^m 
d_{t, i+1}
\bigotimes_{k=1}^m \prod_{\ell=1}^n 
v_{\ell}^{d_{k\ell} + \delta_{k,t}( \delta_{\ell, i} - \delta_{\ell, i+1} )}v_{-\ell}^{1 - \varepsilon_{k \ell}}\\
&
\hspace{0.5in}
+
(-1)^{\kappa(\bbep) + \chi(\bbep) + 1}
\sum_{t=1}^m
(1-\varepsilon_{t,i})\varepsilon_{t,i+1}
\bigotimes_{k=1}^m \prod_{\ell=1}^n 
v_{\ell}^{d_{k\ell} }v_{-\ell}^{1 - \varepsilon_{k \ell}
+ \delta_{k,t}( \delta_{\ell, i+1} - \delta_{\ell, i} )
}
\end{align*}
We note now that, for \(i,t\) such that \(\varepsilon_{t,i+1} = 1\), \(\varepsilon_{t,i}=0\), we have
\begin{align*}
\kappa(\tilde{\bbep}^{(t)}) - \kappa(\bbep) &=\sum_{k<t} (1-\varepsilon_{k,i+1}) - \sum_{k>t} (1- \varepsilon_{k,i})
\\
\chi(\tilde{\bbep}^{(t)}) - \chi(\bbep) &= -m,
\end{align*}
so we have
\begin{align*}
\kappa(\tilde{\bbep}^{(t)}) - \kappa(\bbep) + \beta(\bbep,t) \equiv m - 1 \mod 2,
\end{align*}
and therefore the action of \(\cE_i\) commutes with the map \(\varphi_{m,n}\).

(\(\cF_i\) action) We have that \( \phi_{m,n} \cdot \cF_i \left(
\bigotimes_{\ell = 1}^n \prod_{k = 1}^m v_k^{d_{k\ell}} v_{-k}^{\varepsilon_{k \ell}}
\right)\) is equal to 
\begin{align*}
 &
  \phi_{m,n} \left( \sum_{t=1}^m   d_{t,i}  \bigotimes_{\ell=1}^n \prod_{k=1}^m     v_k^{d_{k\ell} + \delta_{k,t}(  \delta_{\ell,i+1}  - \delta_{\ell,i}     )}v_{-k}^{\varepsilon_{k \ell}}\right)\\
  &+
    \phi_{m,n} \left( \sum_{t=1}^m  (-1)^{\beta(\bbep,t)} \varepsilon_{t,i}(1-\varepsilon_{t,i+1})  \bigotimes_{\ell=1}^n \prod_{k=1}^m     v_k^{d_{k\ell}    }v_{-k}^{\varepsilon_{k \ell }   + \delta_{k,t}(  \delta_{\ell,i+1}  - \delta_{\ell,i} )  }\right)\\
    &=
      (-1)^{\kappa(\bbep) + \chi(\bbep)  } 
 \sum_{t=1}^m   d_{t,i} \bigotimes_{k=1}^m  \prod_{\ell=1}^n     v_\ell^{d_{k\ell} + \delta_{k,t}(  \delta_{\ell,i+1}  - \delta_{\ell,i}     )}v_{-\ell}^{1-\varepsilon_{k \ell}}\\
 &+
\sum_{t=1}^m  (-1)^{\beta(\bbep,t)  +\kappa(\tilde{\bbep}^{(t)})  + \chi(\tilde{\bbep}^{(t)})  } \varepsilon_{t,i}(1-\varepsilon_{t,i+1})   \bigotimes_{k=1}^m  \prod_{\ell=1}^n  v_\ell^{d_{k\ell}    }v_{-\ell}^{1-\varepsilon_{k \ell }   - \delta_{k,t}(  \delta_{\ell,i+1}  - \delta_{\ell,i} )  }\\
\end{align*}
where
\begin{align*}
\beta(\bbep,t) = \sum_{k<t} \varepsilon_{  k ,i+1} + \sum_{k>t} \varepsilon_{k,i}
\end{align*}
and
\begin{align*}
\tilde{\varepsilon}^{(t)}_{k \ell} = 
\begin{cases}
\varepsilon_{t,i+1} & \textup{if }k = t, \ell = i;\\
\varepsilon_{t,i} & \textup{if }k = t, \ell = i+1;\\
\varepsilon_{k \ell}  & \textup{otherwise}.
\end{cases}
\end{align*}
We also have that \(\cF_i \cdot \phi_{m,n}\left( \bigotimes_{\ell = 1}^n \prod_{k = 1}^m v_k^{d_{k\ell}} v_{-k}^{\varepsilon_{k \ell}} \right)\) is equal to
\begin{align*}
&
\cF_i \cdot \left(  (-1)^{\kappa(\bbep) + \chi(\bbep)}
\bigotimes_{k=1}^m \prod_{\ell=1}^n v_\ell^{d_{k\ell}}v_{-\ell}^{1 - \varepsilon_{k \ell}} \right)\\
&=
(-1)^{\kappa(\bbep) + \chi(\bbep)}
\sum_{t=1}^m 
d_{t, i}
\bigotimes_{k=1}^m \prod_{\ell=1}^n 
v_{\ell}^{d_{k\ell} + \delta_{k,t}( \delta_{\ell, i+1} - \delta_{\ell, i} )}v_{-\ell}^{1 - \varepsilon_{k \ell}}\\
&+
(-1)^{\kappa(\bbep) + \chi(\bbep) + 1}
\sum_{t=1}^m
(1-\varepsilon_{t,i+1})\varepsilon_{t,i}
\bigotimes_{k=1}^m \prod_{\ell=1}^n 
v_{\ell}^{d_{k\ell} }v_{-\ell}^{1 - \varepsilon_{k \ell}
+ \delta_{k,t}( \delta_{\ell, i} - \delta_{\ell, i+1} )
}
\end{align*}
We note now that, for \(i,t\) such that \(\varepsilon_{t,i+1} = 0\), \(\varepsilon_{t,i}=1\), we have
\begin{align*}
\kappa(\tilde{\bbep}^{(t)}) - \kappa(\bbep) &= \sum_{k>t} (1- \varepsilon_{k,i})
-\sum_{k<t} (1-\varepsilon_{k,i+1})
\\
\chi(\tilde{\bbep}^{(t)}) - \chi(\bbep) &= m,
\end{align*}
and therefore
\begin{align*}
\kappa(\tilde{\bbep}^{(t)}) - \kappa(\bbep) + \beta(\bbep,t) \equiv m - 1 \mod 2,
\end{align*}
and therefore the action of \(\cF_i\) commutes with the map \(\varphi_{m,n}\).

(\(\cB_i\) action) We have that \( \phi_{m,n} \cdot \cB_i  \left(
\bigotimes_{\ell = 1}^n \prod_{k = 1}^m v_k^{d_{k\ell}} v_{-k}^{\varepsilon_{k \ell}}
\right)\) is equal to 
\begin{align*}
 &
  \phi_{m,n} \left( \sum_{t=1}^m (1- \varepsilon_{t,i+1})
  (-1)^{\mu(\bbep,i) + \xi_1(\bbep,t,i)}
   \bigotimes_{\ell=1}^n \prod_{k=1}^m  v_k^{d_{k\ell} +\delta_{t,k} \delta_{\ell,i} }v_{-k}^{\varepsilon_{k \ell} + \delta_{t,k}\delta_{\ell, i+1}}\right)\\
   &\hspace{0.5in}+
     \phi_{m,n} \left( \sum_{t=1}^m (1- \varepsilon_{t,i})
  (-1)^{\mu(\bbep,i) + \xi_2(\bbep,t,i)}
   \bigotimes_{\ell=1}^n \prod_{k=1}^m  v_k^{d_{k\ell} +\delta_{t,k} \delta_{\ell,i+1} }v_{-k}^{\varepsilon_{k \ell} + \delta_{t,k}\delta_{\ell, i}}\right)\\
   &=
 \sum_{t=1}^m (1- \varepsilon_{t,i+1})
  (-1)^{\mu(\bbep,i) + \xi_1(\bbep,t,i) + \kappa(\bbep^{(t,i+1)}) + \chi(\bbep^{(t,i+1)})}
 \bigotimes_{k=1}^m  \prod_{\ell=1}^n   v_\ell^{d_{k\ell} +\delta_{t,k} \delta_{\ell,i} }v_{-\ell}^{1-\varepsilon_{k \ell} - \delta_{t,k}\delta_{\ell, i+1}}\\
    &\hspace{0.5in}+
 \sum_{t=1}^m (1- \varepsilon_{t,i})
  (-1)^{\mu(\bbep,i) + \xi_2(\bbep,t,i) + \kappa(\bbep^{(t,i)}) + \chi(\bbep^{(t,i)})}
 \bigotimes_{k=1}^m  \prod_{\ell=1}^n v_\ell^{d_{k\ell} +\delta_{t,k} \delta_{\ell,i+1} }v_{-\ell}^{1- \varepsilon_{k \ell} - \delta_{t,k}\delta_{\ell, i}}
\end{align*}
where
\begin{align*}
\mu(\bbep, i) = \sum_{\substack{1 \leq \ell < i+1 \\ 1 \leq k \leq m}} \varepsilon_{k \ell}
\qquad
\xi_1(\bbep,t,i) = \sum_{k <t} \varepsilon_{k,i+1}
\qquad
\xi_2(\bbep,t,i) = \sum_{k>t} \varepsilon_{k,i}.
\end{align*}
and 
  \begin{align*}
  \varepsilon_{k \ell}^{(t,j)}
  =
  \begin{cases}
  \varepsilon_{k \ell} +1 & \textup{if } k=t, \ell = j;\\
   \varepsilon_{k \ell} & \textup{otherwise}.
  \end{cases}
  \end{align*}
We also have that \((-1)^{mn} \cB_i \cdot \phi_{m,n}\left( \bigotimes_{\ell = 1}^n \prod_{k = 1}^m v_k^{d_{k\ell}} v_{-k}^{\varepsilon_{k \ell}} \right)\) is equal to
  \begin{align*}
&
\cB_i \cdot 
(-1)^{\kappa(\bbep) + \chi(\bbep)+mn}
\bigotimes_{k=1}^m \prod_{\ell=1}^n v_\ell^{d_{k\ell}}v_{-\ell}^{1 - \varepsilon_{k \ell}}\\
&= 
(-1)^{\kappa(\bbep) + \chi(\bbep)+mn}
\sum_{t =1}^m 
(1-\varepsilon_{t,i+1})
(-1)^{\rho(\bbep,t,i+1)}
\bigotimes_{k=1}^m \prod_{\ell=1}^n
v_\ell^{d_{k\ell}+ \delta_{t,k}\delta_{\ell,i}}v_{-\ell}^{1 - \varepsilon_{k \ell} - \delta_{t,k}\delta_{\ell,i+1}}\\
&\hspace{0.5in}+
(-1)^{\kappa(\bbep) + \chi(\bbep)+mn}
\sum_{t =1}^m 
(1-\varepsilon_{t,i})
(-1)^{\rho(\bbep,t,i)}
\bigotimes_{k=1}^m \prod_{\ell=1}^n
v_\ell^{d_{k\ell}+ \delta_{t,k}\delta_{\ell,i+1}}v_{-\ell}^{1 - \varepsilon_{k \ell} - \delta_{t,k}\delta_{\ell,i}},
  \end{align*}
where
\begin{align*}
\rho(\bbep, t,j) = \sum_{ \substack{ k < t \\ 1 \leq \ell \leq n  }}(1- \varepsilon_{k \ell})
+
\sum_{ \ell < j} (1- \varepsilon_{t,\ell}).
\end{align*}
When \(\varepsilon_{t, i+1} = 0\), we have:
\begin{align*}
\kappa(\bbep)- \kappa(\bbep^{(t,i+1)})  
&=
\sum_{\substack{ k>t\\ \ell < i+1  }}(1- \varepsilon_{k,\ell})
+
\sum_{\substack{ k<t \\ \ell > i+1}}(1-\varepsilon_{k, \ell})\\
\chi(\bbep^{(t,i+1)}) - \chi(\bbep) 
&=
m(i+1) + m + mn+ t + 1.
\end{align*}
Then it is straightforward to check that 
\begin{align*}
\mu(\bbep,i) + \xi_1(\bbep, t, i) + \kappa(\bbep^{(t,i+1)}) + \chi(\bbep^{(t,i+1)}) 
+\kappa(\bbep) + \chi(\bbep) + \rho(\bbep, t, i+1) +mn \equiv 0 \mod 2.
\end{align*}
When \(\varepsilon_{t,i} = 0\), we have:
\begin{align*}
\kappa(\bbep)- \kappa(\bbep^{(t,i)})  
&=
\sum_{\substack{ k>t\\ \ell < i  }}(1- \varepsilon_{k,\ell})
+
\sum_{\substack{ k<t \\ \ell > i}}(1-\varepsilon_{k, \ell})\\
\chi(\bbep^{(t,i)}) - \chi(\bbep) 
&=
mi + m + mn+t + 1.
\end{align*}
Then it is straightforward to check that 
\begin{align*}
\mu(\bbep,i) + \xi_2(\bbep, t, i) + \kappa(\bbep^{(t,i)}) + \chi(\bbep^{(t,i)}) 
+\kappa(\bbep) + \chi(\bbep) + \rho(\bbep, t, i) +mn\equiv 0 \mod 2.
\end{align*}
so the action of \(\cB_i\) commutes with the map \(\varphi_{m,n}\).

(\(\cC_i\) action)
We have that \(  \phi_{m,n} \cdot \cC_i  \left(
\bigotimes_{\ell = 1}^n \prod_{k = 1}^m v_k^{d_{k\ell}} v_{-k}^{\varepsilon_{k \ell}}
\right)\) is equal to 
\begin{align*}
&
  \phi_{m,n} \left( \sum_{t=1}^m   d_{t,i} \varepsilon_{t,i+1}
  (-1)^{\tau(\bbep,t) + \gamma_1(\bbep,t)}
   \bigotimes_{\ell=1}^n \prod_{k=1}^m     v_k^{d_{k\ell} -\delta_{t,k} \delta_{\ell,i} }v_{-k}^{\varepsilon_{k \ell} - \delta_{t,k}\delta_{\ell, i+1}}\right)\\
   &\hspace{0.5in}+ 
     \phi_{m,n} \left( \sum_{t=1}^m   d_{t,i+1} \varepsilon_{t,i}
  (-1)^{\tau(\bbep,t) + \gamma_2(\bbep,t)}
   \bigotimes_{\ell=1}^n \prod_{k=1}^m     v_k^{d_{k\ell} - \delta_{t,k}\delta_{\ell,i+1} }v_{-k}^{\varepsilon_{k \ell} -\delta_{t,k} \delta_{\ell, i}}\right)\\
   &=
   \sum_{t=1}^m   d_{t,i} \varepsilon_{t,i+1}
  (-1)^{\tau(\bbep,t) + \gamma_1(\bbep,t) + \kappa(\bbep^{(t,i+1)}) + \chi(\bbep^{(t,i+1)})}
   \bigotimes_{k=1}^m    \prod_{\ell=1}^n  v_\ell^{d_{k\ell} - \delta_{t,k}\delta_{\ell,i} }v_{-\ell}^{1-\varepsilon_{k \ell} + \delta_{t,k}\delta_{\ell, i+1}}\\
   &\hspace{0.5in}+
   \sum_{t=1}^m   d_{t,i+1} \varepsilon_{t,i}
  (-1)^{\tau(\bbep,t) + \gamma_2(\bbep,t) + \kappa(\bbep^{(t,i)}) + \chi(\bbep^{(t,i)})}
 \bigotimes_{k=1}^m    \prod_{\ell=1}^n    v_\ell^{d_{k\ell} - \delta_{t,k}\delta_{\ell,i+1} }v_{-\ell}^{1-\varepsilon_{k \ell} +\delta_{t,k} \delta_{\ell, i}}   
  \end{align*}
  where
  \begin{align*}
  \tau(\bbep,t) = \sum_{\substack{1 \leq k \leq m \\ \ell <i+1}} \varepsilon_{k \ell}
  \qquad
  \gamma_1(\bbep,t) =  \sum_{k<t} \varepsilon_{k,i+1}
  \qquad
  \gamma_2(\bbep,t) = \sum_{k>t} \varepsilon_{k,i}
  \end{align*}
  and 
  \begin{align*}
  \varepsilon_{k \ell}^{(t,j)}
  =
  \begin{cases}
  \varepsilon_{k \ell} -1 & \textup{if } k=t, \ell = j;\\
   \varepsilon_{k \ell} & \textup{otherwise}.
  \end{cases}
  \end{align*}
We also have that \((-1)^{mn}\cC_i \cdot \phi_{m,n}\left( \bigotimes_{\ell = 1}^n \prod_{k = 1}^m v_k^{d_{k\ell}} v_{-k}^{\varepsilon_{k \ell}} \right)\) is equal to
\begin{align*}
&
\cC_i \cdot 
(-1)^{\kappa(\bbep) + \chi(\bbep)+mn}
\bigotimes_{k=1}^m \prod_{\ell=1}^n v_\ell^{d_{k\ell}}v_{-\ell}^{1 - \varepsilon_{k \ell}}\\
&=(-1)^{\kappa(\bbep) + \chi(\bbep)+mn} \sum_{t=1}^m d_{t,i} \varepsilon_{t,i+1} 
(-1)^{\zeta(\bbep, i ,t)}
\bigotimes_{k=1}^m \prod_{\ell=1}^n
v_{\ell}^{d_{k\ell} - \delta_{t,k}\delta_{\ell,i}}
v_{\ell}^{1-\varepsilon_{k \ell} + \delta_{t,k}\delta_{\ell, i+1}}\\
&
\hspace{0.5in}
+(-1)^{\kappa(\bbep) + \chi(\bbep)+mn} \sum_{t=1}^m d_{t,i+1} \varepsilon_{t,i} 
(-1)^{\zeta(\bbep, i ,t)+1}
\bigotimes_{k=1}^m \prod_{\ell=1}^n
v_{\ell}^{d_{k\ell} - \delta_{t,k}\delta_{\ell,i+1}}
v_{\ell}^{1-\varepsilon_{k \ell} + \delta_{t,k}\delta_{\ell, i}}
\end{align*}
where
\begin{align*}
\zeta(\bbep, i, t) = \sum_{\substack{ 1 \leq \ell \leq n \\ k < t   }} (1-\varepsilon_{k \ell})
+
\sum_{1 \leq \ell < i+1}(1-\varepsilon_{t \ell}).
\end{align*}
When \(\varepsilon_{t,i+1} = 1\), we have
\begin{align*}
\kappa( \bbep^{(t,i+1)}) - \kappa(\bbep) &= \sum_{\substack{k<t \\ \ell > i+1}} (1- \varepsilon_{k \ell}) + \sum_{\substack{k>t \\ \ell < i+1}} (1- \varepsilon_{k \ell})\\
\chi(\bbep) - \chi( \bbep^{(t,i+1)}) &= m(i+1) + m + mn+t + 1.
\end{align*}
Then it is straightforward to check that
\begin{align*}
\tau(\bbep, t) + \gamma_1(\bbep, t) + \kappa( \bbep^{(t,i+1)}) - \kappa(\bbep) + \chi(\bbep) - \chi( \bbep^{(t,i+1)}) + \zeta(\bbep, i, t) +mn\equiv 0 \mod 2.
\end{align*}
When \(\varepsilon_{t,i} = 1\), we have
\begin{align*}
\kappa( \bbep^{(t,i)}) - \kappa(\bbep) &= \sum_{\substack{k>t \\ \ell < i}} (1- \varepsilon_{k \ell}) + \sum_{\substack{k<t \\ \ell >i}} (1- \varepsilon_{k \ell})\\
\chi(\bbep) - \chi( \bbep^{(t,i)}) &= mi + m +mn+ t + 1.
\end{align*}
Then it is straightforward to check that
\begin{align*}
\tau(\bbep, t) + \gamma_2(\bbep, t) + \kappa( \bbep^{(t,i)}) - \kappa(\bbep) + \chi(\bbep) - \chi( \bbep^{(t,i)}) + \zeta(\bbep, i, t) +mn+1\equiv 0 \mod 2.
\end{align*}
so the action of \(\cC_i\) commutes with the map \(\varphi_{m,n}\).
\end{answer}
\end{proof}

\subsection{Categorical Howe Duality}\label{SS:CategoricalDoubleCentralizer}

In this section we assume all categories, functors, natural transformations, etc.\ are enriched over $\sVec$.   Given supercategories $\mathcal{C}$ and $\mathcal{D}$, let $\Fun(\mathcal C, \mathcal D)$ denote the category whose objects are even superfunctors $\mathcal{C} \to \mathcal{D}$ and whose morphisms are supernatural transformations.   See \cite{BE} for details on these notions in the graded setting.  As a special case, if $\mathcal C$ is a supercategory and $\sVec$ is the category of superspaces, then $\mathcal C\text{-Rep} := \Fun(\mathcal C,\sVec)$ is the category of representations of $\mathcal{C}$.  

\begin{example} \label{R:representations} Consider the following example.  Let $A$ be an associative, unital superalgebra, and let $\mathcal{C}_{A}$ be the supercategory with a single object $\bullet$, and $\Hom_{\mathcal{C}_{A}}(\bullet,\bullet) = A$, with composition given by multiplication in $A$.  Then, a representation of $\mathcal{C}_{A}$ is a functor $F : \mathcal{C}_{A} \to \sVec$.  This is the data of a superspace $V = F(\bullet)$, along with a superalgebra homomorphism $A \to \End_{\k}(V)$.  Hence, a representation of $\mathcal{C}_{A}$ is the same data as an $A$-module.

To continue the example, given two representations $F$ and $G$ of $\mathcal{C}_{A}$, let $\eta : F \Rightarrow G$ be a homogeneous supernatural transformation.   Then, $\eta_\bullet : F(\bullet) \to G(\bullet)$ is a homogeneous linear map which satisfies the additional assumption that for all homogeneous morphisms $a : \bullet \to \bullet$ the following equality holds:
$$
G(a) \circ \eta_{\bullet} = (-1)^{\p{\eta_{\bullet}}\p{a}} \eta_{\bullet} \circ F(a).
$$
Hence, a homogeneous supernatural transformation is precisely the data of a homogeneous intertwiner $F(\bullet) \to G(\bullet)$ of $A$-supermodules.  It is not hard to see that the category of representations of $\mathcal{C}_{A}$ is equivalent to the category of $A$-supermodules.
\end{example}

As explained in \cite[Section 1.1]{BE}, given supercategories $\mathcal{C}_{1}$ and $\mathcal{C}_{2}$ we can form the supercategory $\mathcal{C}_{1} \boxtimes \mathcal{C}_{2}$.  The objects of this category are given by 
\[
\left\{ (x_{1}, x_{2})\mid \text{$x_{i}$ is an object of $\mathcal{C}_{i}$}\right\}
\]
 and the morphisms are given by 
\[
\Hom_{\mathcal{C}_{1}\boxtimes \mathcal{C}_{2}}\left((x_{1}, x_{2}), (y_{1}, y_{2}) \right) = \Hom_{\mathcal{C}_{1}}(x_{1}, y_{1}) \otimes \Hom_{\mathcal{C}_{2}}(x_{2}, y_{2}).
\] Composition is given on homogeneous morphisms by 
\begin{equation}\label{E:composition}
(f_{1} \otimes f_{2}) \circ (g_{1} \otimes g_{2}) = (-1)^{\p{f_{2}}\p{g_{1}}} (f_{1} \circ g_{1}) \otimes (f_{2} \circ g_{2}).
\end{equation} 

For any supercategories $\mathcal{C}_1$ and $\mathcal{C}_2$, there is a functor 
\[
\iota^{\prime} : \mathcal{C}_1 \to \Fun(\mathcal{C}_2, \mathcal{C}_1 \boxtimes \mathcal{C}_2)
\]
given as follows.  First, for an object $x_1$ in $\mathcal{C}_1$ associate a functor $\iota^{\prime}(x_1) =  \iota^{\prime}_{x_1} : \mathcal{C}_2 \to \mathcal{C}_1 \boxtimes \mathcal{C}_2$ given by $\iota^{\prime}_{x_1}(x_2) = (x_1, x_2)$ for all objects $x_{2}$ in $\mathcal C_2$ and by $\iota^{\prime}_{x_1}(g) = 1_{x_1} \otimes g$ for all morphisms $g \in \Hom_{\mathcal{C}_2}(x_2, y_2)$.  Next we define the functor $\iota^{\prime}$ on morphisms.  Given any morphism $f : x_1 \to y_1$ in $\mathcal{C}_1$ construct a supernatural transformation $\iota^{\prime}(f) : \iota^{\prime}_{x_1} \Rightarrow \iota^{\prime}_{y_1}$ by setting, for any object $x_2 \in \mathcal{C}_2$, the morphism $\iota^{\prime}(f)_{x_2}=f \otimes 1_{x_2} : \iota^{\prime}_{x_1}(x_{2})  \to \iota^{\prime}_{y_1}(x_{2})$.  Making the obvious changes yields an analogous functor 
\[
\iota^{\prime \prime} : \mathcal{C}_2 \to \Fun(\mathcal{C}_1, \mathcal{C}_1 \boxtimes \mathcal{C}_2).
\]

\begin{example} \label{R:tensorproduct}  Let $A_{1}$ and $A_{2}$ be associative unital $\k$-superalgebras. For $i=1, 2$ let $\mathcal{C}_{i}$ be the supercategory with one object, $\bullet_{i}$, with $\End_{\mathcal{C}_{i}}(\bullet_{i}) = A_{i}$, and with composition given by superalgebra multiplication.  Then, $\mathcal{C}_1 \boxtimes \mathcal{C}_2$ is a supercategory with a single object $(\bullet_1, \bullet_2)$ whose endomorphism superalgebra is $A_1 \otimes A_2$.  When applied to morphisms, the functor $\iota^{\prime}_{\bullet_{1}} : \mathcal{C}_2 \to  \mathcal{C}_1 \boxtimes \mathcal{C}_2$ induces the superalgebra homomorphism $A_2 \to A_1 \otimes A_2$ given by $b \mapsto 1_{A_1} \otimes b$.  Similarly, $\iota^{\prime \prime}_{\bullet_{2}} : \mathcal{C}_1 \to \mathcal{C}_1 \boxtimes \mathcal{C}_2$ gives rise to the superalgebra homomorphism $A_1 \to A_1 \otimes A_2$ given by $a \mapsto a \otimes 1_{A_2}$.   
\end{example}

Next, assume that we are given a representation $F : \mathcal{C}_1 \boxtimes \mathcal{C}_2 \to \sVec$.  There are then functors $F_{*}:\Fun(\mathcal{C}_2 , \mathcal{C}_1 \boxtimes \mathcal C_2) \to \mathcal{C}_2\text{-Rep}$ and $F_{*}:\Fun(\mathcal{C}_1 , \mathcal{C}_1 \boxtimes \mathcal C_2) \to \mathcal{C}_1\text{-Rep}$ given in both cases by  $H \mapsto F \circ H$.  In turn, there are functors
\begin{align*}
\mathbf{F}_1 &: \mathcal{C}_1 \to \mathcal{C}_2\text{-Rep}, \\
\mathbf{F}_2 &: \mathcal{C}_2 \to \mathcal{C}_1\text{-Rep},
\end{align*}
which are given by $\mathbf{F}_1 = F_* \circ \iota^{\prime}$ and $\mathbf{F}_2 = F_* \circ \iota^{\prime \prime}$.  
We say the representation $F$ \emph{left centralizes} if the functor $\mathbf{F}_{1}$ is full, it \emph{right centralizes} if the functor $\mathbf{F}_{2}$ is full, and it \emph{has the double centralizer property} if both $\mathbf{F}_{1}$ and $\mathbf{F}_{2}$ are full.

\begin{example} Let $A_i$ and $\mathcal{C}_i$ be as in \cref{R:tensorproduct} and recall from the discussion in \cref{R:representations} that a representation $F: \mathcal{C}_{1}\boxtimes \mathcal{C}_{2} \to \sVec$ is the same data as an $A_1 \otimes A_2$-module structure on the superspace $V = F(\bullet_{1}, \bullet_{2})$.

As discussed above, the functor $\mathbf{F}_{1}$ can be used to obtain a representation $\mathbf{F}_1(\bullet_{1}): \mathcal{C}_{2} \to \sVec$.   Translating into module-theoretic terms, this gives an $A_2$-module structure on $V$ where $b \in A_2$ acts on $V$ as the linear map $F(1_{A_1} \otimes b)$.   By definition,  $F$ left centralizes if every supernatural transformation $\mathbf{F}_1(\bullet_{1}) \Rightarrow \mathbf{F}_1(\bullet_{1})$ can be written as $\mathbf{F}_1(a)$ for some $a \in A_1$.  This is equivalent to the statement that every $A_2$-module endomorphism of $V$ can be expressed as $F(a \otimes 1_{A_2})$ for some $a \in A_1$.  That is, the image of $A_1$ in $\End_{\k}(V)$ is the centralizer of the image of $A_2$.

Similarly, the fullness of $\mathbf{F}_2 : \mathcal{C}_2 \to \mathcal{C}_1 \text{-Rep}$ is equivalent to the statement that the image of $A_2$ in $\End_{\k}(V)$ is the centralizer of the image of $A_1$.   Hence, if $F$ satisfies the double centralizer property, then the images of $A_1$ and $A_2$ in $\End_{\k}(V)$ are mutually centralizing.
\end{example}

\begin{lemma}\label{L:RepCreation}  For all $m,n \geq 1$ there exist representations $F,F': \bU (\tfp (m)) \boxtimes \bU (\fp (n)) \to \sVec$ given on $\tuplambda \in X(T_{m})$ and $\tupmu \in X(T_{n})$ by 
\begin{align*}
F(\tuplambda, \tupmu) &= S^{\tuplambda}\left(V_{n} \right)_{\tupmu},\\
F'(\tuplambda, \tupmu) &= S^{\tupmu}\left(W_{m} \right)_{\tuplambda}.
\end{align*}
 The representation $F$ is left centralizing and the representation $F'$ is right centralizing.
\end{lemma}

\begin{proof} The existence of the representations comes from the following general construction.

Let $\mathcal{C}_{1}$ and $\mathcal{C}_{2}$ be small supercategories and let $C_{1}$ and $C_{2}$ be the corresponding locally unital superalgebras; that is, $C_{i} = \bigoplus_{x,y \in \text{ob}(\mathcal{C}_i)} \Hom_{\mathcal{C}_i}(x,y)$, with multiplication operation given by composition when it is defined and zero otherwise.   If we write $1_{x}$ for the identity morphism of an object $x$, then $\left\{1_{x_{1}} \mid x_1 \in \operatorname{ob}(\mathcal{C}_1) \right\}$ and $\left\{1_{x_{2}} \mid  x_2 \in \text{ob}(\mathcal{C}_2)  \right\}$ are distinguished systems of orthogonal idempotents in $C_1$ and $C_2$, respectively.  Let $V$ be a locally unital supermodule for both $C_{1}$ and $C_{2}$, meaning that $V$ is a supermodule over both superalgebras and also satisfies
\[
V = \bigoplus_{x_1 \in \text{ob}(\mathcal{C}_1)} 1_{x_1} V = \bigoplus_{x_2 \in \text{ob}(\mathcal{C}_2)} 1_{x_2} V.
\] We also assume the actions of $C_1$ and $C_2$ mutually commute in the graded sense.

The data above defines a representation $F: \mathcal{C}_{1} \boxtimes \mathcal{C}_{2} \to \sVec$ as follows.  Set 
\[
F(x_{1}, x_{2}) = 1_{x_{1}}1_{x_{2}}V = 1_{x_{2}}1_{x_{1}}V
\] for all objects $(x_{1}, x_{2})$ of $\mathcal{C}_{1} \boxtimes \mathcal{C}_{2}$.  Given homogeneous morphisms $g_{1}: x_{1} \to y_{1}$ in $\mathcal{C}_{1}$ and $g_{2}: x_{2} \to y_{2}$ in $\mathcal{C}_{2}$, set $F(g_{1} \otimes g_{2}): 1_{x_{1}}1_{x_{2}}V \to 1_{y_{1}}1_{y_{2}}V$ by
\begin{equation*}
F(g_{1} \otimes g_{2})(v) = g_{1}. g_{2}.(v)=(-1)^{\p{g_{1}}\p{g_{2}}}g_{2}. g_{1}.(v).
\end{equation*}
That this defines a representation of $\mathcal{C}_{1} \boxtimes \mathcal{C}_{2}$ is an exercise in checking the definitions, which we leave to the reader.

To obtain the result from this general construction, set $\mathcal{C}_1 = \bU(\tfp(m))$ and $\mathcal{C}_2 = \bU(\fp(n))$, and set $C_1 = \dot{U}(\tfp(m))$ and $C_2 = \dot{U}(\fp(n))$.  The functor $F$ in the statement of the lemma is obtained by applying the construction above to the locally unital module $\mathcal{A}$, while the functor $F'$ is obtained by applying the construction to the module $\mathcal{A}'$.  

The assertion that $F$ is left centralizing is a consequence of \cref{T:CategoricalHoweDualityI} and a definition chase.  Likewise, the assertion that $F'$ is right centralizing follows from \cref{T:CategoricalHoweDualityII}.
\end{proof}

An examination of the defining relations for $\bU (\fp (n))$ shows it admits a self-equivalence 
\[
S: \bU (\fp (n)) \to \bU (\fp (n)).
\]
Namely, on objects:  if $\tuplambda \in X(T_{n})$, then $S(\tuplambda ) = \tuplambda + \tupdelta_{n}$.  On morphisms: if $\cX 1_{\tuplambda}$ is any of the generating morphisms, then $S(\cX 1_{\tuplambda})=\cX 1_{\tuplambda + \tupdelta_{n}}$.  There is a similar self-equivalence $\tilde{S}: \bU (\tfp (m)) \to \bU (\tfp (m))$. On objects:  if $\tuplambda \in X(T_{m})$, then $\tilde{S}(\tuplambda ) = \tuplambda + \tupdelta_{m}$.  On morphisms: if $\tcX 1_{\tuplambda}$ is any of the generating morphisms, then $\tilde{S}(\tcX 1_{\tuplambda})=\tcX 1_{\tuplambda + \tupdelta_{m}}$. 

We can twist a representation $G: \bU (\tfp (m)) \boxtimes \bU (\fp (n)) \to \sVec$  by precomposing with $S$ and/or $\tilde{S}$ (or their inverses) some number of times.  Namely, given any integers $a,b$ the functor $G \circ \tilde{S}^{a} \boxtimes S^{b}$ is again a representation of $ \bU (\tfp (m)) \boxtimes \bU (\fp (n))$.

\begin{theorem}\label{T:TheRealCategoricalHoweDuality}  For all $m,n \geq 1$ the representations
\begin{align*}
F \circ \left(\Id_{\bU (\tfp (m))} \boxtimes S^{-m}  \right)&:  \bU (\tfp (m)) \boxtimes \bU (\fp (n)) \to \sVec \\
F' \circ \left( \tilde{S}^{-n} \boxtimes \Id_{\bU (\fp (n))} \right) &:  \bU (\tfp (m)) \boxtimes \bU (\fp (n)) \to \sVec
\end{align*}
are naturally isomorphic via a homogeneous supernatural transformation of parity $\overline{mn}$.
\end{theorem}

\begin{proof} This is essentially a reformulation of \cref{T:IntertwiningActionsCategorical}.  First, observe that since $\mathcal{A} \otimes \k_{m\tupdelta_{n}}$  is a module for $\dot{U}(\tfp (m))$ and $\dot{U}(\fp (n))$ and the two actions commute, the construction in the proof of \cref{L:RepCreation} defines a representation of $\bU (\tfp (m)) \boxtimes \bU (\fp (n))$ and this is precisely $F \circ \left(\Id_{\bU (\tfp (m))} \boxtimes S^{-m} \right)$.  Likewise, $\mathcal{A}'_{m,n} \otimes \tilde{\k}_{n\tupdelta_{m}}$  defines  $F' \circ \left( \tilde{S}^{-n} \boxtimes \Id_{\bU (\fp (n))} \right)$.

Since $\phi_{m,n} \otimes \eta_{m,n}$ is a module isomorphism for both $\dot{U}(\tfp (m))$ and $\dot{U}(\fp (n))$, it preserves weight spaces for both superalgebras.  In particular, for each $\tuplambda_{1} \in X(T_{m})$ and  $\tuplambda_{2} \in X(T_{n})$ the map restricts to define a  superspace map
\[
\xi_{(\tuplambda_{1}, \tuplambda_{2})}: 1_{\tuplambda_{1}}1_{\tuplambda_{2}}\left( \mathcal{A}'_{m,n} \otimes \tilde{\k}_{n\tupdelta_{m}}\right)  \to 1_{\tuplambda_{1}}1_{\tuplambda_{2}} \left( \mathcal{A} \otimes \k_{m\tupdelta_{n}} \right).
\] That is, there is a family of superspace maps 
\[
\xi_{(\tuplambda_{1}, \tuplambda_{2})}: F'(\tuplambda_{1}-n\tupdelta_{m}, \tuplambda_{2})  \to   F(\tuplambda_{1}, \tuplambda_{2}-m\tupdelta_{n})
\] for all pairs $(\tuplambda_{1}, \tuplambda_{2}) \in X(T_{m}) \times X(T_{n})$.  Since  $\phi_{m,n} \otimes \eta_{m,n}$ is homogeneous of parity $\overline{mn}$, then so is each $\xi_{(\tuplambda_{1}, \tuplambda_{2})}$.

The fact that $\xi =\left(\xi_{(\tuplambda_{1}, \tuplambda_{2})} \right)$ is indeed a supernatural transformation can be verified using the fact that $\phi_{m,n} \otimes \eta_{m,n}$ is a module homomorphism for both $\dot{U}(\tfp (m))$ and $\dot{U}(\fp (n))$.  Finally, because $\phi_{m,n} \otimes \eta_{m,n}$ is a superspace isomorphism, the restrictions $\xi_{(\tuplambda_{1}, \tuplambda_{2})}$ are also superspace isomorphisms.  Taken all together, this shows $\xi$ is a supernatural isomorphism of parity $\overline{mn}$, as desired.
\end{proof}
The following result could be viewed as a categorical analogue of Howe duality.

\begin{theorem}\label{T:TheRealCategoricalHoweDuality2}  The representations  $F, F': \bU (\tfp (m)) \boxtimes \bU (\fp (n)) \to \sVec$ both have the double centralizer property.
\end{theorem}

\begin{proof}  Since $S$ and $\tilde{S}$ are self-equivalences, a representation $G: \bU (\tfp (m)) \boxtimes \bU (\fp (n)) \to \sVec$ is left centralizing, right centralizing, or has the double centralizer property if and only if $G \circ \left( \tilde{S}^{a} \boxtimes S^{b} \right)$ has the same property.  This observation along with the previous theorem means that the fact $F'$ is right centralizing means $F$ is right centralizing and hence $F$ has the double centralizer property.  Likewise for $F'$.
\end{proof}

\begin{remark} Theorem~C in the introduction follows from this result and a suitable precomposition with the Chevalley isomorphism $\Psi:\bU (\fp (m)) \to \bU (\tfp (m))$.
\end{remark}

\subsection{Finitary Howe Duality}\label{SS:FinitaryHoweDuality}
The intertwiner given in \cref{SS:IntertwiningActionsCategorical} also allows \cref{T:FinitaryHoweDualityI,T:FinitaryHoweDualityII} to be unified into the following finitary version of Howe duality.

\begin{theorem} \label{T:FinitaryHoweDuality} For every $m, n \geq 1$, the images of the maps $\tilde{\varrho}_{ }$ and $\varrho_{ }$ in $\End_{\k}^{\fin}\left(\mathcal{A}_{m,n} \right)$ given in \cref{T:FinitaryHoweDualityI} are mutually centralizing.

Likewise, for every $m, n \geq 1$, the images of the maps $\tilde{\varrho}'_{  }$ and $\varrho'_{  }$ in $\End_{\k}^{\fin}\left(\mathcal{A}'_{m,n} \right)$ given in \cref{T:FinitaryHoweDualityII} are mutually centralizing. 
\end{theorem}

\begin{proof}
Since $\mathcal{A}_{m,n} \otimes \k_{m \tupdelta_n}$ is module for $\dot{U}(\fp (n)$ and $\dot{U}(\tfp (m))$ there are superalgebra maps $\tau_{ }:\dot{U}(\fp (n)) \to \End_{\k}^{\fin}\left(\mathcal{A}_{m,n} \otimes \k_{m \tupdelta_n} \right)$ and  $\tilde{\tau}_{ }:\dot{U}(\tfp (m)) \to \End_{\k}^{\fin}\left(\mathcal{A}_{m,n} \otimes \k_{m \tupdelta_n} \right)$ (the fact that the image lies in $\End_{\k}^{\fin}\left(\mathcal{A}_{m,n} \otimes \k_{m \tupdelta_n} \right)$ is for the same reasons as in the proof of \cref{T:FinitaryHoweDualityI}).

First, consider the parity preserving isomorphism of endomorphism superalgebras
\[
\alpha_{1}: \End_{\k}^{\fin}\left( \mathcal{A}_{m,n}\right) \to \End_{\k}^{\fin}\left(\mathcal{A}_{m,n} \otimes \k_{m \tupdelta_n} \right)
\] given by $f \mapsto f \otimes 1$.  Since the action of $\dot{U}(\tfp (m))$ on $ \k_{m \tupdelta_n}$ is trivial, it is immediate that $\alpha_{1}$ is an isomorphism  between the image of $\tilde{\varrho}_{ }$ and the image of $\tilde{\tau}_{ }$, and between the centralizer of the image of $\tilde{\varrho}_{ }$ and the centralizer of the image of $\tilde{\tau}_{ }$.  Next, observe that $\alpha_{1}({\varrho}_{ }(X)) = {\tau}_{ }(X)$ for the generators of $\dot{U}(\fp (n))$ listed in \cref{D:Udot} (because they act trivially on $\k_{m \tupdelta_n}$) and  $\alpha_{1}({\varrho}_{ }(1_{\tuplambda})) = {\tau}_{ }(1_{\tuplambda+m\tupdelta_{n}})$ for all $\tuplambda \in X(T_{n})$.  That is, $\alpha_{1}$ is a bijection on the images of the generators of $\dot{U}(\fp (n))$ and, hence,  $\alpha_{1}$ is an isomorphism between the image of ${\varrho}_{ }$ and the image of $\tau_{ }$ and an isomorphism between the centralizer of the image of ${\varrho}_{ }$ and the centralizer of the image of ${\tau}_{ }$. 

There is also a parity preserving isomorphism of endomorphism superalgebras
\[
\alpha_{2}: \End_{\k}^{\fin}\left( \mathcal{A}'_{m,n}\right) \to \End_{\k}^{\fin}\left(\mathcal{A}'_{m,n} \otimes \k_{m \tupdelta_n} \right)
\] given by $f \mapsto f \otimes 1$.  The arguments of the previous paragraph apply equally well here and yield parallel results.

Third, there is a parity preserving isomorphism of superalgebras
\[
\upsilon: \End_{\k}^{\fin}\left(\mathcal{A}_{m,n} \otimes \k_{m \tupdelta_n} \right) \to \End_{\k}^{\fin}\left(\mathcal{A}'_{m,n} \otimes \k_{m \tupdelta_n} \right)
\] given by conjugation by $\phi_{m,n}\otimes \eta_{m,n}$.  Since $\phi_{m,n}\otimes \eta_{m,n}$ is a module isomorphism for both $\dot{U}(\tfp (m))$ and $\dot{U}(\fp (n))$ it is immediate that it defines an isomorphism between images and centralizers of images for the representations of  $\dot{U}(\tfp (m))$ and, likewise, $\dot{U}(\fp (n))$.

Combining the previous three paragraphs yields a parity preserving superalgebra isomorphism
\[
\alpha_{2}^{-1}\circ \upsilon \circ \alpha_{1} :\End_{\k}^{\fin}\left(\mathcal{A}_{m,n} \right) \to \End_{\k}^{\fin}\left(\mathcal{A}'_{m,n} \right)
\] which is an isomorphism between the images of $\varrho_{ }$ and $\varrho_{  }$, between the images of $\tilde{\varrho}_{ }$ and $\tilde{\rho}_{  }$, and between the centralizers of the images of these pairs of maps.  This in conjunction with \cref{T:FinitaryHoweDualityI,T:FinitaryHoweDualityII} proves the claim.

An identical argument applies to the images of the maps $\tilde{\varrho}'_{  }$ and $\varrho'_{  }$ in $\End_{\k}^{\fin}\left(\mathcal{A}'_{m,n} \right)$ given in \cref{T:FinitaryHoweDualityII}. 
\end{proof}

\begin{remark}  Theorem~B from the introduction follows immediately from this result and the canonical isomorphism 
\[
S(Y_{m}\otimes V_{n}) \cong  S(V_{n})^{\otimes m}.
\]  This is clearly a parity preserving isomorphism of $\dot{U}(\fp (n))$-modules.  The maps in Theorem~B are defined by conjugating the maps $\varrho_{ }$ and $\varrho'_{ }$ in \cref{T:FinitaryHoweDualityI}  by this isomorphism along with an appropriate composition with the Chevalley isomorphism $\tau: \dot{U}(\fp (m)) \to \dot{U}(\tfp (m))$.
\end{remark}

\section{Classical Howe Duality}\label{S:ClassicalHoweDuality}

\subsection{Invariant Polynomials}\label{SS:invariantPolynomials} 

Let $V_{n}$ be the natural module for $\fp (n)$ and let $Y_{m}$ be an $m$-dimensional purely even superspace with trivial $\fp (n)$ action and fixed  basis $\left\{ y_{i}\mid 1 \leq i \leq m \right\}$.  Then 
\begin{equation}\label{E:Udef}
U = U_{m,n}:= Y_{m}\otimes V_{n}
\end{equation} is a $\fp (n)$-module.  Recalling the fixed homogeneous basis $\left\{v_{j} \mid j \in I_{n|n} \right\}$ for $V_{n}$, the dual $V_{n}^{*}$ has a distinguished homogeneous basis $\left\{v_{j}^{*} \mid j \in I_{n|n} \right\}$ determined by $v_{j}^{*}(v_{k})=\delta_{j,k}$ with the parity of $v_{j}^{*}$ equal to the parity of $v_{j}$.  The superspace  $U$ has homogeneous basis $\left\{ z_{i,j} := y_{i} \otimes v_{j} \mid 1 \leq i \leq m, j \in I_{n|n}\right\}$ and $U^{*}$ has homogeneous basis  $\left\{ z_{i,j}^{*} \mid 1 \leq i \leq m, j \in I_{n|n} \right\}$ determined by  $z_{i,j}^{*}(z_{k,l}) = \delta_{i,k}\delta_{j,l}$.  The parities of both $z_{i,j}$ and $z_{i,j}^{*}$ equal the parity of $v_{j}$.

Consider the induced action of $\fp (n)$ on $S(U \oplus U^{*})$ and consider the subspace of $\fp (n)$-invariants, 
\[
S(U \oplus U^{*})^{\fp (n)} = \left\{  s \in S(U \oplus U^{*}) \mid x . s = 0 \text{ for all } x \in \fp(n)\right\} .
\]
  As a $\fp (n)$-module one has
\[
S(U \oplus U^{*}) \cong S(U)\otimes S(U^{*}) \cong \bigoplus_{\bba , \bbb \in \Z_{\geq 0}^{m}} S^{\bba}(V_{n}) \otimes S^{\bbb}(V_{n}^{*}) \cong  \bigoplus_{\bba , \bbb \in \Z_{\geq 0}^{m}} S^{\bba}(V_{n}) \otimes S^{\bbb}(V_{n})^{*}.
\]  On generators the composite of these isomorphisms is given by $z_{i,j} \mapsto \left(1^{\otimes  i-1} \otimes v_{j} \otimes 1^{m-i} \right) \otimes \left(1^{\otimes m} \right)$ and $z_{i,j}^{*} \mapsto \left(1^{\otimes m} \right) \otimes \left(1^{\otimes  i-1} \otimes v_{j}^{*} \otimes 1^{m-i} \right)$.

These isomorphisms induce superspace isomorphisms 
\begin{align}\label{isomSUstar}
S(U \oplus U^{*})^{\fp (n)} \cong \Hom_{\fp (n)}\left(\k  , S(U\oplus U^{*}) \right) \cong \bigoplus_{\bba , \bbb \in \Z_{\geq 0}^{m}}  \Hom_{\fp (n)}\left(\k , S^{\bba }(V_{n})\otimes S^{\bbb}(V_{n})^{*} \right).
\end{align}

\begin{theorem}\label{T:SpanningSet} The set 
\begin{equation}\label{E:SpanningSet}
\left\{ 
\prod_{ 1 \leq i < j \leq m} \left(\sum_{r \in I_{n|n}} (-1)^{\p{r}}z_{i,r} z_{j,-r} \right)^{r_{i,j}} \prod_{1 \leq i,j \leq m} \left(\sum_{r \in I_{n|n}} z_{i,r} z^{*}_{j, r} \right)^{s_{i,j}} \prod_{1 \leq i \leq j \leq m} \left(\sum_{r \in I_{n|n}} z^{*}_{i,r} z^{*}_{j,-r} \right)^{t_{i,j}}  \right\},
\end{equation}
where $s_{i,j} \in \Z_{\geq 0}$ and $ r_{i,j}, t_{i,j} \in \left\{0,1 \right\}$, spans $S(U \oplus U^{*})^{\fp (n)}$.

\end{theorem}

\begin{proof} By \cref{T:WebsTheorem}, for every $\bba, \bbb  \in \Z_{\geq 0}^{m}$ the map  
\[
G_{\uparrow \downarrow}: \Hom_{\pWeb_{\uparrow \downarrow}}\left(\emptyset, \uparrow_{\bba}\downarrow_{\bbb} \right) \to \Hom_{\fp (n)}\left(\k , S^{\bba }(V_{n})\otimes S^{\bbb}(V_{n})^{*} \right)
\] is surjective. In \cite[Lemma 6.9.1]{DKM} it is shown that $\Hom_{\pWeb_{\uparrow \downarrow}}\left(\emptyset, \uparrow_{\bba}\downarrow_{\bbb} \right)$ is spanned by diagrams of the form
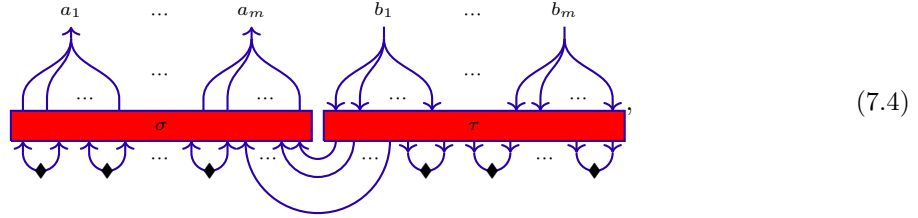
\begin{align}\label{0abdiag}
\hackcenter{}
\hackcenter{
\begin{tikzpicture}[scale=0.8]
   \draw[thick, color=\clr, fill=red] (-1,1.7)--(4,1.7)--(4,2.2)--(-1,2.2)--(-1,1.7);
    \node at (1.5,1.95) {$\scriptstyle \sigma$};
 \draw[thick, color=\clr,<-] (0,3.4+0.2)--(0,3.4+0);
 \draw[thick, color=\clr] (0,3.4+0) .. controls ++(0,-0.5) and ++(0,0.5) .. (-0.8,3.4+-1)--(-0.8,3.4+-1.2);
  \draw[thick, color=\clr] (0,3.4+0) .. controls ++(0,-0.5) and ++(0,0.5) .. (-0.4,3.4+-1)--(-0.4,3.4+-1.2);
       \draw[thick, color=\clr] (0,3.4+0) .. controls ++(0,-0.5) and ++(0,0.5) .. (0.8,3.4+-1)--(0.8,3.4+-1.2);
                  \node[left] at (0.6,3.4+-1) {$\scriptstyle \cdots $};
            \node[above] at (0,3.4+0.2) {$\scriptstyle a_1 $};
  \node at (1.5,2.8) {$\scriptstyle \cdots $};
   \node at (1.5,3.8) {$\scriptstyle \cdots $};
      \node at (1.5,1.4) {$\scriptstyle \cdots $};
          \node at (7.9,1.4) {$\scriptstyle \cdots $};
            \node at (3.3,1.4) {$\scriptstyle \cdots $};
             \node at (5,1.4) {$\scriptstyle \cdots $};
 \draw[thick, color=\clr,<-] (3+0,3.4+0.2)--(3+0,3.4+0);
 \draw[thick, color=\clr] (3+0,3.4+0) .. controls ++(0,-0.5) and ++(0,0.5) .. (3+-0.8,3.4+-1)--(3+-0.8,3.4+-1.2);
  \draw[thick, color=\clr] (3+0,3.4+0) .. controls ++(0,-0.5) and ++(0,0.5) .. (3+-0.4,3.4+-1)--(3+-0.4,3.4+-1.2);
       \draw[thick, color=\clr] (3+0,3.4+0) .. controls ++(0,-0.5) and ++(0,0.5) .. (3+0.8,3.4+-1)--(3+0.8,3.4+-1.2);
                  \node[left] at (3+0.6,3.4+-1) {$\scriptstyle \cdots $};
            \node[above] at (3+0,3.4+0.2) {$\scriptstyle a_m $};
     \draw[thick, color=\clr, fill=red] (5.2+-1,1.7)--(5.2+4,1.7)--(5.2+4,2.2)--(5.2+-1,2.2)--(5.2+-1,1.7);
    \node at (5.2+1.5,1.95) {$\scriptstyle \tau$};
 \draw[thick, color=\clr] (5.2+0,3.4+0.2)--(5.2+0,3.4+0);
 \draw[thick, color=\clr,->] (5.2+0,3.4+0) .. controls ++(0,-0.5) and ++(0,0.5) .. (5.2+-0.8,3.4+-1)--(5.2+-0.8,3.4+-1.2);
  \draw[thick, color=\clr,->] (5.2+0,3.4+0) .. controls ++(0,-0.5) and ++(0,0.5) .. (5.2+-0.4,3.4+-1)--(5.2+-0.4,3.4+-1.2);
       \draw[thick, color=\clr,->] (5.2+0,3.4+0) .. controls ++(0,-0.5) and ++(0,0.5) .. (5.2+0.8,3.4+-1)--(5.2+0.8,3.4+-1.2);
                  \node[left] at (5.2+0.6,3.4+-1) {$\scriptstyle \cdots $};
            \node[above] at (5.2+0,3.4+0.2) {$\scriptstyle b_1 $};
  \node at (5.2+1.5,2.8) {$\scriptstyle \cdots $};
   \node at (5.2+1.5,3.8) {$\scriptstyle \cdots $};
 \draw[thick, color=\clr] (5.2+3+0,3.4+0.2)--(5.2+3+0,3.4+0);
 \draw[thick, color=\clr,->] (5.2+3+0,3.4+0) .. controls ++(0,-0.5) and ++(0,0.5) .. (5.2+3+-0.8,3.4+-1)--(5.2+3+-0.8,3.4+-1.2);
  \draw[thick, color=\clr,->] (5.2+3+0,3.4+0) .. controls ++(0,-0.5) and ++(0,0.5) .. (5.2+3+-0.4,3.4+-1)--(5.2+3+-0.4,3.4+-1.2);
       \draw[thick, color=\clr,->] (5.2+3+0,3.4+0) .. controls ++(0,-0.5) and ++(0,0.5) .. (5.2+3+0.8,3.4+-1)--(5.2+3+0.8,3.4+-1.2);
                  \node[left] at (5.2+3+0.6,3.4+-1) {$\scriptstyle \cdots $};
            \node[above] at (5.2+3+0,3.4+0.2) {$\scriptstyle b_m $};
    \draw (4.4,1.7) arc (0:-180:0.3cm) [thick, color=\clr,->];
     \draw (4.7,1.7) arc (0:-180:0.6cm) [thick, color=\clr,->];
       \draw (5.3,1.7) arc (0:-180:1.2cm) [thick, color=\clr,->];
        \draw (-0.2+0,0+1.5) arc (0:-180:0.3cm) [thick, color=\clr];
   \draw[thick, color=\clr,<-] (-0.2+-0.6,0.2+1.5)--(-0.2+-0.6,0+1.5);
     \draw[thick, color=\clr,<-] (-0.2+0,0.2+1.5)--(-0.2+0,0+1.5);
  \node at (-0.2+-0.3,-0.3+1.5) {$\scriptstyle\blacklozenge$};
      \draw (0.9+0,0+1.5) arc (0:-180:0.3cm) [thick, color=\clr];
   \draw[thick, color=\clr,<-] (0.9+-0.6,0.2+1.5)--(0.9+-0.6,0+1.5);
     \draw[thick, color=\clr,<-] (0.9+0,0.2+1.5)--(0.9+0,0+1.5);
  \node at (0.9+-0.3,-0.3+1.5) {$\scriptstyle\blacklozenge$};
      \draw (2.6+0,0+1.5) arc (0:-180:0.3cm) [thick, color=\clr];
   \draw[thick, color=\clr,<-] (2.6+-0.6,0.2+1.5)--(2.6+-0.6,0+1.5);
     \draw[thick, color=\clr,<-] (2.6+0,0.2+1.5)--(2.6+0,0+1.5);
  \node at (2.6+-0.3,-0.3+1.5) {$\scriptstyle\blacklozenge$};
        \draw (6.4+-0.2+0,0+1.5) arc (0:-180:0.3cm) [thick, color=\clr];
   \draw[thick, color=\clr,->] (6.4+-0.2+-0.6,0.2+1.5)--(6.4+-0.2+-0.6,0+1.5);
     \draw[thick, color=\clr,->] (6.4+-0.2+0,0.2+1.5)--(6.4+-0.2+0,0+1.5);
  \node at (6.4+-0.2+-0.3,-0.3+1.5) {$\scriptstyle\blacklozenge$};
      \draw (6.4+0.9+0,0+1.5) arc (0:-180:0.3cm) [thick, color=\clr];
   \draw[thick, color=\clr,->] (6.4+0.9+-0.6,0.2+1.5)--(6.4+0.9+-0.6,0+1.5);
     \draw[thick, color=\clr,->] (6.4+0.9+0,0.2+1.5)--(6.4+0.9+0,0+1.5);
  \node at (6.4+0.9+-0.3,-0.3+1.5) {$\scriptstyle\blacklozenge$};
      \draw (6.4+2.6+0,0+1.5) arc (0:-180:0.3cm) [thick, color=\clr];
   \draw[thick, color=\clr,->] (6.4+2.6+-0.6,0.2+1.5)--(6.4+2.6+-0.6,0+1.5);
     \draw[thick, color=\clr,->] (6.4+2.6+0,0.2+1.5)--(6.4+2.6+0,0+1.5);
  \node at (6.4+2.6+-0.3,-0.3+1.5) {$\scriptstyle\blacklozenge$};
  \end{tikzpicture}},
\end{align}
where all strands except those at the very top of the the diagram are labeled with a $1$ and \(\sigma, \tau\) consist only of upward and downward crossings, respectively.

We now use \cite[Theorem 6.5.1]{DKM} to compute the image of $G_{\uparrow\downarrow}$ on elements of the form \cref{0abdiag}. First, the image of thin-strand crossings that comprise \(\sigma, \tau\) are given by the super flip map (i.e., $x \otimes y \mapsto (-1)^{\p{x}\p{y}}y \otimes x$), and the image of the merge maps at the top of the diagram are given by multiplication in $S\left(V_{n} \right)$ and $S\left(V_{n}^{*} \right)$. It remains to compute $G_{\uparrow \downarrow}$ on the cups at the bottom of the diagram.

As explained in \cite{DKM}, $G_{\uparrow \downarrow}$ yields the following homomorphisms:
\begin{align*}
G_{\uparrow\downarrow}\left(
\hackcenter{
\begin{tikzpicture}[scale=0.5]
 \draw (0,0) arc (0:-180:0.4cm) [thick, color=\clr];
   \draw[thick, color=\clr] (-0.8,0.3)--(-0.8,0);
     \draw[thick, color=\clr] (0,0.3)--(0,0);
      \draw[thick, color=\clr,<-] (-0.8,0.3)--(-0.8,0.2);
     \draw[thick, color=\clr,<-] (0,0.3)--(0,0.2);
         \node[above] at (0,0.1) {${}$};
  \node at (-0.4,-0.4) {$\scriptstyle\blacklozenge$};
\end{tikzpicture}}
\right): \k \to V_{n} \otimes V_{n},
\qquad
1
 &\mapsto \sum_{r \in I_{n|n}}(-1)^{\p{r}} v_{r} \otimes v_{-r};
 \\
G_{\uparrow\downarrow}\left(
\hackcenter{
\begin{tikzpicture}[scale=0.5]
  \node[white] at (-0.4,-0.4) {$\scriptstyle\blacklozenge$};
 \draw (0,0) arc (0:-180:0.4cm) [thick, color=\clr];
   \draw[thick, color=\clr] (-0.8,0.3)--(-0.8,0);
     \draw[thick, color=\clr] (0,0.3)--(0,0);
      \draw[thick, color=\clr,<-] (-0.8,0.3)--(-0.8,0.2);
     \draw[thick, color=\clr] (0,0.3)--(0,0.2);
         \node[above] at (0,0.1) {${}$};
\end{tikzpicture}}
\right): \k \to V_{n} \otimes V_{n}^{*},
\qquad
1
 &\mapsto \sum_{r \in I_{n|n}} v_{r} \otimes v_{r}^{*};
 \\
G_{\uparrow\downarrow}\left(
\hackcenter{
\begin{tikzpicture}[scale=0.5]
 \draw (0,0) arc (0:-180:0.4cm) [thick, color=\clr];
   \draw[thick, color=\clr] (-0.8,0.3)--(-0.8,0);
     \draw[thick, color=\clr] (0,0.3)--(0,0);
      \draw[thick, color=\clr,->] (-0.8,0.21)--(-0.8,0.15);
     \draw[thick, color=\clr,->] (0,0.21)--(0,0.15);
         \node[above] at (0,0.2) {${}$};
  \node at (-0.4,-0.4) {$\scriptstyle\blacklozenge$};
\end{tikzpicture}}
\right): \k \to V_{n}^{*} \otimes V_{n}^{*},
\qquad
1
&\mapsto \sum_{r \in I_{n|n}} v_{-r}^{*} \otimes v_{r}^{*}.
\end{align*}

Combining these calculations with image of the crossing and merge morphisms, the fact $G_{\uparrow\downarrow}$ is a monoidal functor, and the isomorphism \cref{isomSUstar} shows that applying $G_{\uparrow\downarrow}$ to the diagrams of the form \cref{0abdiag} exactly gives the elements in the set \cref{E:SpanningSet}.  This proves the claim.
\end{proof}

\subsection{The Weyl-Clifford Superalgebra} 
Let $U=U_{\0} \oplus U_{1}$ be a finite-dimensional superspace.  Given a homogeneous element $u \in U$ let 
\[
\mu _{u}: S(U) \to S(U)
\] be the linear map given by left multiplication by $u$.  The parity of the map $\mu _{u}$ equals the parity of $u$.  Let $U^{*}= U_{\0}^{*} \oplus U_{\1}^{*}$ and let $\langle -, - \rangle: U^{*} \otimes U \to \k$ denote the canonical pairing $f \otimes u \mapsto f(u)$.  Given a homogeneous element $f \in U^{*}$, let 
\[
D_{f}:S(U) \to S(U)
\] denote the derivative with respect to $f$.  Namely, if $x_{1}, \dotsc , x_{p} \in U_{\0}$ and $y_{1}, \dotsc , y_{q} \in U_{\1}$, then $D_{f}$ is the linear map defined by
\[
D_{f}\left(x_{1}\dotsb x_{p}y_{1}\dotsb y_{q} \right) = \begin{cases} \sum_{i=1}^{p} \langle f, x_{i} \rangle x_{1}\dotsb \widehat{x}_{i}\dotsb x_{p}y_{1}\dotsb y_{q}, & \text{if $f$ is even}; \\
                        \sum_{i=1}^{q}(-1)^{i-1} \langle f, y_{i} \rangle x_{1}\dotsb x_{p}y_{1}\dotsb \widehat{y}_{i}\dotsb  y_{q}, &\text{if $f$ is odd}.

\end{cases}
\] Here $\widehat{x}$ denotes the deletion of the element $x$.  The parity of $D_{f}$ equals the parity of $f$.

By definition the unital associative subsuperalgebra of $\End_{\k}\left(S(U) \right)$ generated by 
\[
\left\{\mu_{u} \mid u \in U_{\0}, u \in U_{\1 } \right\} \cup \left\{D_{f}\mid f \in U_{\0}^{*}, f \in U_{\1}^{*} \right\}
\]
is the Weyl-Clifford superalgebra $\mathcal{WC}\left(U \right)$.  We call elements of $\mathcal{WC}\left(U \right)$ \emph{polynomial differential operators} on the superspace $U$.  See \cite[Section 5.2]{CWBook} for details.

\subsection{The Weyl-Clifford Superalgebra for \texorpdfstring{$Y_{m} \otimes V_{n}$}{Ym tensor Vn}}\label{SS:WCforVn}
Now let $U=Y_{m} \otimes V_{n}$ as in \cref{E:Udef}. As before, $U$ has distinguished homogeneous basis $z_{i,j} = y_{i} \otimes v_{j}$.  For short, we write $z_{i,j}$ for $\mu_{z_{i,j}}$ and $\partial_{i,j}$ for $D_{z_{i,j}^{*}}$.  By \cite[Lemma 5.1]{CWBook} or a direct calculation, the following supercommutators hold:
\begin{gather}\label{E:WCcommutator}
[z_{i,j}, z_{k,l}]=0=[\partial_{i,j}, \partial_{k,l}] \\ 
\notag\text{ and } \\
\label{E:WCcommutator2} [\partial_{i,j}, z_{k,l}] = \langle z_{i,j}^{*}, z_{k,l} \rangle =  \delta_{i,k} \delta_{j,l} .
\end{gather}

There is a filtration by total degree on $\mathcal{WC}(U)$ where the elements $z_{i,j}$ and $\partial_{i,j}$ have filtration degree one.  Write $\mathcal{WC}(U)_{k}$ for the $k$th filtered piece of $\mathcal{WC}(U)$.  The commutator formulas imply that the associated graded superalgebra $\operatorname{gr}  \mathcal{WC}(U)$ is isomorphic to $S\left(U \oplus U^{*} \right)$.  As a consequence we obtain the following lemma.
\begin{lemma} \label{E:WCBasis} The set 
\begin{equation*}
\left\{ \prod_{\substack{1 \leq i \leq m  \\ j \in I_{n|n} }} z_{i,j}^{r_{i,j}} \prod_{\substack{1 \leq i \leq m  \\ j \in I_{n|n} }} \partial_{i,j}^{s_{i,j}}  \right\},
\end{equation*}
where $r_{i,j}, s_{i,j} \in \Z_{\geq 0}$ if $\p{z_{i,j}}=\0$ and $r_{i,j}, s_{i,j} \in \{0,1 \}$ if $\p{z_{i,j}}=\1$, is a basis for $\mathcal{WC}(U)$.  
\end{lemma}

Note that because of \cref{E:WCcommutator} the factors in the two products in the basis above \cref{E:WCBasis} can be taken in any order (up to a sign).  For definiteness we order the products by ordering the pairs $(i,j)$ lexicographically.

We fix an explicit, parity-preserving, superspace isomorphism,
\[
\varsigma: \mathcal{WC}(U) \to S(U \oplus U^{*})
\] given by 
\[
\varsigma \left( \prod_{\substack{ 1 \leq i \leq m \\ j \in I_{n|n} }} z_{i,j}^{r_{i,j}} \prod_{\substack{ 1 \leq i \leq m\\ j \in I_{n|n} }} \partial_{i,j}^{s_{i,j}}  \right) = \prod_{\substack{ 1 \leq i \leq m\\ j \in I_{n|n} }} z_{i,j}^{r_{i,j}} \prod_{\substack{ 1 \leq i \leq m\\ j \in I_{n|n} }} (z^{*}_{i,j})^{s_{i,j}},
\]  called the \emph{symbol map}.

The natural action of $\gl (n|n)$ on $V_{n}$ induces an action on $S(U)$ and, in turn, on $\End_{\k}(S(U))$ by $(x.T)(p) = x.(T (p)) - (-1)^{\p{x}\p{T}}T (x.p)$ for all homogeneous $x \in \gl (n|n)$, $T \in \End_{\k}(S(U))$, $p \in S(U)$. Note the action of $\gl(n|n)$ is by superderivations.   That is, for homogeneous $x \in \gl (n|n)$ and homogeneous $T_{1}, T_{2} \in \End_{\k}(S(U))$,
\[
x.(T_{1}T_{2}) = (x.T_{1})T_{2} + (-1)^{\p{x}\p{T_{1}}} T_{1}(x.T_{2}). 
\]  
Combining this observation with a check of the action of $\gl(n|n)$ on the generators of $\mathcal{WC}(U) \subset \End_{\k}(S(U))$ verifies that $\mathcal{WC}(U)$ is preserved by the action of $\gl(n|n)$.  Because $\fp(n)$ is a Lie subsuperalgebra of $\gl(n|n)$ there is an action of $\fp (n)$ on $\mathcal{WC}(U)$ as well.  Write $\mathcal{WC}(U)^{\fp (n)}$ for the invariants under this action.

\begin{lemma}\label{L:somestuff}  For any $m \geq 1$ the following statements hold for $\mathcal{WC}(U)=\mathcal{WC}(Y_{m}\otimes V_{n})$:
\begin{enumerate}
\item For any fixed $1 \leq i \leq m$, the subspace of $\mathcal{WC}(U)$ spanned by $\left\{z_{i,j} \mid j \in I_{n|n} \right\}$ is isomorphic as a $\gl (n|n)$-module to $V_{n}$   via $v_{j} \mapsto z_{i,j}$.
\item For any fixed $1 \leq i \leq m$, the subspace of $\mathcal{WC}(U)$ spanned by $\left\{\partial_{i,j} \mid j \in I_{n|n} \right\}$ is isomorphic as a $\gl (n|n)$-module to $V^{*}_{n}$ via $v^{*}_{j} \mapsto \partial_{i,j}$.
\item For all $k \geq 0$, the symbol map $\varsigma: \mathcal{WC}(U)_{k} \to \bigoplus_{d=0}^{k} S^{d}(U \oplus U^{*})$ defines an isomorphism as $\mathfrak{gl}(n|n)$-modules and, hence, as $\fp (n)$-modules. 
\item The subsuperspace  $\mathcal{WC}(U)^{\fp (n)}$ is a subsuperalgebra of $\mathcal{WC}(U)$.
\item The symbol map defines a superspace isomorphism 
\[
\mathcal{WC}(U)^{\fp (n)} \cong S(U \oplus U^{*})^{\fp (n)}.
\]
 
\end{enumerate}

\end{lemma}

\begin{proof} Direct calculations verify the first two statements.  This along with \cref{E:WCcommutator} demonstrate that the subspace of $\mathcal{WC}(U)_{k}$ spanned by the basis elements from \cref{E:WCBasis} of filtration degree $d$ for any fixed $d \leq k$ is a $\gl (n|n)$-module.  Moreover, the symbol map defines an isomorphism as $\gl (n|n)$-modules between this subspace and $S^{d}(U \oplus U^{*})$.  From this the claims in the third and fifth statements of the lemma follow.  The fourth follows from the fact that $\fp (n)$ acts by superderivations.
\end{proof}

\subsection{Polynomial Differential Operators, I}
For $i,j \in I_{n|n}$, let 
\[
\ttE_{i,j} = \sum_{p=1}^{m} z_{p,i}\partial_{p,j} \in \mathcal{WC}(U).
\]  For $1 \leq i,j \leq n$, let
\begin{equation*}
\ttA_{i,j} = \ttE_{i,j}-\ttE_{-j,-i} 
\end{equation*}
For $1 \leq i < j \leq n$, let 
\begin{align*}
\ttB_{i,j} =  \ttE_{i,-j}+\ttE_{j,-i}.
\end{align*}
For $1 \leq i \leq n$, let 
\[
\ttB_{i,i} = \ttE_{i,-i}.
\]
For $1 \leq i < j \leq n$, let 
\begin{equation*}
\mathtt{C}_{i,j} =  \ttE_{-i,j}-\ttE_{-j,i}.
\end{equation*}

\begin{lemma}\label{L:fpembedding}  The map $\fp (n) \to \mathcal{WC}(U)$ given by 
\[
a_{i,j}\mapsto \ttA_{i,j} \quad  b_{i,j} \mapsto \mathtt{B}_{i,j} \quad c_{i,j} \mapsto \mathtt{C}_{i,j}
\] defines an embedding of Lie superalgebras.  Moreover, the action of $\fp (n)$ on $S(U)$ via this map coincides with the one induced by the natural action of $\fp (n)$ on $V_{n}$ and the adjoint action of $\fp (n)$ on $\mathcal{WC}(U)$ via this map coincides with the action induced by the natural action of $\fp (n)$ on $V_{n}$.
\end{lemma}

\begin{proof}
As explained in \cite[Lemma 5.13]{CWBook}, the element $\ttE_{i,j} \in \mathcal{WC}(U)$ acts on $S(U)$ as the matrix unit $E_{i,j} \in \mathfrak{gl}(n|n)$ (where the action is the one induced by the natural action of $\gl (n|n)$ on $V_{n}$).  Moreover, the map $E_{i,j} \mapsto \ttE_{i,j}$ defines an embedding of Lie superalgebras $\mathfrak{gl}(n|n) \to \mathcal{WC}(U)$.  Restricting this embedding to $\fp (n)$ yields the claimed embedding and from this it follows the actions on $\mathcal{WC}(U)$ coincide.
\end{proof}

Define the following elements of $\mathcal{WC}(U)$.
For $1 \leq i, j \leq m$, 
\begin{equation}\label{E:EijDifferential}
\tildettA_{i,j} = \sum_{r \in I_{n|n}} z_{i,r} \partial_{j, r} -n\delta_{i,j} = \sum_{r=1}^{n}    z_{i,r} \partial_{j, r} + \sum_{r=1}^{n} z_{i,-r} \partial_{j, -r} -n\delta_{i,j} 
\end{equation} 
 For $1 \leq i \neq j \leq m$, 
\begin{align}\label{E:BijDifferential}
\tildettB_{i,j}  &= \sum_{r \in I_{n|n}} \partial_{i,r} \partial_{j,-r}= \sum_{r=1}^{n} \partial_{i,r} \partial_{j,-r} + \sum_{r=1}^{n} \partial_{i,-r} \partial_{j,r}, \\
\tildettB_{i,i} &=\sum_{r =1}^{n} \partial_{i,r} \partial_{i,-r}.
\end{align}
For $1 \leq i < j \leq m$, 
\begin{align}\label{E:CijDifferential} 
\tildettC_{i,j}  &= \sum_{r \in I_{n|n}} (-1)^{\p{r}}z_{i,r} z_{j,-r}= \sum_{r=1}^{n} z_{i,r} z_{j,-r} - \sum_{r=1}^{n}  z_{i,-r}  z_{j,r},\\
\tildettC_{j,i} &=-\tildettC_{i,j}.
\end{align}

\begin{lemma}\label{L:tfpembedding}  The elements $\tildettA_{i,j}$, $\tildettB_{i,j}$, $\tildettB_{i,i}$, and $\tildettC_{i,j}$ defined above are elements of $\mathcal{WC}(U)^{\fp (n)}$.  Moreover, the subspace they span is a Lie subsuperalgebra of $\mathcal{WC}(U)$ isomorphic to $\tfp (m)$.

In particular, there is a superalgebra homomorphism 
\begin{equation}\label{E:ClassicalHoweDualityMap}
\tilde{\theta}: U(\tfp (m)) \to \mathcal{WC}(U)^{\fp (n)}
\end{equation}
given by $\tilde{a}_{i,j}\mapsto \tildettA_{i,j}$, $\tilde{b}_{i,j}\mapsto \tildettB_{i,j}$, and $\tilde{c}_{i,j}\mapsto \tildettC_{i,j}$.
\end{lemma}

\begin{proof} By applying \cref{L:fpembedding}, the claim that $\tildettA_{i,j}$, $\tildettB_{i,j}$, and $\tildettC_{i,j}$ are $\fp (n)$-invariant can be verified by checking these elements commute with ${\ttA}_{i,j}$, ${\ttB}_{i,j}$, and ${\ttC}_{i,j}$ in $\mathcal{WC}(U)$.  To do so, we first note by \cite[Lemma 5.13]{CWBook} that the elements $\tildettA_{i,j}$ are $\gl (n|n)$-invariant (hence $\fp (n)$-invariant).  Second, a direct calculation using \cref{E:WCcommutator,E:WCcommutator2}  verifies 
\begin{align*}
\left[\tildettB_{i,j}, \ttE_{a,b} \right] &= \partial_{i,-a}\partial_{j,b}+\partial_{j,-a}\partial_{i,b}, \\
\left[\tildettC_{i,j}, \ttE_{a,b} \right] &= (-1)^{\p{a}}z_{j,a}z_{i,-b}-(-1)^{\p{a}}z_{i,a}z_{j,-b}.
\end{align*}
Using these along with \cref{L:fpembedding} it is straightforward to check that $\tildettB_{i,j}$ and $\tildettC_{i,j}$ also are $\fp (n)$-invariant.

It is clear  $\tildettA_{i,j}$, $\tildettB_{i,j}$, and $\tildettC_{i,j}$ are linearly independent.  The Chevalley isomorphism given in \cref{E:TauOnGenerators} along with \cref{L:Commutators} provides a basis and commutator formulas for $\tfp (m)$.   Direct calculations using \cref{E:WCcommutator,E:WCcommutator2} verifies that these formulas are also satisfied by  $\tildettA_{i,j}$, $\tildettB_{i,j}$, and $\tildettC_{i,j}$.  Thus the subspace they span is isomorphic to $\tfp (m)$ and the claimed map $\tilde{\theta}$ exists.
\end{proof}

Using the map $\tilde{\theta}$, $S(U)$ can be viewed as a $\tfp (m)$-module where the action commutes with the natural action of $\fp (n)$ on $S(U)$.  That is, $S(U)$ is a $U(\tfp (m)) \otimes U(\fp (n))$-module.

\subsection{Classical Howe Duality, I}\label{SS:classicalHowedualityI}

\begin{theorem}\label{T:ClassicalHoweDuality1}  The map $\tilde{\theta}$ is surjective.
\end{theorem}

\begin{proof} Since the elements  $\tilde{a}_{i,j}$, $\tilde{b}_{i,j}$, and $\tilde{c}_{i,j}$ form a basis of $\tfp (m)$, by the PBW theorem for $U(\tfp(m))$ the image of $\tilde{\theta}$ is spanned by elements of the form 
\begin{equation}\label{E:Xproduct}
\prod_{ 1 \leq i < j \leq m} \tildettC^{r_{i,j}}_{i,j} \prod_{1 \leq i,j \leq m} \tildettA^{s_{i,j}}_{i,j} \prod_{1 \leq i \leq j \leq m} \tildettB^{t_{i,j}}_{i,j},
\end{equation}
where the products are taken in some fixed order, and where $r_{i,j}, t_{i,j} \in \left\{0,1 \right\}$ and $s_{i,j} \in \Z_{\geq 0}$. Applying the symbol map to \cref{E:Xproduct} yields following element of $S\left(U \oplus U^{*} \right)^{\fp (n)}$: 
\[
\prod_{ 1 \leq i < j \leq m} \left(\sum_{r \in I_{n|n}} (-1)^{\p{r}}z_{i,r} z_{j,-r} \right)^{r_{i,j}} \prod_{1 \leq i,j \leq m} \left(\sum_{r \in I_{n|n}} z_{i,r} z^{*}_{j, r} \right)^{s_{i,j}} \prod_{1 \leq i \leq j \leq m} \left(\sum_{r \in I_{n|n}} z^{*}_{i,r} z^{*}_{j,-r} \right)^{t_{i,j}} + (***),
\] where $(***)$ consists of terms in $S(U \oplus U^{*})^{\fp (n)}$ of strictly smaller total degree determined by the commutator formulas \cref{E:WCcommutator,E:WCcommutator2}.

Comparing this to the spanning set for $S(U \oplus U^{*})^{\fp (n)}$ given in \cref{T:SpanningSet} shows that the image of the PBW basis under $\varsigma \circ \tilde{\theta}$ is related to the spanning set by a unitriangular matrix.  Consequently they span the same subspace of $\mathcal{WC}(U)$ and $\tilde{\theta}$ is surjective.
\end{proof}

\subsection{Polynomial Differential Operators, II} There is an asymmetry in the roles of $\tfp (m)$ and $\fp (n)$ in \cref{T:ClassicalHoweDuality1}  which we now rectify.

Given $m,n \geq 1$, let $W_{m}$ be the natural $\tfp (m)$-module and let $Y_{n}$ with purely even basis $\left\{y_{\ell} \mid 1 \leq \ell \leq n \right\}$ and with trivial $\tfp (m)$-action. Set  $U'=U'_{m,n} = W_{m} \otimes Y_{n}$.  Then $U'$ has homogeneous basis 
\[
\left\{\zprime_{k, \ell} := w_{k}\otimes y_{\ell} \mid   k \in I_{m|m} \text{ and } 1 \leq  \ell \leq n   \right\},
\]
where $\zprime_{k,\ell}$ has the same parity as $w_{k}$.  The action of $\tfp (m)$ on $U'$ induces an action on $S(U')$.  As a matter of notation, we write $\zprime_{k, \ell}$ and $\partialprime_{k, \ell}$ for the generators of $\mathcal{WC}(U')$.

For $i,j \in I_{m|m}$, let 
\[
\ttEprime_{i,j} = \sum_{p=1}^{n} \zprime_{i,p}\partialprime_{j,p} \in \mathcal{WC}(U').
\]  For $1 \leq i,j \leq m$, let
\begin{equation*}
\tildettAprime_{i,j} = \ttEprime_{i,j}-\ttEprime_{-j,-i} 
\end{equation*}
For $1 \leq i < j \leq m$, let 
\begin{align*}
\tildettBprime_{i,j} =  \ttEprime_{-i,j}+\ttEprime_{-j,i}.
\end{align*}
For $1 \leq i \leq m$, let 
\[
\tildettBprime_{i,i} = \ttEprime_{-i,i}.
\]
For $1 \leq i < j \leq m$, let 
\begin{equation*}
\tildettCprime_{i,j} =  \ttEprime_{i,-j}-\ttEprime_{j,-i}.
\end{equation*}
The next result is proven just as \cref{L:fpembedding}.
\begin{lemma}\label{L:tfpembedding2}  The map $\tfp (m) \to \mathcal{WC}(U')$ given by 
\[
\ta_{i,j}\mapsto \tildettAprime_{i,j} \quad  \tb_{i,j} \mapsto \tildettBprime_{i,j} \quad \tc_{i,j} \mapsto \tildettCprime_{i,j}
\] defines an embedding of Lie superalgebras.  Moreover, the action of $\tfp (m)$ on $U'$ is the one induced by the natural action of $\tfp (m)$ on $W_{m}$ and the adjoint action of $\tfp (m)$ on $\mathcal{WC}(U')$ via this map coincides with the action induced by the natural action of $\tfp (m)$ on $W_{m}$.
\end{lemma}

Define the following elements of $\mathcal{WC}(U')$.
For $1 \leq i, j \leq n$, 
\begin{equation}\label{E:EprimeijDifferential}
\ttAprime_{i,j} = \sum_{r \in I_{m|m}} \zprime_{r,i} \partialprime_{r,j} -m\delta_{i,j} \Id_{U'} = \sum_{r=1}^{m}   \zprime _{r,i} \partialprime_{r,j} + \sum_{r=1}^{m} \zprime_{-r,i} \partialprime_{-r,j} -m\delta_{i,j} \Id_{U'}.
\end{equation} 
 For $1 \leq i \neq j \leq n$, 
\begin{align}\label{E:BprimeijDifferential}
\ttBprime_{i,j}  &= \sum_{r \in I_{m|m}} \zprime_{r,i} \zprime_{-r,j}= \sum_{r=1}^{m} \zprime_{r,i} \zprime _{-r,j} + \sum_{r=1}^{m} \zprime_{-r,i} \zprime_{r,j}, \\
\ttBprime_{i,i} &=\sum_{r =1}^{m} \zprime_{r,i} \zprime_{-r,i}.
\end{align}
For $1 \leq i < j \leq n$, 
\begin{align}\label{E:CprimeijDifferential} 
\ttCprime_{i,j}  &= \sum_{r \in I_{m|m}} (-1)^{\p{r}}\partialprime_{-r,i} \partialprime_{r,j}= \sum_{r=1}^{m} \partialprime_{-r,i} \partialprime_{r,j} - \sum_{r=1}^{m}  \partialprime_{r,i}  \partialprime_{-r,j},\\
\ttCprime_{j,i} &=-\tildettCprime_{i,j}.
\end{align}
We have the following analogue of \cref{L:tfpembedding}.
\begin{lemma}  The elements $\ttAprime_{i,j}$, $\ttBprime_{i,j}$,  $\ttBprime_{i,i}$, and $\ttCprime_{i,j}$ defined above are elements of $\mathcal{WC}(U')^{\tfp (m)}$.  Moreover, the subspace they span is a Lie subsuperalgebra of $\mathcal{WC}(U')$ isomorphic to $\fp (n)$.

In particular, there is a superalgebra homomorphism 
\begin{equation}\label{E:ClassicalHoweDualityMapPrime}
\theta': U(\fp (n)) \to \mathcal{WC}(U')^{\tfp (m)}
\end{equation}
given by $a_{i,j}\mapsto \ttAprime_{i,j}$, $b_{i,j}\mapsto \ttBprime_{i,j}$, and $c_{i,j}\mapsto \ttCprime_{i,j}$.
\end{lemma}
\begin{proof} By applying \cref{L:tfpembedding2}, the claim that $\ttAprime_{i,j}$, $\ttBprime_{i,j}$, and $\ttCprime_{i,j}$ are $\tfp (m)$-invariant can be verified by checking these elements commute with $\tildettAprime_{i,j}$, $\tildettBprime_{i,j}$, and $\tildettCprime_{i,j}$ in $\mathcal{WC}(U')$.  First, by \cite[Lemma 5.13]{CWBook} the elements $\ttAprime_{i,j}$ are $\gl (m|m)$-invariant and, hence, $\tfp (m)$ invariant.  Second, a direct calculation using \cref{E:WCcommutator,E:WCcommutator2}  verifies 
\begin{align*}
\left[\ttBprime_{i,j}, \ttEprime_{a,b} \right] &= (-1)^{\p{a}+\p{b}+1}\zprime_{a,i}\zprime_{-b,j}+(-1)^{\p{a}+\p{b}+1}\zprime_{a,j}\zprime_{-b,i}, \\
\left[\ttCprime_{i,j}, \ttEprime_{a,b} \right] &= (-1)^{\p{a}}\partialprime_{-a,i}\partialprime_{b,j}-(-1)^{\p{a}}\partialprime_{-a,j}\partialprime_{b,i}.
\end{align*}
Using these along with \cref{L:fpembedding} it is straightforward to check these elements also are $\tfp (m)$-invariant.

It is clear  $\ttAprime_{i,j}$, $\ttBprime_{i,j}$, and $\ttCprime_{i,j}$ are linearly independent.   Direct calculations using \cref{E:WCcommutator,E:WCcommutator2} verifies that these elements satisfy the commutator formulas given in \cref{L:Commutators}.  Thus the subspace they span is isomorphic to $\fp (n)$ and the claimed map $\theta'$ exists.
\end{proof}

Using the map $\theta'$, $S(U')$ can be viewed as a $\fp (n)$-module where the action commutes with the natural action of $\tfp (m)$ on $S(U')$.  That is, $S(U')$ is a $U(\tfp (m)) \otimes U(\fp (n))$-module.

\subsection{Classical Howe Duality, II}\label{SS:classicalHowedualityII}

\begin{theorem}\label{T:ClassicalHoweDuality2}  The map $\theta'$ is surjective.
\end{theorem}

\begin{proof}  The isomorphism of supercategories $\mathcal{T}  : \fp(m) \text{-mod} \to \tfp(m) \text{-mod}$ given in \cref{SS:Chevalleyisomorphism} defines a superspace isomorphism 
\begin{equation*}
\mathcal{T} :  \Hom_{\fp (m)}\left(\k , S^{\bba }(V_{m})\otimes S^{\bbb}(V_{m})^{*} \right) \xrightarrow{\cong}  \Hom_{\tfp (m)}\left(\k , S^{\bba }(W_{m})^{*}\otimes S^{\bbb}(W_{m}) \right)
\end{equation*}
for all $\bba , \bbb \in \Z_{\geq 0}^{n}$.  Taken together with \cref{isomSUstar} (and its $\tfp (m)$ analogue) this gives a superspace isomorphism:
\begin{multline*}
S(U \oplus U^{*})^{\fp (m)} \cong \bigoplus_{\bba , \bbb \in \Z_{\geq 0}^{m}}  \Hom_{\fp (m)}\left(\k , S^{\bba }(V_{m})\otimes S^{\bbb}(V_{m})^{*} \right) \\
\xrightarrow{\mathcal{\cong}} \bigoplus_{\bba , \bbb \in \Z_{\geq 0}^{n}} \Hom_{\tfp (m)}\left(\k , S^{\bba }(W_{m})^{*}\otimes S^{\bbb}(W_{m}) \right) \cong S(U'^{*} \oplus U')^{\tfp (m)}.
\end{multline*}  Applying this isomorphism to the elements given in \cref{T:SpanningSet} shows the set 
\[
\left\{ \prod_{ 1 \leq i < j \leq n} \left(\sum_{r \in I_{m|m}} (-1)^{\p{r}}\zprime^{*}_{-r,i} \zprime^{*}_{r,j} \right)^{r_{i,j}} \prod_{1 \leq i,j \leq n} \left(\sum_{r \in I_{m|m}} \zprime^{*}_{r,i} \zprime_{r,j} \right)^{s_{i,j}} \prod_{1 \leq i \leq j \leq n} \left(\sum_{r \in I_{m|m}} \zprime_{r,i} \zprime_{-r,j} \right)^{t_{i,j}} \right\}
\] spans $S(U'^{*} \oplus U')^{\tfp (m)}$.

We can now proceed as in the proof of \cref{T:SpanningSet}. That is, if we write $\varsigma': \mathcal{WC}(U') \to S(U'^{*} \oplus U')^{\tfp (m)}$ for the corresponding symbol map, then the image of an appropriately ordered PBW basis for $U(\fp (m))$ under $\varsigma' \circ \theta'$ is related to this spanning set by a unitriangular matrix and, hence, $\theta'$ is surjective.
\end{proof}

\subsection{Intertwining Actions} \label{SS:IntertwiningActionsClassical}  There is also an analogue of \cref{T:IntertwiningActionsCategorical} which relates the $U(\tfp (m)) \otimes U(\fp (n))$-module structures of $S(U)$ and $S(U')$.

Given $m, n \geq 1$ recall that $I_{m,n}=\left\{1, \dotsc , m \right\} \times \left\{1, \dotsc , n \right\}$ is ordered with the lexicographic ordering, $\preceq$.  Given $\epsilon = (\varepsilon_{k,l})_{(k,\ell) \in I_{m,n}}$ with $\varepsilon_{k,\ell} \in \Z_{2}$, let 
\[
\xi(\varepsilon) = \sum_{(a,b) \in I_{m,n}} \left(mn\varepsilon_{a,b}+\sum_{(k, \ell) \preceq (a,b)} \varepsilon_{a,b}\right) =  \sum_{(a,b) \in I_{m,n}} \left(n(a-1)+b + mn  - 1 \right)\varepsilon_{a,b}.
\]
Define a linear map
\begin{equation}\label{E:Phidef}
\Phi_{m,n}: S(U' ) \to S\left( U \right)
\end{equation} by 
\[
\prod_{(k,\ell) \in I_{m,n}} (\zprime)_{k, \ell}^{d_{k,l}}(\zprime)_{-k,\ell}^{\varepsilon_{k,l}}  \mapsto (-1)^{\xi (\varepsilon)} \prod_{(k,\ell) \in I_{m,n}} z_{k,l}^{d_{k,l}}z_{k,-l}^{1-\varepsilon_{k,l}},
\] for all $d_{k,l} \in \Z_{\geq 0}$ and $\varepsilon_{k,l} \in \left\{0,1 \right\}$, where $\varepsilon = (\varepsilon_{k,l})_{(k,l) \in I_{m,n}}$. 
Comparing the parity of the input and output vectors shows the map $\Phi_{m,n}$ has parity $mn$.

\begin{theorem}\label{T:IntertwiningActionsClassical}  For all $m,n \geq 1$, the map
\begin{equation*}
\Phi_{m,n}: S(U' ) \to S\left( U \right)
\end{equation*} is an isomorphism of  $U(\tfp (m)) \otimes U(\fp (n))$-modules.
\end{theorem}

\begin{proof}  The map \(\Phi_{m,n} \) is an isomorphism of superspaces since, up to a sign, it is a bijection on bases. 

For any $(a,b) \in I_{m,n}$ it is straightforward to verify 
\[
\Phi_{m,n}\left( \zprime_{-a,b}. \prod_{(k,\ell) \in I_{m,n}} (\zprime_{k,l})^{d_{k,l}}(\zprime_{-k,-l})^{\varepsilon_{k,l}} \right) =  \pm \partial_{a,-b}.\Phi_{m,n}\left( \prod_{(k,\ell) \in I_{m,n}} (\zprime_{k,l})^{d_{k,l}}(\zprime_{-k,-l})^{1-\varepsilon_{k,l}} \right).
\]  Observe that the sign which appears after calculating the left-hand side is given by
\[
(-1)^{\sum_{(k,l) \prec (a,b)} \varepsilon_{k,\ell} + \xi (\varepsilon')},
\] where $\varepsilon'$ is defined by 
\[
\varepsilon'_{k,\ell} = \begin{cases}  \varepsilon_{a,b}+1, & (k,\ell) = (a,b);\\
                                      \varepsilon_{k,\ell}, & (k, \ell) \neq (a,b).
\end{cases}
\] But 
\[
\sum_{(k,l) \prec (a,b)} \varepsilon_{k,\ell} + \xi (\varepsilon') = \sum_{(k,l) \prec (a,b)} \varepsilon_{k,\ell} + \xi (\varepsilon) + \left(n(a-1)+b + mn  - 1 \right),
\] where $\varepsilon = (\varepsilon_{k,\ell})_{(k, \ell) \in I_{m,n}}$.
On the other hand the sign which appears after calculating the right-hand side is 
\[
(-1)^{\sum_{(k,l) \prec (a,b)} (1-\varepsilon_{k,\ell}) + \xi (\varepsilon)}.
\] But
\[
\sum_{(k,l) \prec (a,b)} (1-\varepsilon_{k,\ell}) + \xi (\varepsilon) = (n(a-1)+b  - 1) - \sum_{(k,l) \prec (a,b)} \varepsilon_{k,\ell}  + \xi (\varepsilon).
\]  Comparing the parities of these shows the difference in signs on each side is $(-1)^{mn}$.  Hence, $\Phi_{m,n}$ intertwines the action of the odd elements $\zprime_{-a,b}$ and $\partial_{a,-b}$.  An identical calculation shows $\Phi_{m,n}$ intertwines the action of the odd elements $\partialprime_{-a,b}$ and $z_{a,-b}$ and the even elements $\partialprime_{a,b}$ and $\partial_{a,b}$, and $\zprime_{a,b}$ and $z_{a,b}$.

With these formulas in hand it is straightforward to verify $\Phi_{m.n}$ intertwines the actions of $\tildettAprime_{i,j}$ and $\tildettA_{i,j}$, $\tildettBprime_{i,j}$ and $\tildettB_{i,j}$, and $\tildettCprime_{i,j}$ and $\tildettC_{i,j}$, and that $\Phi_{m.n}$ intertwines the actions of $\ttAprime_{i,j}$ and $\ttA_{i,j}$, $\ttBprime_{i,j}$ and $\ttB_{i,j}$, and $\ttCprime_{i,j}$ and $\ttC_{i,j}$.  Therefore $\Phi_{m.n}$ is a $U(\tfp (m)) \otimes U(\fp (n))$-module isomorphism.
\end{proof}

By the previous theorem there is an isomorphism of superalgebras
\begin{equation}\label{E:intertwineriso}
\End_{\k }\left(S(U) \right) \to \End_{\k}\left(S(U') \right)
\end{equation}
given by $f \mapsto \Phi_{m,n} \circ f \circ \Phi^{-1}_{m,n}$.  From the proof of the theorem we know $\Phi_{m,n}$ intertwines polynomial differential operators and, hence, the map in \cref{E:intertwineriso} restricts to define a superalgebra isomorphism 
\[
\mathcal{WC}(U) \to \mathcal{WC}(U')
\] which is $U(\tfp (m)) \otimes U(\fp (n))$-linear.  The following theorem follows immediately from this observation along with \cref{T:IntertwiningActionsClassical,T:ClassicalHoweDuality1,T:ClassicalHoweDuality2}.

\begin{theorem}\label{C:ClassicalHoweDualityDoubleCentralizer}  The superalgebras $U(\tfp (m))$ and $U(\fp (n))$ define mutually centralizing subsuperalgebras of $\mathcal{WC}(U) = \mathcal{WC}(Y_{m} \otimes V_{n})$ and $\mathcal{WC}(U')=\mathcal{WC}(V_{m} \otimes Y_{n})$.
\end{theorem}

\begin{remark}  Theorem~A from the introduction is recovered from the previous result by suitably incorporating an application of the Chevalley isomorphism.
\end{remark}

\begin{remark} Note that one could choose to tensor $\Phi_{m,n}$ by an appropriate scaling of the supertrace representation as we did in \cref{SS:IntertwiningActionsCategorical} to eliminate the scalar multiplies of the identity which appear in $\tildettA_{i,i}$ and $\ttAprime_{i,i}$  in \cref{E:EijDifferential,E:EprimeijDifferential}.  Doing so has no effect on the Lie superalgebras these elements generate nor the classical Howe duality theorems. 
\end{remark}

\bibliographystyle{eprintamsplain}
\bibliography{Biblio.bib}

\end{document}